\documentclass[10pt]{article}
\usepackage{latexsym,amsfonts,amsmath,amssymb,array,tabularx,theorem}
\usepackage{graphicx}
\usepackage{caption}
\usepackage{subcaption}
\usepackage[normalem]{ulem}
\usepackage[pdfpagelabels]{hyperref}
\usepackage{bbm}
\usepackage{authblk}
\usepackage{bm}
\usepackage{mathtools} %\bigtimes
\usepackage{dsfont} %\mathbbm{1}
\usepackage{algpseudocode}
\usepackage[section]{algorithm}
\usepackage[normalem]{ulem}

\usepackage{graphicx}
\usepackage{epstopdf}
\usepackage[all]{xy}
\DeclareFontFamily{OT1}{pzc}{}
\DeclareFontShape{OT1}{pzc}{m}{it}{<-> s * [1.10] pzcmi7t}{}
\DeclareMathAlphabet{\mathpzc}{OT1}{pzc}{m}{it}

\usepackage{float}
\newtheorem{theorem}{Theorem}[section]
\newtheorem{lemma}[theorem]{Lemma}
\newtheorem{corollary}[theorem]{Corollary}
\newtheorem{proposition}[theorem]{Proposition}
\newtheorem{remark}[theorem]{Remark}
\newtheorem{definition}[theorem]{Definition}

\newenvironment{proof}{\begin{trivlist}
    \item[\hskip\labelsep{\bf Proof.}]}{$\hfill\Box$\end{trivlist}}
\newenvironment{proofof}[1]{\begin{trivlist}
    \item[\hskip\labelsep{\bf Proof of {#1}.}]}{$\hfill\Box$\end{trivlist}}
{\theoremstyle{plain} \theorembodyfont{\rmfamily}
}

\numberwithin{equation}{section}
\numberwithin{figure}{section}

\allowdisplaybreaks[1]

\newcommand{\bsv}{{\boldsymbol{v}}}

\newcommand{\bsnu}{{\boldsymbol{\nu}}}

\newcommand{\bsvarrho}{{\boldsymbol{\varrho}}}

\newcommand{\bbR}{\mathbb{R}}

\newcommand{\R}{\mathbb{R}}

\newcommand{\N}{\mathbb{N}}

\newcommand{\bbP}{\mathbb{P}}

\newcommand{\calF}{\mathcal{F}}

\newcommand{\calH}{\mathcal{H}}
\newcommand{\calI}{\mathcal{I}}

\newcommand{\calNN}{\mathcal{NN}}

\newcommand{\calT}{\mathcal{T}}

\newcommand{\calE}{\mathcal{E}}

\newcommand{\calV}{\mathcal{V}}

\newcommand{\mask}[1]{{}}

\newcommand{\norm}[2][]{\| #2 \|_{#1}}

\newcommand{\snorm}[2][]{| #2 |_{#1}}

\newcommand{\setc}[2]{\left\{#1\, :\,#2\right\}}
\newcommand{\set}[2]{\{#1\,:\,#2\}}
\newcommand{\simp}[2]{\conv (\{#1, \dots, #2\})}
\newcommand{\domain}{\Omega}
\newcommand{\Sop}[1]{\operatorname{S}^{\rm 1}_{#1}}
\newcommand{\RTp}[1]{\operatorname{RT}_{#1}}
\newcommand{\Szp}[1]{\operatorname{S}^{\rm 0}_{#1}}
\newcommand{\So}{\Sop{1}}
\newcommand{\RT}{\RTp{0}}
\newcommand{\Sz}{\Szp{0}}
\newcommand{\Ne}{\operatorname{N}_{0}}
\newcommand{\CR}{\operatorname{CR}_{0}}
\DeclareMathOperator{\divv}{div}
\DeclareMathOperator{\gradd}{grad}
\DeclareMathOperator{\curl}{curl}
\newcommand{\Ned}{N\'ed\'elec }
\newcommand{\be}{\begin{equation}}
\newcommand{\ee}{\end{equation}}
\newcommand{\bea}{\begin{eqnarray}}
\newcommand{\eea}{\end{eqnarray}}
\newcommand{\beas}{\begin{eqnarray*}}
\newcommand{\eeas}{\end{eqnarray*}}

\DeclareMathOperator*{\argmin}{argmin}
\DeclareMathOperator*{\Span}{span}
\newcommand{\relu}{{\rho}}
\newcommand{\heavi}{{\sigma}}
\newcommand{\heavii}{{\tilde\sigma}}
\newcommand{\realiz}[1]{{\rm R}(#1)} %normal brackets
\newcommand{\realizc}[1]{{\rm R}\left(#1\right)} %large brackets
\newcommand{\depth}{L}
\newcommand{\size}{M}
\newcommand{\Parallel}[1]{{\rm P}(#1)} 

\newcommand{\Parallelc}[1]{{\rm P}\left(#1\right)} 

\newcommand{\sconc}{\odot}

\newcommand{\sizefirst}{\size_{\operatorname{in}}}
\newcommand{\sizelast}{\size_{\operatorname{out}}}

\newcommand{\inradius}{r}
\DeclareMathOperator{\conv}{conv}
\DeclareMathOperator{\interior}{int}

\newcommand{\Id}{{\rm Id}}

\begin{document}

\bibliographystyle{abbrv}
\title{De Rham compatible Deep Neural Network FEM} \author[$\dagger$]{
  \mbox{ } Marcello Longo} \author[$\dagger$]{Joost A. A. Opschoor}
\author[$\ddagger$]{Nico Disch} \author[$\dagger$]{\mbox{ } \newline
  Christoph Schwab} \author[$\ddagger$]{Jakob Zech}

\affil[$\dagger$]{\footnotesize Seminar for Applied Mathematics, ETH
  Z\"{u}rich, R\"{a}mistrasse 101, CH--8092 Z\"urich, Switzerland.
  \newline \texttt{marcello.longo@sam.math.ethz.ch,\;
    joost.opschoor@sam.math.ethz.ch,\;
    christoph.schwab@sam.math.ethz.ch}}
\affil[$\ddagger$]{\footnotesize IWR, Universit\"at Heidelberg, Im
  Neuenheimer Feld 205, 69120 Heidelberg, Germany.  \newline
  \texttt{n.disch@stud.uni-heidelberg.de,\;
    jakob.zech@uni-heidelberg.de}}
\maketitle
\date{}
\begin{abstract}
  On general regular simplicial partitions $\calT$ of bounded polytopal
  domains $\domain \subset \bbR^d$, $d\in\{2,3\}$,
  we construct \emph{exact neural network (NN) emulations} of all
  lowest order finite element spaces in the discrete de Rham complex.
  These include the spaces of piecewise constant functions, continuous
  piecewise linear (CPwL) functions, the classical ``Raviart-Thomas
  element'', and the ``N\'{e}d\'{e}lec edge element''. For all but the
  CPwL case, our network architectures employ both ReLU (rectified
  linear unit) and BiSU (binary step unit) activations to capture
  discontinuities.  In the important case of CPwL functions, 
  we prove that it suffices to work with pure ReLU nets. 
  Our construction and DNN architecture
  generalizes previous results in that no geometric restrictions 
  on the regular simplicial partitions $\calT$ of $\domain$ are required 
  for DNN emulation. 
  In addition,
  for CPwL functions our DNN construction is valid in any dimension $d\ge 2$.
  Our ``FE-Nets'' are required in the variationally correct,
  structure-preserving approximation of boundary value problems of
  electromagnetism in nonconvex polyhedra $\domain \subset \R^3$.
  They are thus an essential ingredient in the application of e.g.,
  the methodology of ``physics-informed NNs'' or ``deep Ritz methods''
  to electromagnetic field simulation via deep learning techniques.
  We indicate
  generalizations of our constructions to higher-order compatible
  spaces and other, non-compatible classes of discretizations, in
  particular the ``Crouzeix-Raviart'' elements and Hybridized, Higher
  Order (HHO) methods.
\end{abstract}



\tableofcontents

%%%%%%%%%%%%%%%%%%%%%%%%%%%%%%%%%%%%%%%%%%%%%%%%%%%%%%%%%%%%%%%%%
\section{Introduction}
\label{sec:Intro}
%%%%%%%%%%%%%%%%%%%%%%%%%%%%%%%%%%%%%%%%%%%%%%%%%%%%%%%%%%%%%%%%%%

Recent years have seen the emergence of Deep Neural Network (DNN)
based methods for the numerical approximation of solutions to partial
differential equations (PDEs for short). In one class of proposed
methods, DNNs serve as \emph{approximation architectures} in a
suitable, weak form of the PDE of interest. In \cite{EYuDeepRitz},
for elliptic, self-adjoint PDEs the variational principle associated
to the PDE is computationally minimized over suitable DNNs, so that
the energy functional of the physical system of interest gives rise to
a consistent loss function for the training of the DNN. Numerical
solutions obtained from training the approximating DNN in this way
correspond to approximate variational solutions of the PDE under
consideration.

The recently promoted ``physics-informed NNs'' (PiNNs),
e.g.\ \cite{RPK_2019,Yang_2021} and references there, insert DNN
approximations with suitably smooth activations (e.g.\ softmax or
tanh) as approximation architecture into the strong form of the
governing PDE.  Approximate solutions are obtained by numerical
minimization of loss functions obtained by discretely enforcing
smallness of the residual at collocation points in the spatio-temporal
domain.  While empirically successful in a large number of test cases,
also DNN based approximations are subject to the fundamental paradigm
that ``stability and consistency implies convergence''.  A key factor
of recent successful DNN deployment in numerical PDE solution is their
excellent approximation properties, in particular on high-dimensional
state- and parameter-spaces, e.g.\ \cite{PV2018,MR3564936,SZ19_2592}
and the references there.  High smoothness of DNNs with smooth
activations may, however, preclude convergence of so-called ``deep
Ritz'' approaches where loss functions in DNN training are derived
from energies in variational principles \cite{EYuDeepRitz}, even for
linear, deterministic and well posed PDEs.

To leverage the methodology of PiNNs and e.g., the variational Ritz
method for computational electromagnetics, computational
magneto-hydrodynamics etc., \emph{structure-preserving DNNs} must be
adopted.  We provide here, therefore, 
\emph{de Rham complex compatible DNN emulations} of the standard, lowest
order finite element spaces on regular, simplicial triangulations of 
polytopal domains $\domain \subset \bbR^d$, $d=2,3$.
These spaces satisfy exact (de Rham) sequence
properties, and also spawn discrete boundary complexes on
$\partial\domain$ that satisfy exact sequence properties for the
surface divergence and curl operators $\divv_{{\partial \Omega}}$ and
$\curl_{{\partial \Omega}}$.  These in turn enable ``neural boundary
elements'' for computational electromagnetism, recently proposed in
\cite{AHS22_989}.
%%%%%%%%%%%%%%%%%%%%%%%%%%%%%%%%%%%%%%%%%%%%%%%%%%%%%%%%%%%%%%%%%
\subsection{Previous work}
\label{sec:PreWrk}
%%%%%%%%%%%%%%%%%%%%%%%%%%%%%%%%%%%%%%%%%%%%%%%%%%%%%%%%%%%%%%%%%
The connection between DNNs with Rectified Linear Unit 
(ReLU for short) activation and continuous, piecewise linear (CPwL)
spline approximation spaces has been known for some time:
\emph{nodal discretizations} based on CPwL finite element methods (FEM) 
can be emulated by ReLU NNs (e.g.\ as introduced in \cite{ABMM2016} and \cite{HLXZ2020}).

When CPwL finite elements are applied to, for example, weak
formulations of the time-harmonic Maxwell equations, 
they are known to converge to
the correct solution, generally, only for convex 
polygons or polyhedra: 
if $\domain$ has re-entrant corners or edges, then with\footnote{
Definitions of the (standard) spaces 
$H^1(\domain)$, $H^0(\divv,\domain)$ and $H^0(\curl,\domain)$ 
are recalled in Section \ref{sec:derham}.}
\begin{equation*}
  X_N(\domain) := H^0(\divv,\domain) \cap H^0(\curl,\domain)\cap\{ u \colon u \times n = 0  \text{ on }\ \partial\domain\},
\end{equation*}
where $ n $ is a unit normal vector to the boundary
$ \partial \domain $ of $ \domain $,
the vector fields $[H^1(\domain)]^3 \cap X_N(\domain)$ are closed in
$X_N(\domain)$ \emph{without being dense}, see, e.g.,
\cite{CostabelCoerCMaxw,CoDaNicMaxw}.  For such nonconvex polyhedra,
the weak solution to the time-harmonic Maxwell's equations 
is generally not contained in $[H^1(\domain)]^3$.

Since any discrete conforming space based on a standard nodal finite
element method is contained in $[H^1(\domain)]^3$, nodal FEM in this
situation converges to a wrong solution (in $[H^1(\domain)]^3$) as the
meshwidth tends to zero (respectively as the width of the
corresponding NN tends to infinity) \cite{CoDaMaxEVP}.  Similar issues
will arise for PiNN numerical approximations of low-regularity
solutions for $H^0(\curl,\domain)$-based PDEs such as the
time-harmonic Maxwell equations.  They will persist also for DNN
surrogates with more regular activation functions such as $ReLU^k$
for $k\in\N$ and sigmoidal or softmax activations.  
On bounded
sets, such NNs realize Lipschitz continuous functions, which are in
$[H^1(\domain)]^3$ and therefore may converge to an incorrect solution.

A second broad class of variational models, where continuous nodal FEM
  may
  cause problems, are ``deep Ritz'' type approaches such as in \cite{EYuDeepRitz}, which attempt to minimize energy functionals.
  For certain nonlinear problems the so-called ``Lavrentiev gap''
  incurred by CPwL approximation architectures is known to be a
  fundamental obstruction to obtain convergent families of discrete
  minimizers, see e.g., \cite{MR969900,MR718413,MR1836612}. 
Again, relaxing
continuity below $H^1$-conformity is known to remedy this issue; see,
e.g.\ \cite{balci2021crouzeixraviart} and the discussion and
references there. 
Accordingly, in Section~\ref{sec:CRElement} of the present paper we present CR-Net,
a DNN emulation of the Crouzeix-Raviart element with BiSU and ReLU activations, 
on general regular, simplicial partitions 
of polytopal domains $\domain \subset \bbR^d$, $d\geq 2$,
which, when used in a deep Ritz method style
approach for variational problems, affords convergent sequences of DNN
approximations of minimizers.  CR-Net will also afford advantages in
variational image segmentation (e.g.\ \cite{ChambPock} and the
references there).

Structure preservation in scientific machine learning
is also the topic of \cite{THH2022}.
For machine learning models on graphs,
a data driven exterior calculus is introduced
which strongly enforces physical laws,
e.g. those in the de Rham complex,
while allowing for additional information
to be learned from data.

%%%%%%%%%%%%%%%%%%%%%%%%%%%%%%%%%%%%%%%%%%%%%%%%%%%%%%%%%%%%%%%%%
\subsection{Contributions}
\label{sec:Contr}
%%%%%%%%%%%%%%%%%%%%%%%%%%%%%%%%%%%%%%%%%%%%%%%%%%%%%%%%%%%%%%%%%
The purpose of the present paper is the design of DNNs which emulate
\emph{exactly}, on arbitrary regular, simplicial partitions
$\calT$ of polytopal domains $\domain\subset \R^d$, $d=2,3$, 
the FE spaces $\So(\mathcal{T},\domain)$ 
(continuous, piecewise linear functions), 
$\Ne(\mathcal{T},\domain)$ (the \Ned element),
  $\RT(\mathcal{T},\domain)$ (the Raviart-Thomas element) and
  $ \Sz(\mathcal{T},\domain)$ (the piecewise constant functions).  
  The precise definitions of these spaces will be given in Section
  \ref{sec:fem}.

We provide constructions of DNNs based on a combination of ReLU
\eqref{eq:reludef} 
and BiSU (Binary Step Unit)
\eqref{eq:heavidef} 
activations, which emulate these classical, lowest-order FE spaces in
the de Rham complex on a regular, simplicial partition $\calT$ of
$ \domain $.
We underline that our construction of NNs which emulate, in
particular, the classical ``Courant Finite Elements''
$\So(\mathcal{T},\domain)$, as well as $\Sz(\mathcal{T},\domain)$ and
$\RT(\mathcal{T},\domain)$, applies to polytopal domains $\domain$ of
any dimension $d\geq 2$.  For the practically relevant space
$\So(\mathcal{T},\domain)$, the so-called ``continuous, piecewise linear (CPwL) functions'',
we provide DNN constructions based on ReLU activation only, 
which work in arbitrary, finite dimension $d\geq 2$
(the univariate case $d=1$ being trivial).

Our constructions accommodate general, regular simplicial partitions
$\calT$ of $\domain$.  In particular, apart from regularity of the
simplicial partition $\calT$ of the polytopal domain $\domain$, no further
constraints of geometric nature are imposed on $\calT$, in arbitrary
dimension $d\geq 2$.  Our results on ReLU NN emulation of CPwL
functions in Section~\ref{sec:cpwlrelu} therefore unify and
quantitatively improve earlier ones such as, e.g., 
\cite[Section 3]{HLXZ2020}, 
which covered only CPwL FE spaces on particular triangulations of $\domain$.  
Our main results, Propositions
\ref{prop:basisnet} and \ref{prop:CPLbasisnet} and
Theorem~\ref{thm:derhamapx} in Section~\ref{sec:approximation},
provide mathematically exact DNN realizations of the lowest order FE
spaces in the exact sequence \eqref{discderham} on general regular,
simplicial partitions of the contractible, polytopal domain $\domain$.
In our main results using ReLU and BiSU activations, the network size
scales linearly with the cardinality $ \snorm \calT $ of $\calT$.  
For the ReLU NN emulation of CPwL functions, the network size is in
general of the order $\snorm\calT \log(\snorm\calT)$, which
can be improved to order $|\calT|$ for shape regular meshes.

%%%%%%%%%%%%%%%%%%%%%%%%%%%%%%%%%%%%%%%%%%%%%%%%%%%%%%%%%%%%%%%%%
\subsection{Layout}
\label{sec:outline}
%%%%%%%%%%%%%%%%%%%%%%%%%%%%%%%%%%%%%%%%%%%%%%%%%%%%%%%%%%%%%%%%%
The structure of this paper is as follows. In Section
  \ref{sec:fem} we introduce the de Rham complex. 
In
Section~\ref{sec:nndef}, we recapitulate notation and basic
definitions for the NNs which we consider.  
We also review 
a basic NN calculus that shall be used subsequently in order to derive
several properties of the proposed NN architectures.

Sections~\ref{sec:nnfem} and \ref{sec:cpwlrelu} contain the core
material of the paper: in Section \ref{sec:nnfem}, using ReLU and BiSU
activations, we provide explicit emulations for bases of all the FE
spaces considered in this paper, without geometric conditions on the
regular triangulations $\calT$ of $\domain$. In Section
\ref{sec:cpwlrelu} we show that for the special case of the emulation
of CPwL functions,
the same can be achieved employing solely ReLU activations.  In both
sections, we show that the network size depends only moderately on the
space dimension $d$ and the shape regularity of the partition.

In Section~\ref{sec:approximation} we combine NN emulations of basis
functions from Sections \ref{sec:nnfem} and \ref{sec:cpwlrelu} to
provide emulations of the FE spaces and discuss the implications of
our results for function approximation by NNs in the respective
Sobolev spaces. Section~\ref{sec:traces} provides a construction of NN
emulations for compatible spaces on the boundary
$\Gamma=\partial\domain$ of the polytopal domains. These spaces are
required in the deep neural network approximation of boundary integral
equations in electromagnetics, among others, as discussed in
\cite{BHvS03_22,BDKSVW2020} and the references there.  Finally, in
Section \ref{sec:ExtConcl} we present conclusions and explain how our
analysis may be extended to higher order polynomial spaces and to
certain Finite Element families which are non-compatible with
\eqref{derham}.

%%%%%%%%%%%%%%%%%%%%%%%%%%%%%%%%%%%%%%%%%%%%%%%%%%%%%%%%%%%%%%%%%
\subsection{Notation and Finite Element spaces}
\label{sec:fem}
%%%%%%%%%%%%%%%%%%%%%%%%%%%%%%%%%%%%%%%%%%%%%%%%%%%%%%%%%%%%%%%%%
We recall definitions of the de
Rham complex and corresponding lowest order FE spaces.
%%%%%%%%%%%%%%%%%%%%%%%%%%%%%%%%%%%%%%%%%%%%%%%%%%%%%%%%%%%%%%%%%
\subsubsection{Meshes}
\label{sec:meshes}
%%%%%%%%%%%%%%%%%%%%%%%%%%%%%%%%%%%%%%%%%%%%%%%%%%%%%%%%%%%%%%%%%
The term ``mesh'' shall denote certain simplicial partitions of
polyhedral domains $ \domain \subseteq \R^d $ for some
$ d \in \N = \{1,2,\ldots\} $.  To specify these, for
$k\in\{0,\ldots,d\}$ we define a $ k $-simplex $ T $ by
$ T = \conv(\{a_0,\ldots, a_k\}) \subset \R^d $, for some
$ a_0,\ldots,a_k \in \R^d $ which do not all lie in one affine
subspace of dimension $k-1$, and where
\begin{equation*}
  \conv (Y) := \setc{x = \sum_{y \in Y}\lambda_y y}{ \lambda_y>0 \, \text{ and }\, \sum_{y \in Y} \lambda_y = 1 }
\end{equation*}
denotes the open convex hull.  By $\snorm{T}$ we will denote the
$k$-dimensional Lebesgue measure of a $k$-simplex.  We consider a
simplicial mesh $ \calT $ on $ \domain $ of $d$-simplices,
i.e. $\calT$ satisfies that
$ \overline{\domain} = \bigcup_{T\in \calT} \overline{T} $ and
$ T\cap T' = \emptyset$, for all $ T\neq T' $.
We assume that $ \calT $ is a \emph{regular} partition, i.e. for all
distinct $T, T' \in\calT$ it holds that
$\overline{T}\cap\overline{T'}$ is the closure of a $k$-subsimplex of
$T$ for some $k\in\{0,\ldots,d-1\}$, i.e. there exist
$a_0,\ldots,a_d \in \domain$ such that $T = \conv(\{a_0,\ldots,a_d\})$
and
$\overline{T}\cap\overline{T'} =
\overline{\conv(\{a_0,\ldots,a_k\})}$.
The \emph{shape-regularity constant} $C_{\rm sh}:=C_{\rm sh}(\calT)$
of a simplicial partition $\calT$ of $\domain$ is
$ C_{\rm sh} :=\max_{T\in \calT}\tfrac{h_T }{\inradius_T} > 0 $.  Here
$h_T := \operatorname{diam}(T)$ and $\inradius_T$ is the radius of the
largest ball contained in $\overline{T}$.
Let $ \calV $ be the set of vertices of $ \calT $.  We also let
$ \calF,\calE $ be the sets of $ (d-1) $- and $ 1 $-subsimplices of
$ \calT $, whose elements are called \emph{faces} and \emph{edges},
respectively, that is
\begin{align*}
  \calF &:= \set{f \subset \overline{\domain} }{ \exists T = \conv(\{a_0,\ldots, a_d\})\in \calT, \exists i\in\{0,\ldots,d\}
          \text{ with } f = \conv(\{a_0,\ldots, a_{d}\}\backslash \{a_i\} )}, \\
  \calE &:=\set{e \subset \overline{\domain}}{\exists T = \conv(\{a_0,\ldots, a_d\}) \in \calT, \exists i,j\in\{0,\ldots,d\}, i\neq j,
          \text{ with } e = \conv(\{a_i,a_j \})}.
\end{align*}
We denote the boundary of $ \domain $ by $ \partial \domain $ and the
\emph{skeleton} of $ \calT $ by
$ \partial \calT := \bigcup_{T\in \calT} \partial T $.

%%%%%%%%%%%%%%%%%%%%%%%%%%%%%%%%%%%%%%%%%%%%%%%%%%%%%%%%%%%%%%%%%
\subsubsection{De Rham complex}
\label{sec:derham}
%%%%%%%%%%%%%%%%%%%%%%%%%%%%%%%%%%%%%%%%%%%%%%%%%%%%%%%%%%%%%%%%%
We write $ H^1(\domain)$, $ H^0(\divv,\domain)$,
$ H^0(\curl,\domain) $ to indicate the following Sobolev spaces
\begin{align*}
  H^1(\domain) &:= \{v \in L^2(\domain) \colon \nabla v \in [L^2(\domain)]^d \}, \\
  H^0(\curl,\domain)& := \{v \in [L^2(\domain)]^d \colon \curl v \in [L^2(\domain)]^d \} \qquad \text{ for } d = 2,3,\\
  H^0(\divv,\domain) & := \{v \in [L^2(\domain)]^d \colon \divv v \in L^2(\domain) \},
\end{align*}
where the two-dimensional $\curl$ is defined as
$\curl v = (-\partial_1 v_2, \partial_2 v_1)$.  These are Hilbert
spaces.  Specifically, let us assume that $\domain\subset \R^3$ is a
contractible Lipschitz\footnote{{That is with boundary parametrized
    locally by Lipschitz continuous maps \cite[Definition
    3.2]{ErnGuermondBookI2021}.}}  domain with connected boundary
$ \partial \domain $.
Then, it is well-known that the following de
Rham complex is an exact sequence 
(e.g.\ \cite[Proposition 16.14]{ErnGuermondBookI2021} and the references there):
\begin{equation}\label{derham}
  \xymatrix{
    \R\ar^-{i}[r] &
    H^1(\domain)\ar^-{\operatorname{grad}}[r] &
    H^0(\curl,\domain) \ar^-{\curl}[r] &
    H^0(\divv,\domain)  \ar^-\divv[r] &
    L^2(\domain) \ar^-o[r] & \{0\}.
  }
\end{equation}
Here, 
$i$ denotes an injection and 
$o$ denotes the
zero operator.

%%%%%%%%%%%%%%%%%%%%%%%%%%%%%%%%%%%%%%%%%%%%%%%%%%%%%%%%%%%%%%%%%
\subsubsection{First order discrete de Rham complex}
\label{sec:discderham}
%%%%%%%%%%%%%%%%%%%%%%%%%%%%%%%%%%%%%%%%%%%%%%%%%%%%%%%%%%%%%%%%%
Finite dimensional subspaces preserving this structure are usually
required to fit into a \emph{discrete de Rham complex} (e.g.\
\cite[Proposition 16.15]{ErnGuermondBookI2021})
\begin{equation}\label{discderham}
  \xymatrix{
    \R\ar^-{i}[r] &
    \So(\mathcal{T},\domain) \ar^-{\operatorname{grad}}[r] &
    \Ne(\mathcal{T},\domain) \ar^-{\curl}[r] &
    \RT(\mathcal{T},\domain) \ar^-\divv[r] &
    \Sz(\mathcal{T},\domain) \ar^-o[r] &
    \{0\}.
  }
\end{equation}
We next define these spaces.

Throughout, denote by $ \mathbb{P}_k $ the set of polynomials of
degree at most $ k \in \N \cup\{0\} $.  For a given regular
triangulation $\mathcal{T}$ of a domain $\domain$,
$ \So(\calT, \domain) $ 
is the class of continuous, piecewise linear (CPwL) functions on
$ \calT $, i.e.\
\begin{equation}\label{def:S1}
  \So(\calT, \domain) := \set{v \in H^1(\domain)}{v|_T \in \bbP_1, \ \forall T\in \calT } \subset  H^1(\domain).
\end{equation}
Moreover, $ \Sz(\calT,\domain) $ is the class of piecewise constant functions
on the partition $\calT$
\begin{equation}\label{def:S0}
  \Sz(\calT,\domain)
  :=
  \set{v \in L^2(\domain)}{v|_T \equiv c_T \in \R, \ \forall T\in \calT } \subset  L^2(\domain).
\end{equation}

Next, we recall $\RT(\mathcal{T},\domain)$,
the lowest order Raviart-Thomas space.
Define the vector-valued polynomial space
$ \mathbb{RT}_0 = (\mathbb{P}_0)^{d} \oplus x\mathbb{P}_0 $ and, for
all $f\in\calF$, let $ n_f $ denote a unit normal to the face $ f $.
Denote by $ [v\cdot n_f]_f $ the jump of the normal component of a
vector field across $ f $, that is
$$ [v\cdot n_f]_f(x_0) = \lim_{\epsilon\searrow 0} (v(x_0+\epsilon
n_f)- v(x_0-\epsilon n_f))\cdot n_f\qquad \text{for all}\quad x_0 \in f.$$  The
Raviart-Thomas finite element space of lowest order is
(e.g. \cite[Section 14.1]{ErnGuermondBookI2021})
\begin{align} \label{def:RT0} \RT(\calT,\domain) := \set{v\in
    (L^1(\domain))^d}{v|_T \in \mathbb{RT}_0 \ \forall T \in \calT
    \text{ and } [v\cdot n_f]_f = 0 \ \forall f\in \calF , f \subset
    \domain } .
\end{align}
This space has one degree of freedom per face $ f \in \calF $ and it
satisfies $ \RT(\calT,\domain)\subset H^0(\divv,\domain) $.

To define the lowest order \Ned space $\Ne(\mathcal{T},\domain)$, 
for $d=2$ define
$ \mathbb{NE}_0 = (\mathbb{P}_0)^{d} \oplus \mathbb{P}_0 ( -x_2, x_1
)$, and for $d=3$ define
$ \mathbb{NE}_0 = (\mathbb{P}_0)^{d} \oplus x\times (\mathbb{P}_0)^{d}
$.
Let $ [v\times n_f]_f $ denote the jump of the tangential component of
a vector field across $ f $, that is
$$ [v\times n_f]_f(x_0) = \lim_{\epsilon\searrow 0} (v(x_0+\epsilon
n_f)- v(x_0-\epsilon n_f))\times n_f  \qquad\text{for all}\quad x_0 \in f.$$  Then
the \Ned finite element space of lowest order \cite[Section
15.1]{ErnGuermondBookI2021} reads
\begin{equation} \label{def:N0} \Ne(\calT,\domain) := \set{v\in
    (L^1(\domain))^{d}}{v|_T \in \mathbb{NE}_0 \ \forall T \in \calT
    \text{ and } [v\times n_f]_f = 0 \ \forall f\in \calF , f\subset
    \domain} .
\end{equation}
This space has one degree of freedom per edge $ e \in \calE $ and it
satisfies $ \Ne(\calT,\domain) \subset H^0(\curl,\domain) $.  For
$d=2$, $\Ne(\calT,\domain)$ is closely related to
$\RT(\calT,\domain)$ (see \eqref{eq:NeRTtwodim}).
For each of the spaces in \eqref{discderham}, we state in Section
\ref{sec:nnfem} a basis of the space.
The FE spaces from \eqref{discderham} have the advantage of being
\emph{conforming}, i.e., they are finite dimensional spaces, each
strictly contained in the respective Sobolev space in \eqref{derham}.
Furthermore, the ($\calT$-dependent) projections
$\Pi_{\So},\Pi_{\Ne},\Pi_{\RT},\Pi_{\Sz}$ on these subspaces
introduced in \cite[Sec.\ 19.3]{ErnGuermondBookI2021} commute with
  the differential operators as shown in the following diagram
  \cite[Lemma 19.6]{ErnGuermondBookI2021}:
$$
\xymatrix{
  H^1(\domain)\ar^-{\operatorname{grad}}[r] \ar^{\Pi_{\So}}[d] &
  H^0(\curl,\domain) \ar^-{\curl}[r] \ar^{\Pi_{\Ne}}[d] &
  H^0(\divv,\domain) \ar^-\divv[r] \ar^{\Pi_{\RT}}[d] & L^2(\domain)
  \ar^{\Pi_{\Sz}}[d]
  \\
  \So(\mathcal{T},\domain) \ar^-{\operatorname{grad}}[r] &
  \Ne(\mathcal{T},\domain) \ar^-{\curl}[r] & \RT(\mathcal{T},\domain)
  \ar^-\divv[r] & \Sz(\mathcal{T},\domain)
}
$$
For these reasons we say that these spaces are \emph{de Rham
  compatible}.

These spaces also appear in the Helmholtz decomposition
of vector fields in bounded, contractible polyhedral domains
$\domain \subset \R^3$.
For every vector field $v \in [L^2(\domain)]^3$
there exist $\varphi \in H^1(\domain)$
and $\psi \in H^0(\curl,\domain) \cap H^0(\divv,\domain)$
such that $v = \gradd \varphi + \curl \psi$,
see e.g. \cite[Section 3.5]{ABDG1998}.

%%%%%%%%%%%%%%%%%%%%%%%%%%%%%%%%%%%%%%%%%%%%%%%%%%%%%%%%%%%%%%%%%%
\section{Neural networks}
\label{sec:nndef}
%%%%%%%%%%%%%%%%%%%%%%%%%%%%%%%%%%%%%%%%%%%%%%%%%%%%%%%%%%%%%%%%%%
To accommodate for both continuous components and discontinuous
components in the functions we want to emulate, we consider neural
networks where several different activation functions are used
throughout the network.  We define neural networks as a collection of
parameters and for each position in the network (also called neuron or
unit) we specify the activation function used.  Associated to such a
neural network is a function, called realization, which is the
iterated composition of affine transformations defined in terms of the
parameters and non-linear activation functions.

%%%%%%%%%%%%%%%%%%%%%%%%%%%%%%%%%%%%%%%%%%%%%%%%%%%%%%%%%%%%%%%%%
\subsection{Feedforward NNs}
\label{sec:ffnn}
%%%%%%%%%%%%%%%%%%%%%%%%%%%%%%%%%%%%%%%%%%%%%%%%%%%%%%%%%%%%%%%%%
  For $d,L\in\N$, a \emph{neural network $\Phi$} 
  with input dimension $d \geq 1$ and number of layers $L\geq 1$, 
  comprises a finite collection of activation 
  functions $\bsvarrho = \{\varrho_{\ell}\}_{\ell=1}^L$ 
  and a finite sequence of matrix-vector tuples, i.e.
  \begin{align*}
    \Phi = ((A_1,b_1,\varrho_1),(A_2,b_2,\varrho_2),\ldots,(A_L,b_L,\varrho_L)).
  \end{align*}

  For $N_0 := d$ and \emph{numbers of neurons $N_1,\ldots,N_L\in\N$ per layer}, 
  for all $\ell=1,\ldots, L$ it holds that
  $A_\ell\in\bbR^{N_\ell \times N_{\ell-1} }$ and
  $b_\ell\in\bbR^{N_\ell}$, and that $\varrho_\ell$ is a list of
  length $N_\ell$ of \emph{activation functions}
  $(\varrho_\ell)_i: \R\to\R$, $i=1,\ldots,N_\ell$, acting on node $i$
  in layer $\ell$.

  The \emph{realization} of $\Phi:\R^{N_0} \to \R^{N_L}$ as a map is
  the function
  \begin{align*}
    \realiz{\Phi}: \R^d\to\R^{N_L} : x \to x_L,
  \end{align*}
  where
  \begin{align*}
    x_0 & := x,
    \\
    x_\ell & := \varrho_\ell( A_\ell x_{\ell-1} + b_\ell ),
             \qquad\text{ for }\ell=1,\ldots,L-1,
    \\
    x_L & := A_L x_{L-1} + b_L.
  \end{align*}
  Here, for $\ell=1,\ldots,L-1$, the list of activation functions
  $\varrho_\ell$ of length $N_\ell$ is effected componentwise: for
  $y = (y_1,\ldots,y_{N_\ell})\in\bbR^{N_\ell}$ we denote
  $\varrho_\ell(y) = ( (\varrho_\ell)_1(y_1), \ldots,
  (\varrho_\ell)_{N_\ell}(y_{N_\ell}) )$.
  I.e., $(\varrho_\ell)_i$ is the activation function applied in
  position $i$ of layer $\ell$.

  We call the layers indexed by $\ell=1,\ldots,L-1$ \emph{hidden layers}, 
  in those layers activation functions are applied.  No
  activation is applied in the last layer of the NN.  For consistency
  of notation, we define $\varrho_L := \Id_{\R^{N_L}}$.

  We refer to $\depth(\Phi) := L$ as the \emph{depth} of $\Phi$.  
  For $\ell=1,\ldots,L$ we denote by
  $\size_\ell(\Phi) := \norm[0]{A_\ell} + \norm[0]{b_\ell} $ the
  \emph{size of layer $\ell$}, which is the number of nonzero
  components in the weight matrix $A_\ell$ and the bias vector
  $b_\ell$, and call $\size(\Phi) := \sum_{\ell=1}^L \size_\ell(\Phi)$
  the \emph{size} of $\Phi$.  Furthermore, we call $d$ and $N_L$ the
  \emph{input dimension} and the \emph{output dimension}, and denote
  by $\sizefirst(\Phi) := \size_1(\Phi)$ and
  $\sizelast(\Phi) := \size_L(\Phi)$ the size of the first and the
  last layer, respectively.

Our networks will use two different activation functions.  Firstly, we
use the \emph{Rectified Linear Unit} (\emph{ReLU}) activation
\begin{equation}
  \label{eq:reludef}
  \relu(x) = \max\{ 0, x \}.
\end{equation}
We will often use the elementary identities
$ \relu(x) + \relu(-x) = |x| $, $ \relu(x) - \relu(-x) = x $ to
construct composite functions.
Networks which only
contain ReLU activations realize continuous, piecewise linear
functions. By \emph{ReLU NNs} we refer to NNs which only have ReLU
activations, including networks of depth $1$, which do not have hidden
layers and realize affine transformations.
Secondly, for the emulation of discontinuous functions, we 
additionally
use the \emph{Binary Step Unit} (\emph{BiSU}) activation
\begin{align}
  \label{eq:heavidef}
  \heavi(x) =
  \begin{cases}
    0 &\text{if } x\le0,
    \\
    1 &\text{if } x>0,
  \end{cases}
\end{align}
which is also called \emph{Heaviside} function.
Alternatively, the BiSU can be defined to equal $\tfrac12$ in $x=0$.
That function, which we denote by $\heavii$, can be expressed
in terms of $\heavi$ via
$2 \heavii(x) = \heavi(x) + 1 - \heavi(-x)$ for all $x\in\R$.  Hence,
for every NN with $\heavii$ as activation function, there exists a
network with $\heavi$-activations instead, with proportional depth and
size.

%%%%%%%%%%%%%%%%%%%%%%%%%%%%%%%%%%%%%%%%%%%%%%%%%%%%%%%%%%%%%%%%%
\subsection{Operations on NNs}
\label{sec:nnoperations}
%%%%%%%%%%%%%%%%%%%%%%%%%%%%%%%%%%%%%%%%%%%%%%%%%%%%%%%%%%%%%%%%%
In the following sections, we will construct NNs from smaller networks
using a ReLU-based \emph{calculus of NNs}, which we now recall from
\cite{PV2018}.
\begin{proposition}[Parallelization of NNs {{\cite[Definition
      2.7]{PV2018}}}]
  \label{prop:parallel}
  For $d,L\in\N$ let $\Phi^1 = $ \linebreak
  $ \left
    ((A^{(1)}_1,b^{(1)}_1,\varrho^{(1)}_1),\ldots,(A^{(1)}_L,b^{(1)}_L,\varrho^{(1)}_L)\right
  ) $ and
  $\Phi^2 = \left
    ((A^{(2)}_1,b^{(2)}_1,\varrho^{(2)}_1),\ldots,(A^{(2)}_L,b^{(2)}_L,\varrho^{(2)}_L)\right
  ) $ be two NNs with input dimension $d$ and depth $L$.  Let the
  \emph{parallelization} $\Parallel{\Phi^1,\Phi^2}$ of $\Phi^1$ and
  $\Phi^2$ be defined by
  \begin{align*}
    \Parallel{\Phi^1,\Phi^2} := &\, ((A_1,b_1,\varrho_1),\ldots,(A_L,b_L,\varrho_L)),
    &&
    \\
    A_1 = &\, \begin{pmatrix} A^{(1)}_1 \\ A^{(2)}_1 \end{pmatrix},
    \quad
    A_\ell = \begin{pmatrix} A^{(1)}_\ell & 0\\0&A^{(2)}_\ell \end{pmatrix},
    &&
       \text{ for } \ell = 2,\ldots L,
    \\
    b_\ell = &\, \begin{pmatrix} b^{(1)}_\ell \\ b^{(2)}_\ell \end{pmatrix},
    \quad
    \varrho_\ell = \begin{pmatrix} \varrho^{(1)}_\ell \\ \varrho^{(2)}_\ell \end{pmatrix},
                                &&
                                   \text{ for } \ell = 1,\ldots L.
  \end{align*}

  Then,
  \begin{align*}
    \realiz{\Parallel{\Phi^1,\Phi^2}} (x)
    = &\, ( \realiz{\Phi^1}(x), \realiz{\Phi^2}(x) ),
        \quad
        \text{ for all } x\in\R^d,
    \\
    \depth(\Parallel{\Phi^1,\Phi^2}) = L, &
                                            \qquad
                                            \size(\Parallel{\Phi^1,\Phi^2}) = \size(\Phi^1) + \size(\Phi^2)
                                                  .
  \end{align*}

\end{proposition}

The parallelization of more than two NNs is done by repeated application of
Proposition \ref{prop:parallel}.
\begin{proposition}[Sum of NNs]
  \label{prop:sum}

  For $d,N,L\in \N$ let
  $ \Phi^1 = \left
    ((A^{(1)}_1,b^{(1)}_1,\varrho^{(1)}_1),\ldots,(A^{(1)}_L,b^{(1)}_L,\varrho^{(1)}_L)\right
  ) $ and
  $\Phi^2 = \Big
  ((A^{(2)}_1,b^{(2)}_1,\varrho^{(2)}_1),\ldots,(A^{(2)}_L,b^{(2)}_L,\varrho^{(2)}_L)\Big)
  $
  be two NNs with input dimension $d$, output dimension $ N $ and
  depth $L$.  Let the \emph{sum} $\Phi^1+\Phi^2$ of $\Phi^1$ and
  $\Phi^2$ be defined by
  \begin{align*}
    \Phi^1 + \Phi^2 := &\, ((A_1,b_1,\varrho_1),\ldots,(A_L,b_L,\varrho_L)),
    &&
    \\
    A_1 =& \begin{pmatrix} A^{(1)}_1 \\ A^{(2)}_1 \end{pmatrix},
    \quad
    b_1 = \begin{pmatrix} b^{(1)}_1 \\ b^{(2)}_1 \end{pmatrix},
    \quad
    \varrho_1 = \begin{pmatrix} \varrho^{(1)}_1 \\ \varrho^{(2)}_1 \end{pmatrix},
    \\
    A_\ell =& \begin{pmatrix} A^{(1)}_\ell & 0\\0&A^{(2)}_\ell \end{pmatrix},
                                                   \quad
                                                   b_\ell =  \begin{pmatrix} b^{(1)}_\ell \\ b^{(2)}_\ell \end{pmatrix},
    \quad
    \varrho_\ell = \begin{pmatrix} \varrho^{(1)}_\ell \\ \varrho^{(2)}_\ell \end{pmatrix},
                       &&
                          \text{ for } \ell = 2,\ldots L-1.
    \\
    A_L =& \begin{pmatrix} A^{(1)}_L & A^{(2)}_L \end{pmatrix},
                                       \quad
                                       b_L = b^{(1)}_L + b^{(2)}_L, 
                                       \quad
                                       \varrho_L = \Id_{\R^N}.
  \end{align*}

  Then,
  \begin{align*}
    \realiz{\Phi^1+\Phi^2} (x)
    = &\, \realiz{\Phi^1}(x) + \realiz{\Phi^2}(x),
        \quad
        \text{ for all } x\in\R^d,
    \\
    \depth(\Phi^1+\Phi^2) = &\, L, \qquad
                              \size(\Phi^1+\Phi^2)
                              \leq \size(\Phi^1) + \size(\Phi^2)
                                                                .
  \end{align*}
\end{proposition}
Next, we define the sparse concatenation of two NNs, which realizes
exactly the composition of the realizations of the two networks using
the fact that we allow the ReLU activation.  See Figure
\ref{fig:sparse-conc} for a sketch of the NN structure.  The sparse
concatenation is a construction which allows to bound the size of the
concatenation as 2 times the sum of the sizes of the individual
networks. This bound does not hold if we combine the affine
transformation of the output layer of $ \Phi^2 $ with the affine
transformation of the input layer of $ \Phi^1 $.
\begin{proposition}[Sparse Concatenation of NNs {{\cite[Remark
      2.6]{PV2018}}}]
  \label{prop:concat}
  For $L^{(1)},L^{(2)}\in\N$, let
  $ \Phi^1 = $
  $ \left
    ((A^{(1)}_1,b^{(1)}_1,\varrho^{(1)}_1),\ldots,(A^{(1)}_{L^{(1)}},b^{(1)}_{L^{(1)}},\varrho^{(1)}_{L^{(1)}})\right
  ) $ and
  $\Phi^2 = \Big
  ((A^{(2)}_1,b^{(2)}_1,\varrho^{(2)}_1),\ldots,(A^{(2)}_{L^{(2)}},b^{(2)}_{L^{(2)}},\varrho^{(2)}_{L^{(2)}})\Big)
  $
  be two NNs with depths $L^{(1)}$ and $L^{(2)}$, respectively, such
  that $N^{(2)}_{L^{(2)}} = N^{(1)}_0$, i.e. the output dimension of
  $\Phi^2$ equals the input dimension of $\Phi^1$.  Let the
  \emph{sparse concatenation} $\Phi^1\sconc \Phi^2$ of $\Phi^1$ and
  $\Phi^2$ be a NN of depth $L := L^{(1)} + L^{(2)}$ defined by
  \begin{align*}
    \Phi^1 \sconc \Phi^2 := &\, ((A_1,b_1,\varrho_1),\ldots,(A_L,b_L,\varrho_L)),
    \\
    (A_\ell,b_\ell,\varrho_\ell) = &\, (A^{(2)}_\ell,b^{(2)}_\ell,\varrho^{(2)}_\ell),
                                   &&
                                      \text{ for }\ell=1,\ldots,L^{(2)}-1,
    \\
    A_{L^{(2)}} = &\, \begin{pmatrix} A^{(2)}_{L^{(2)}} \\ -A^{(2)}_{L^{(2)}} \end{pmatrix},
    \quad
    b_{L^{(2)}} = \begin{pmatrix} b^{(2)}_{L^{(2)}} \\ -b^{(2)}_{L^{(2)}} \end{pmatrix},
    \quad
    \varrho_{L^{(2)}} = \begin{pmatrix} \relu \\ \vdots \\ \relu \end{pmatrix},
    \\
    A_{L^{(2)}+1} = &\, \begin{pmatrix} A^{(1)}_1 & -A^{(1)}_1 \end{pmatrix},
                                                    \quad
                                                    b_{L^{(2)}+1} = b^{(1)}_1,
                                                    \quad
                                                    \varrho_{L^{(2)}+1} = \varrho^{(1)}_1,
    \\
    (A_\ell,b_\ell,\varrho_\ell) = &\, (A^{(1)}_{\ell-L^{(2)}},b^{(1)}_{\ell-L^{(2)}}, \varrho^{(1)}_{\ell-L^{(2)}}),
                                   &&
                                      \text{ for }\ell=L^{(2)}+2,\ldots,L^{(1)}+L^{(2)}.
  \end{align*}

  Then, it holds that
  \begin{align*}
    \realiz{\Phi^1 \sconc \Phi^2} = &\, \realiz{\Phi^1} \circ \realiz{\Phi^2},
                                      \qquad
                                      \depth(\Phi^1 \sconc \Phi^2) = \, L^{(1)} + L^{(2)},
    \\
    \size(\Phi^1 \sconc \Phi^2) \leq &\, \size(\Phi^1) + \sizefirst(\Phi^1) + \sizelast(\Phi^2) + \size(\Phi^2)
                                       \leq 2 \size(\Phi^1) + 2 \size(\Phi^2)
                                                               .
  \end{align*}

\end{proposition}
\begin{figure}
  \centering \includegraphics[width=0.7\textwidth]{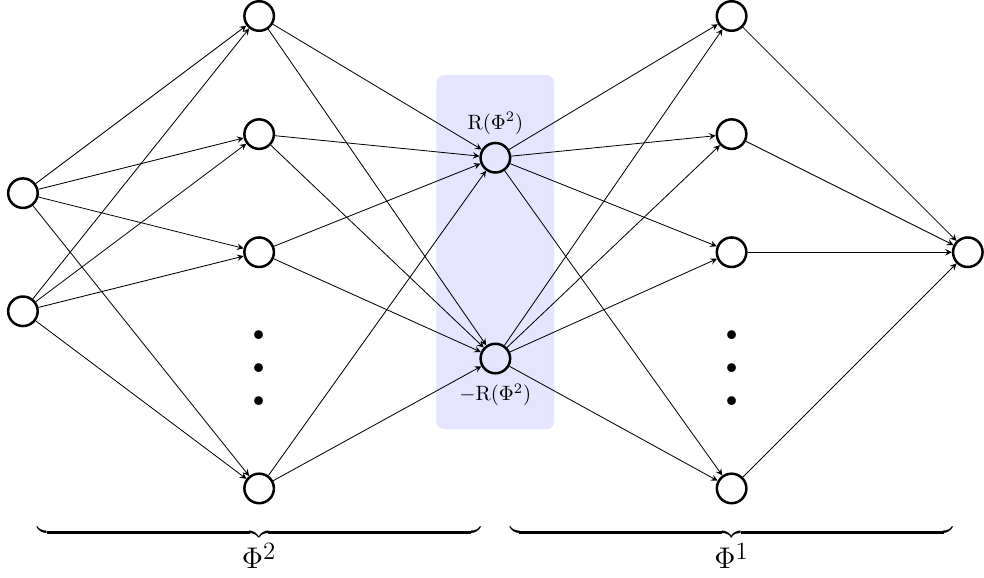}
  \caption{Illustration of Proposition \ref{prop:concat}.  In blue,
    the additional ReLU layer that evaluates the positive and negative
    parts of $ \realiz{\Phi^2} $. }
  \label{fig:sparse-conc}
\end{figure}

Proposition \ref{prop:parallel}
only applies to networks of equal depth.  To parallelize two networks
of unequal depth, the shallowest can be concatenated with a network
that emulates the identity using Proposition \ref{prop:concat}.  One
example of ReLU NNs that emulate the identity is provided by the
following proposition.
\begin{proposition}[ReLU NN emulation of $\Id_{\R^d}$ {{\cite[Remark
      2.4]{PV2018}}}]
  \label{prop:idnn}
  For all $d,L\in\N$, there exists a ReLU NN $\Phi^{\Id}_{d,L}$ with
  input dimension $d$, output dimension $d$ and depth $L$ which
  satisfies $\realiz{\Phi^{\Id}_{d,L}} = \Id_{\R^d}$,
  $\depth(\Phi^{\Id}_{d,L}) = L$ and
  $\size(\Phi^{\Id}_{d,L}) \leq 2dL$.
\end{proposition}

%%%%%%%%%%%%%%%%%%%%%%%%%%%%%%%%%%%%%%%%%%%%%%%%%%%%%%%%%%%%%%%%%
\subsection{Expression of specific functions}
\label{sec:nnspecificfct}
%%%%%%%%%%%%%%%%%%%%%%%%%%%%%%%%%%%%%%%%%%%%%%%%%%%%%%%%%%%%%%%%%

In the following, we will need ReLU NNs which emulate the minimum or
maximum of $d\in\N$ inputs.  
These are provided in Lemma \ref{lem:minmaxrelu}.  
We also need to multiply values from a bounded
interval $[-\kappa,\kappa]$ for $\kappa>0$ by values from the discrete
set
$\{0, 1\}$, which is the range of the BiSU defined in
\eqref{eq:heavidef}.  A ReLU NN which emulates such multiplications
exactly is constructed in the proof of Proposition
\ref{prop:multbystep-alt} below.

\begin{lemma}[ReLU NN emulation of $\min$ and $\max$, {{\cite[Proof
      Theorem 3.1]{HLXZ2020}}}]
  \label{lem:minmaxrelu}
  For all $d\in\N$, there exist ReLU NNs $\Phi^{\max}_d$ and
  $\Phi^{\min}_d$ which satisfy
  \begin{align*}
    \realiz{\Phi^{\max}_d} (x) = &\, \max\{x_1,\ldots,x_d\},
    &&
       \text{ for all } x\in\R^d,
    \\
    \realiz{\Phi^{\min}_d} (x) = &\, \min\{x_1,\ldots,x_d\},
    &&
       \text{ for all } x\in\R^d,
    \\
    \depth(\Phi^{\max}_d) = &\, \depth(\Phi^{\min}_d) \leq 2 + \log_2(d), \qquad
                              \size(\Phi^{\max}_d) = \, \size(\Phi^{\min}_d) \leq Cd 
                                                             .
  \end{align*}
  Here, the constant $C>0$ is independent of $d$ and of the NN sizes
  and depths.
\end{lemma}

In space dimension $d=1$, we may take
$\Phi^{\min}_1 := \Phi^{\max}_1 := \Phi^{\Id}_{1,2}$.

\begin{remark}\label{rmk:EstMinMax}
  The network $\Phi_d^{\max}$ is obtained by repeated applications of
  $\Phi_2^{\max}$, which itself can for instance be constructed as
  \begin{equation*}
    \Phi^{\max}_2 :=\, \left(
      \left( \begin{pmatrix} 1 & -1 \\ 0 & 1 \\ 0 & -1 \end{pmatrix},
        \begin{pmatrix} 0 \\ 0 \\ 0\end{pmatrix} ,
        \begin{pmatrix} \relu \\ \relu \\ \relu \end{pmatrix} \right) ,
      \left( \begin{pmatrix} 1 & 1 & -1 \end{pmatrix},  0 ,
        \Id_\R \right)
    \right).
  \end{equation*}
  We point out that this construction of $\Phi_2^{\max}$ leads to a
  slightly more efficient representation of $\Phi_d^{\max}$ than the
  one given in \cite[Theorem 3.1]{HLXZ2020}, as it requires less
  neurons, weights and biases. However, this will merely improve the
  constant $C$ in Lemma \ref{lem:minmaxrelu}, but not the stated
  asymptotic $d$-dependence of $L(\Phi_2^{\max})$ and
  $M(\Phi_2^{\max})$. The remark applies verbatim to $\min$ networks.
\end{remark}
\begin{remark}[min/max with recurrent nets]\label{rmk:EstMinMax2}
  The $d$-dependence can be completely avoided by admitting recurrent
  neural nets (RNNs), i.e., RNNs can express the maximum of $d$ inputs
  with a network of size, depth and width $O(1)$. We briefly sketch
  the idea: An RNN allows for information to flow backwards, i.e., we
  can take the output of $\Phi_2^{\max}$ in time step $t$ as one of
  its inputs at time step $t+1$. With the initialization
  $\tilde x_{0}:=x_1$, this leads to the iteration
  \begin{equation*}
    \tilde{x}_t = \Phi_2^{\max}(\tilde{x}_{t-1}, x_{t}),
  \end{equation*}
  where the network receives in step $t$ the input $x_t$. Then the
  network's output $\tilde x_n$ in step $n$ equals
  $\max\{x_1,\dots,x_n\}$.
  The remark applies verbatim to $\min$ networks.
\end{remark}

The following proposition provides the exact ReLU NN emulation of
products of elements from a bounded interval $[-\kappa,\kappa]$ for
$\kappa>0$ by elements from the discrete set
$\{0,1\}$. 
The network depth and size are independent of $\kappa$.

\begin{proposition}
  \label{prop:multbystep-alt}
  For all $ d\in\N$ and $\kappa>0$ there exists a ReLU NN
  $\Phi^{\times}_{d,\kappa}$
  \begin{align*}
    \realiz{\Phi^{\times}_{d,\kappa}}(x_1,\ldots,x_d,y) = &\, x y = (x_1y,\ldots,x_dy)^\top,
    &&
       \text{ for all } x\in [-\kappa,\kappa]^{d} \text{ and } y\in\{0,1\},
    \\
    \depth(\Phi^{\times}_{d,\kappa}) \leq &\, 2, \qquad
                                            \size(\Phi^{\times}_{d,\kappa}) \leq \, 12 d.
  \end{align*}
\end{proposition}
A proof of Proposition \ref{prop:multbystep-alt} is given in the
appendix.

The exact BiSU emulation of indicator functions is the topic of the
following lemma.  An illustration of the network defined in the lemma
is given in Figure \ref{fig:indic-func}.
\begin{lemma}[Emulation of indicator functions]
  \label{lem:indicemulation-lowdimalt}
  For $d,N\in\N$ and $n \in \{0,\ldots,N\}$, let
  $A_1,\ldots,A_N\in\R^{1\times d}$ and $b_1,\ldots,b_N\in\R^1$ be
  such that
  \begin{align*}
    \domain := \bigcap_{i=1,\ldots,n} \set{x\in \R^d}{A_i x + b_i = 0}
    \cap \bigcap_{i=n+1,\ldots,N} \set{x\in \R^d}{A_i x + b_i > 0}
    \neq \emptyset.
  \end{align*}
  Let the NN $\Phi^{\mathbbm{1}}_\domain$ with layer sizes $N_0 = d$,
  $N_1 = N+ n$ and $N_2 = 1 = N_3$ be defined as
  \begin{align*}
    \Phi^{\mathbbm{1}}_\domain := &\, \left(
                                    \left( \begin{pmatrix} A_1 \\ - A_1 \\ \vdots \\ A_{n} \\ - A_{n} \\ A_{n+1} \\ \vdots \\ A_N \end{pmatrix},
    \begin{pmatrix} b_1 \\ - b_1 \\ \vdots \\ b_{n} \\ - b_{n} \\ b_{n+1} \\ \vdots \\ b_N \end{pmatrix},
    \begin{pmatrix} \heavi \\ \heavi \\ \vdots \\ \heavi \\ \heavi \\ \heavi \\ \vdots \\ \heavi \end{pmatrix} \right) ,
    ( A, b, \varrho ),
    \left( 1 ,  0 , \Id_\R \right)
    \right),
    \\
    A := &\, \begin{pmatrix} -1 & \cdots & -1 & 1 & \cdots & 1 \end{pmatrix} \in\bbR^{1\times (N+{n})},
                                                             \quad
                                                             b := -(N-{n}-\tfrac14) \in\bbR^1,
                                                             \quad
                                                             \varrho := \heavi ,
  \end{align*}
  where the first $2{n}$ elements of $A$ equal $-1$ and the last
  $N - {n}$ equal $1$.

  Then, for all $x\in\R^d$
  \begin{align*}
    \realiz{\Phi^{\mathbbm{1}}_{\domain}}(x)
    = \,
    \begin{cases} 1 & \text{ if } x\in \domain,
      \\
      0 & \text{ otherwise},
    \end{cases}
         \qquad
         \depth(\Phi^{\mathbbm{1}}_\domain)
         = \, 3 , \qquad
         \size(\Phi^{\mathbbm{1}}_\domain)
                       \leq \, (d+2)(N+n) + 2
                           .
  \end{align*}
\end{lemma}
A proof of Lemma \ref{lem:indicemulation-lowdimalt} is provided in the
appendix.

\begin{figure}
  \centering \includegraphics[width=0.7\textwidth]{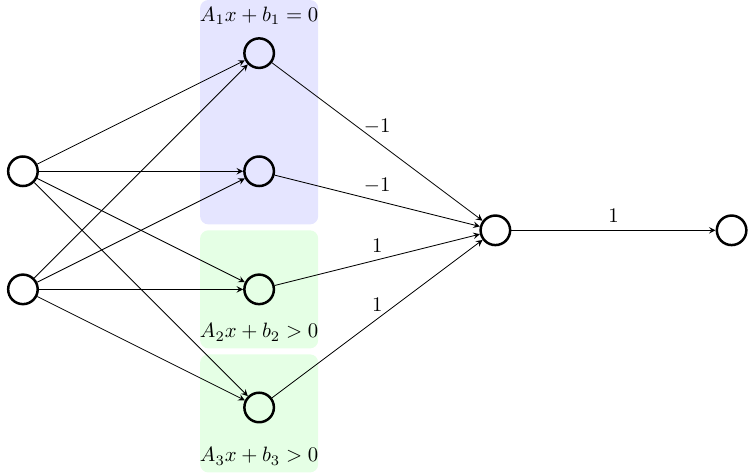}
  \caption{Illustration of Lemma \ref{lem:indicemulation-lowdimalt}
    for $ d=2$, $n = 1$, $N = 3$.  In blue, the neurons corresponding
    to one hyperplane ($ i=1 $), in green the half-spaces ($ i=2,3 $).
  }
  \label{fig:indic-func}
\end{figure}

%%%%%%%%%%%%%%%%%%%%%%%%%%%%%%%%%%%%%%%%%%%%%%%%%%%%%%%%%%%%%%%%%%
\section{NN emulation of lowest order conforming Finite Element shape
  functions}
\label{sec:nnfem}
%%%%%%%%%%%%%%%%%%%%%%%%%%%%%%%%%%%%%%%%%%%%%%%%%%%%%%%%%%%%%%%%%%
Consider a bounded polytopal domain $ \domain \subset \R^d $,
$ d \in \N\backslash\{1\} $, and a regular simplicial partition
$\calT$ of $\domain$.

In Sections \ref{sec:pwc}--\ref{sec:cpwlnonconvex}, we will present
neural network emulations of the lowest order conforming FEM spaces
for $H^1(\domain), H^0(\curl,\domain), H^0(\divv,\domain)$ and
$L^2(\domain)$.\footnote{Throughout this section, we will regularly
  refer to \cite{ErnGuermondBookI2021}, where only Lipschitz domains
  are considered.  We stress that the finite element spaces
  $\Sz(\calT,\domain)$, $\Ne(\calT,\domain)$, $\RT(\calT,\domain)$ and
  $\So(\calT,\domain)$ can be defined on regular, simplicial
  triangulations $\calT$ of all bounded, polytopal domains $\domain$
  and that our NN emulation results from this section apply to all
  such $\domain$.}  These finite-dimensional spaces appear naturally
in structure-preserving discretizations of the de Rham complex.
They are a key ingredient for \emph{variationally consistent DNN
  emulations} of differential operators appearing in the de Rham
complex.  For each type of shape function, we explicitly define a
network which emulates that shape function exactly.  Global
approximations can be obtained by taking a linear combination of these
shape functions using Proposition \ref{prop:sum} (scalar multiples of
shape functions are obtained by scaling all weights and biases of the
output layer).  We will detail this in Proposition
\ref{prop:basisnet}.

For shape functions which are discontinuous after extending them to
$\domain$ by the value zero outside their domain of definition, we use
Lemma \ref{lem:indicemulation-lowdimalt} based on BiSU activation to
emulate indicator functions of (parts of) their domain of definition.
We then use Proposition \ref{prop:multbystep-alt} based on ReLU
activation to multiply a continuous, piecewise linear function, which
is equal to the shape function on part of $\domain$, by the indicator
function of that part of the domain.

The following lemma provides NN emulations of possibly discontinuous,
piecewise linear functions, and will be used repeatedly in Sections
\ref{sec:pwc}--\ref{sec:cpwlnonconvex}.  A sketch of the NN structure
is given in Figure \ref{fig:pwl-nn}.

\begin{lemma}[Emulation of piecewise linear
  functions] 
  \label{lem:dpwl}
  For $d,s,\mu\in\N$ let $\domain\subset\R^d$ be a bounded polytope
  and $\calT$ be a regular, simplicial partition of $\domain$ with
  $s = \snorm\calT$ elements, $\calT = \{ T_i \}_{i=1,\ldots,s}$.
  Let $u:\domain\to\R^{\mu}$ be a function which for all
  $i=1,\ldots,s$ satisfies $u|_{T_i} \in[\bbP_1]^{\mu}$ and
  $u|_{T_i}(x) = A^{(i)} x + b^{(i)}$, $x\in T_i$.

  Then, for any
  \begin{equation}
    \label{eq:dpwlkappa}
    \kappa \ge \max_{i=1,\ldots,s}\sup_{x \in T_i} \norm[\infty]{A^{(i)} x + b^{(i)}},
  \end{equation}
  \begin{equation}
    \label{eq:dpwldef}
    \Phi_u^{PwL} := \sum_{i = 1}^{s} \Phi^{\times}_{{\mu},\kappa} \sconc \Parallelc{ \Phi_{{\mu},2}^{\Id} \sconc
      \left( \left(A^{(i)},b^{(i)},\Id_{\R^{\mu}}\right )\right ) , \Phi_{T_i}^{\mathbbm{1}}}
  \end{equation}
  satisfies $u(x) = \realizc{\Phi_u^{PwL}}(x)$ for all
  $x \in \cup_{i=1}^s T_i$ and $\realizc{\Phi_u^{PwL}}(x) = 0$ for all
  $x\in\R^d\setminus\cup_{i=1}^s T_i$.  Furthermore, if
  $\norm[0]{A^{(i)}} + \norm[0]{b^{(i)}} \leq m$ for all
  $i=1,\ldots,s$, then there exists $C>0$ independent of $d$ and
  $\calT$ such that
  \begin{align*}
    \depth(\Phi_u^{PwL}) = 5\;,\qquad
    \size(\Phi_u^{PwL}) \leq &\, C s ( \mu + m + d^2)
                                     .
  \end{align*}
\end{lemma}
\begin{proof}
  Firstly, we observe that indeed $ u(x) = \realizc{\Phi_u^{PwL}}(x) $
  for all $x \in \cup_{i=1}^s T_i$
  and $\realizc{\Phi_u^{PwL}}(x) = 0$ for all
  $x \in \R^d \setminus \cup_{i=1}^s T_i$.

  Secondly, we estimate
  \begin{align*}
    \depth(\Phi_u^{PwL}) = &\, \depth( \Phi^{\times}_{\mu,\kappa} ) + \depth( \Phi_{T_i}^{\mathbbm{1}} )
                             =5
                             ,
    \\
    \size(\Phi_u^{PwL}) \leq &\, s \left( 2 \size\left( \Phi^{\times}_{{\mu},\kappa} \right)
                               + 2 \size\left( \Parallelc{ \Phi_{{\mu},2}^{\Id} \sconc
                               \left( \left(A^{(i)},b^{(i)},\Id_{\R^{\mu}}\right )\right ) , \Phi_{T_i}^{\mathbbm{1}}} \right) \right)
    \\
    \leq &\, s \left( 2 \size\left( \Phi^{\times}_{{\mu},\kappa} \right)
           + 4 \size( \Phi_{{\mu},2}^{\Id} )
           + 4 \size\left( \left( \left( A^{(i)},b^{(i)},\Id_{\R^{\mu}} \right )\right ) \right)
           + 2 \size\left( \Phi_{T_i}^{\mathbbm{1}} \right) \right)
    \\
    \leq &\, s ( C {\mu} + C {\mu} + C m + C d^2 )
           \leq C s ({\mu} + m + d^2)
                                            ,
  \end{align*}
  where we applied in the last line Lemma
  \ref{lem:indicemulation-lowdimalt} with $ N = d+1, n = 0 $.
\end{proof}
\begin{figure}
  \centering \includegraphics[width=0.8\textwidth]{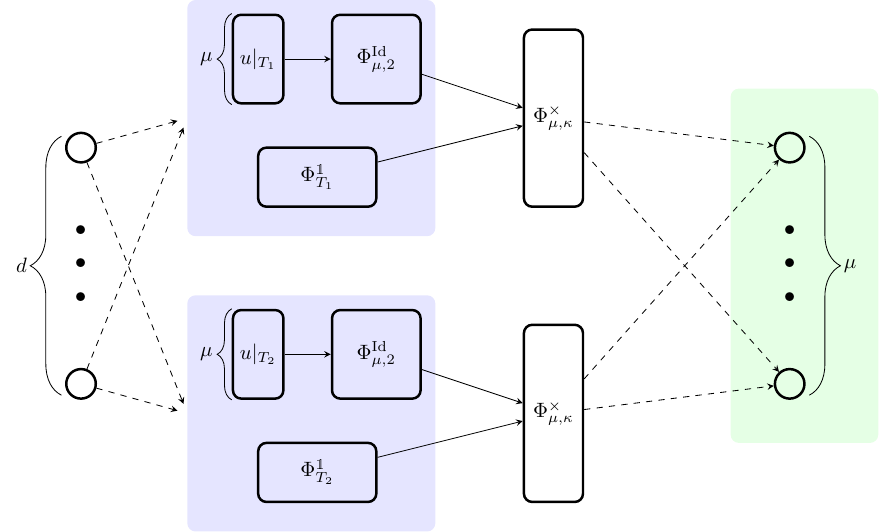}
  \caption{Illustration of Lemma \ref{lem:dpwl} for $ s=2 $.  Solid
    lines indicate sparse concatenation, dashed lines visualize layers
    integrated in the blocks.  Each of the blue groups represents a
    parallelization from Proposition \ref{prop:parallel}.  The
    elementwise assembly is done in the output layer (green).}
  \label{fig:pwl-nn}
\end{figure}
%%%%%%%%%%%%%%%%%%%%%%%%%%%%%%%%%%%%%%%%%%%%%%%%%%%%%%%%%%%%%%%%%%
\subsection{Piecewise constants $\Sz$}
\label{sec:pwc}
%%%%%%%%%%%%%%%%%%%%%%%%%%%%%%%%%%%%%%%%%%%%%%%%%%%%%%%%%%%%%%%%%%
The lowest order approximation space for $ L^2(\domain) $ is the
finite dimensional subspace $\Sz(\calT,\domain)$ from \eqref{def:S0}.
A basis is given by $ \{ \theta_T^{\Sz}\}_{T\in \calT} $, whose
elements are indicator functions $ \theta_T^{\Sz} := \mathbbm{1}_T $.
They can be expressed by applying Lemma
\ref{lem:indicemulation-lowdimalt} with $N=d+1$ and $n=0$: for all
$T = \conv(\{a_0,\ldots,a_d\}) \in \calT$, we define
$ (A_i,b_i) \in \R^{1\times (d+1)} $, $ i=1,\ldots,d+1 $ by the
relations
\begin{equation}\label{S0affine}
  (A_i,b_i)\begin{pmatrix}
    (a_{0})_1 &  & (a_d)_1 \\
    & \ddots& \\
    (a_0)_d &  & (a_d)_d \\
    1 & \cdots & 1
  \end{pmatrix} = \bm{e}_i^\top , \qquad \text{ for all }
  i=1,\ldots,d+1,
\end{equation}
where $ (\bm{e}_i)_j = \delta_{ij} $, so that
$ T = \bigcap_{i=1,\ldots,d+1} \set{x\in \R^d}{A_i x + b_i>0} $.
Then there exists $C>0$ independent of $d$ and $\calT$ such that for
all $T\in\calT$ the NN $\Phi^{\Sz}_T := \Phi_T^{\mathbbm{1}}$
satisfies
\begin{align}\label{dnnS0}
  \theta_T^{\Sz} = \, \realiz{\Phi^{\Sz}_T} ,
  \quad
  \depth(\Phi^{\Sz}_T)
  = \, 3
  ,
  \quad
  \size(\Phi^{\Sz}_T) \leq (d+2)(d+1) + 2 \leq C d^2.
\end{align}
%%%%%%%%%%%%%%%%%%%%%%%%%%%%%%%%%%%%%%%%%%%%%%%%%%%%%%%%%%%%%%%%%%
\subsection{Raviart-Thomas elements $\RT$}
\label{sec:rt}
%%%%%%%%%%%%%%%%%%%%%%%%%%%%%%%%%%%%%%%%%%%%%%%%%%%%%%%%%%%%%%%%%%

We introduce a basis of $\RT(\calT,\domain)$ from \eqref{def:RT0}, and
provide a NN emulation of those basis functions.  For
$ f \subset \partial \domain $, we define $s(f) = 1$ and
$ \theta^{\RT}_f(x) := \frac{|f|}{d|T_1 |}(x- a_1) \mathbbm{1}_{T_1}
$, where $ f\subset \overline{{T_1}} $, $T_1 \in \calT $ and $a_1$ is
the only vertex of $ T_1 $ that does not belong to
$ \overline{f}$.\footnote{ We use a different normalization of the
  shape functions than in \cite[Section 14.1]{ErnGuermondBookI2021}.
  This is inconsequential for the ensuing analysis.}  For interior
faces $ f \subset \domain $ we define $s(f) = 2$ and construct
$ \theta^{\RT}_f $ by assembling local shape functions of the
neighboring simplices $ T_1, T_2 $ with
$ \overline{f} = \overline{T_1}\cap \overline{T_2}$, \cite[Equation
(14.3)]{ErnGuermondBookI2021}
\begin{equation} \label{eq:RTshape} \theta_f^{\RT}(x):=
  \begin{cases}
    \frac{|f|}{d|T_1|}(x-a_1) & \text{ if } x \in T_1, \\
    -\frac{|f|}{d|T_2|}(x-a_2) & \text{ if } x \in T_2, \\
    0 & \text{ if } x \notin \overline{T_1\cup T_2},
  \end{cases}
\end{equation}
where $ a_1,a_2 $ are the the only vertices of $ T_1,T_2 $,
respectively, not belonging to $ \overline{f} $.
The functions $ \{\theta^{\RT}_f\}_{f\in \calF} $ form a basis of
$ \RT(\calT,\domain)$ (see, e.g., \cite[Proposition
14.1]{ErnGuermondBookI2021}).
\begin{proposition} \label{prop:RT} Given $ f \in \calF $, let
  $ \{T_i\}_{i=1}^{s(f)}$ be the simplices adjacent to $ f $ and let
  $ a_i := (\calV \cap \overline{T_i}) \setminus \overline{f} \in
  \R^{d}, i=1,\ldots,s(f)$.  Then
  \begin{equation}
    \label{eq:RTdef2}
    \Phi_f^{\RT} := \sum_{i = 1}^{s(f)} (-1)^{i-1} \Phi^{\times}_{d,\kappa} \sconc \Parallelc{ \Phi_{d,2}^{\Id} \sconc \left (\left (\tfrac{|f|}{d|T_i|}\Id_{d\times d},-\tfrac{|f|}{d|T_i|}a_i,\Id_{\R^d}\right )\right ) , \Phi_{T_i}^{\mathbbm{1}}}
  \end{equation}
  satisfies $\theta_f^{\RT}(x) = \realiz{\Phi^{\RT}_f}(x)$ for a.e.\
  $ x \in \domain$, for any
  \begin{equation}
    \label{eq:RTkappa2}
    \kappa \ge \max_{i=1,\ldots,s(f)}\sup_{x \in T_i} \tfrac{|f|}{d|T_i|}  \norm[\infty]{x-a_i}.
  \end{equation}
  In addition, there exists an absolute constant $C>0$ independent of
  $d$ and $\calT$ such that for all $f\in\calF$
  \begin{align*}
    \depth(\Phi_f^{\RT}) = 5
    , \qquad
    \size(\Phi_f^{\RT}) \leq C d^2 s(f) \leq 2 C d^2
    .
  \end{align*}
\end{proposition}

\begin{remark}
  \label{rem:RTkappa}
  We note that the right-hand side of \eqref{eq:RTkappa2} is bounded
  from above by a constant which only depends on $d$ and the shape
  regularity constant $C_{\rm sh}$.
\end{remark}

\begin{proofof}{Proposition \ref{prop:RT}}
  Firstly, we observe that indeed
  $ \theta_f^{\RT}(x) = \realizc{\Phi_f^{\RT}}(x) $
  for all $ x \in \domain \setminus \partial \calT $, where
  $ \partial \calT := \bigcup_{T\in \calT} \partial T$.

  Secondly, we apply Lemma \ref{lem:dpwl} with $\mu=d$, $m=2d$ and
  $s = s(f)$.
\end{proofof}

Alternatively, we can build the same shape functions by enforcing
strongly, via ReLU activation, continuity of the component normal to
$ f $, as imposed in \eqref{def:RT0}.
We select the unit normal vector $ n_f $ to $ f $ pointing towards
$ T_2 $ and an orthonormal system $ \{t_1,\ldots,t_{d-1}\} $ spanning
the hyperplane tangent to $ f $. Then, we decompose
\begin{equation}\label{orthRT}
  \theta_f^{\RT}(x) = (\theta_f^{\RT}(x)\cdot n_f) n_f + \sum_{j = 1}^{d-1}(\theta_f^{\RT}(x)\cdot t_j) t_j.
\end{equation}
Thus, it suffices to compute separately $\theta_f^{\RT}(x)\cdot n_f $
and $ \theta_f^{\RT}(x)\cdot t_j $ and to take the linear combination
\eqref{orthRT} in the last layer.  The proof of the next proposition
will be given in the appendix.
\begin{proposition} \label{prop:RTv2} Given $ f \in \calF $, let
  $ \{T_i\}_{i=1}^{{s}(f)} $ be the simplices adjacent to $ f $.  Then
  there exist
  $ A_{n_f}^{(i)} \in \R^{1\times d},b_{n_f}^{(i)} \in \R$,
  $ i = 1,\ldots,{s}(f)$ such that
  \begin{equation}\label{eq:RTdefn}
    \Phi^{\RT,{\perp}}_f := \Phi^\times_{1,1} \sconc
    \Parallelc{\Phi_{{s}(f)}^{\min}\sconc \left ( \left (
          \begin{pmatrix} A_{n_f}^{(1)} \\ \vdots \\ A_{n_f}^{({s}(f))} \end{pmatrix},
          \begin{pmatrix} b_{n_f}^{(1)} \\ \vdots \\ b_{n_f}^{({s}(f))} \end{pmatrix},
          \Id_{\R^{s(f)}} \right ) \right ),
      \sum_{i=1}^{{s}(f)} \Phi_{T_i}^{\mathbbm{1}} + \Phi_{f}^{\mathbbm{1}} }
  \end{equation}
  satisfies
  $ \theta_f^{\RT}(x)\cdot n_f = \realizc{ \Phi^{\RT,{\perp}}_f
  }(x) $ for a.e.\ $ x \in \domain $ and every $ x\in f $.  
  There also
  exist $ A_{t_j}^{(i)} \in \R^{1\times d},b_{t_j}^{(i)} \in \R$,
  $ i=1,\ldots,s(f) $, $ j = 1,\ldots,d-1 $ 
  such that
  \begin{equation}\label{eq:RTdeft}
    \Phi^{\RT,t_j}_f  := \sum_{i=1}^{s(f)}\Phi^\times_{1,\kappa} \sconc \Parallelc{\Phi^{\Id}_{1,2} \sconc \left ((A_{t_j}^{(i)},b_{t_j}^{(i)},\Id_{\R})\right ) , \Phi_{T_i}^{\mathbbm{1}}}
  \end{equation}
  satisfies
  $ \theta_f^{\RT}(x)\cdot t_j = \realizc{ \Phi^{\RT,t_j}_f }(x) $ for
  a.e.\ $ x \in \domain $, where
  \begin{equation}\label{eq:RTkappat}
    \kappa \ge \max_{\substack{i=1,\ldots,s(f) \\ j = 1,\ldots,d-1}}\sup_{x \in T_i} \left |A_{t_j}^{(i)}x+b_{t_j}^{(i)}\right |.
  \end{equation}
  In addition, there exists a constant $C>0$ that is independent of
  $d$ and $\calT$ such that for all $ f \in \calF $
  \begin{align*}
    \depth(\Phi^{\RT, \perp}_f) = &\, 5
                                        ,
    & \depth(\Phi^{\RT,t_j}_f) = &\, 5
                                   ,
    \\
    \size(\Phi^{\RT,\perp}_f) \leq &\, C d^2 s(f) \leq 2 C d^2
                                          ,
    & \size(\Phi^{\RT,t_j}_f)
      \leq
                                 &\, C d^2 s(f) \leq 2 C d^2
                                   .
  \end{align*}
\end{proposition}

\begin{remark}
  \label{rem:RTv2kappa}
  The right-hand side of Equation \eqref{eq:RTkappat} is bounded from
  above by a constant which only depends on $d$ and the shape
  regularity constant $C_{\rm sh}$ of the mesh $ \calT $.
\end{remark}

\begin{corollary}
  \label{cor:RTv2}
  For all $f\in\calF$, the NN
  \begin{equation}
    \label{eq:RTv2def}
    \Phi^{\RT,*}_f
    := \left(\left(  n_f , 0 ,
	\Id_{\R^d} \right)\right) \sconc \Phi^{\RT,{\perp}}_f
    + \sum_{j=1}^{d-1} \left(\left( t_j , 0 ,
	\Id_{\R^d} \right)\right) \sconc \Phi^{\RT,t_j}_f
  \end{equation}
  satisfies $ \theta_f^{\RT}(x) = \realizc{ \Phi^{\RT,*}_f }(x) $ for
  a.e.\ $x\in\domain$ and
  $ ( \theta_f^{\RT}(x)\cdot n_f ) n_f = \realizc{ \Phi^{\RT,*}_f }(x)
  $ for all $ x\in f $.  In addition, there exists $C>0$ independent
  of $d$ and $\calT$ such that for all $f\in\calF$
  \begin{align*}
    \depth( \Phi^{\RT,*}_f ) = \, 6
    ,
    \qquad
    \size( \Phi^{\RT,*}_f ) \leq \, C d^3 .
  \end{align*}
\end{corollary}

\begin{proof}
  We estimate the network size and depth as follows:
  \begin{align*}
    \depth( \Phi^{\RT,*}_f ) = &\, 6
                                 ,
    \\
    \size( \Phi^{\RT,*}_f ) \leq
                               &\, 2\size\left( \left(\left( n_f , 0 , \Id_{\R^d} \right)\right) \right) + 2 \size( \Phi^{\RT,\perp}_f )
    \\
                               &\, + \sum_{j=1}^{d-1} \left( 2 \size\left( \left(\left( t_j , 0 , \Id_{\R^d} \right)\right) \right)
                                 + 2 \size( \Phi^{\RT,t_j}_f ) \right)
                                 \leq \, C d^3
                                     .
  \end{align*}

\end{proof}
%%%%%%%%%%%%%%%%%%%%%%%%%%%%%%%%%%%%%%%%%%%%%%%%%%%%%%%%%%%%%%%%%%
\subsection{\Ned elements $\Ne$}
\label{sec:ne}
%%%%%%%%%%%%%%%%%%%%%%%%%%%%%%%%%%%%%%%%%%%%%%%%%%%%%%%%%%%%%%%%%%
In this section we restrict ourselves to the space dimension
$d \in \{2,3\}$.  For $d = 2$,
we can relate the \Ned basis functions to the Raviart-Thomas basis.
In fact, one can verify that, for an edge $ f \in \calE = \calF $
(which is also a face), the finite element basis
$ \{\theta_f^{\Ne}\}_{f\in \calF} $ for $ \Ne $ satisfies
\begin{equation}
  \label{eq:NeRTtwodim}
  \theta_f^{\Ne} \cdot t_f = \theta_f^{\RT}\cdot n_f, \quad \text{ and } \quad  \theta_f^{\Ne} \cdot n_f = - \theta_f^{\RT}\cdot t_f,
\end{equation}
where $ n_f:=((n_f)_1, (n_f)_2) $ is a unit normal vector to $ f $ as
in \eqref{orthRT} and $ t_f = (-(n_f)_2,(n_f)_1) $ is a unit vector
tangent to $ f $.  Hence, a NN emulation for $ \theta_f^{\Ne} $ can be
derived from Proposition \ref{prop:RT} or Corollary \ref{cor:RTv2}.

We now focus on the case $ d = 3 $.
A basis $ \{\theta^{\Ne}_e\}_{e\in \calE} $ for $ \Ne(\calT,\domain)$
can be constructed by assembling local shape functions of all
simplices $ T_{1}\ldots,T_{s(e)} $, $ s(e)\in \N $ sharing an edge
$ e $.  We fix $ e \in \calE $ and a unit vector $ t_e $ tangent to
$ e $, and denote the midpoint of $ e $ by $ m_e $.  We denote by
$ \tilde{e}(i) $ the only edge of $ T_i $ that does not share a vertex
with $ e $, and let $ t_{\tilde{e}(i)} $ be a unit vector tangent to
$\tilde{e}(i)$, directed in such a way that
$ t_e \cdot [(m_e - m_{\tilde{e}(i)})\times t_{\tilde{e}(i)}]>0 $.

Then,
\begin{equation}
  \label{eq:NEshape}
  \theta_e^{\Ne}(x):=
  \begin{cases}
    \dfrac{(x - m_{\tilde{e}(i)})\times t_{\tilde{e}(i)} }{t_e \cdot
      [(m_e - m_{\tilde{e}(i)})\times t_{\tilde{e}(i)}] }
    & \text{ if } x \in T_i, i = 1,\ldots,s(e), \vspace*{0.5em}\\
    \ 0 & \text{ if } x \notin \overline{\bigcup_{i=1\ldots,s(e)}
      T_i}.
  \end{cases}
\end{equation}
Note that
\begin{equation}
  \label{eq:mathfraks}
  \mathfrak{s}(\calE) := \max_{e \in \calE} s(e)
\end{equation}
is bounded from above by a constant only dependent on the shape
regularity constant $ C_{\rm sh} $ of $ \calT $.  See \cite[Remark
11.5 and Proposition 11.6]{ErnGuermondBookI2021}.

\begin{proposition}\label{prop:NE}
  Given $ e \in \calE $, let
  $ A_{e}^{(i)} \in \R^{3\times 3}, b_{e}^{(i)} \in \R^{3} $ be such
  that for $ i = 1,\ldots,s(e) $
  \begin{align*}
    A_{e}^{(i)} x := \frac{x\times t_{\tilde{e}(i)} }{t_e \cdot [(m_e - m_{\tilde{e}(i)})\times t_{\tilde{e}(i)}] }
    \quad\forall x\in \R^3, \quad
    b_{e}^{(i)} := - \frac{m_{\tilde{e}(i)}\times t_{\tilde{e}(i)} }{t_e \cdot [(m_e - m_{\tilde{e}(i)})\times t_{\tilde{e}(i)}] }
    .
  \end{align*}
  Then
  \begin{equation}
    \label{eq:NEdef}
    \Phi_e^{\Ne} := \sum_{i = 1}^{s(e)} \Phi^{\times}_{3,\kappa} \sconc \Parallelc{ \Phi_{3,2}^{\Id} \sconc
      \left( \left(A_{e}^{(i)},b_{e}^{(i)},\Id_{\R^3}\right )\right ) , \Phi_{T_i}^{\mathbbm{1}}}
  \end{equation}
  satisfies $ \theta_e^{\Ne}(x) = \realizc{\Phi_e^{\Ne}}(x) $ for
  a.e.\ $ x \in \domain $, for any $ \kappa $ such that
  \begin{equation}
    \label{eq:NEkappa}
    \kappa \ge \max_{i=1,\ldots,s(e)}\sup_{x \in T_i} \frac{\norm[\infty]{(x - m_{\tilde{e}(i)})\times t_{\tilde{e}(i)}}}{t_e \cdot [(m_e - m_{\tilde{e}(i)})\times t_{\tilde{e}(i)}] }.
  \end{equation}
  Furthermore, there exists $C>0$ independent of $\calT$ such that for
  all $ e \in \calE $
  \begin{align*}
    \depth(\Phi_e^{\Ne}) = &\, 5 , \quad \size(\Phi_e^{\Ne}) \leq \, C s(e) \leq C \mathfrak{s}(\calE)
                             .
  \end{align*}
\end{proposition}
\begin{proof}
  Firstly, we observe that indeed
  $ \theta_e^{\Ne}(x) = \realizc{\Phi_e^{\Ne}}(x) $ for all
  $ x \in \domain \setminus \partial \calT $.  Secondly, we use Lemma
  \ref{lem:dpwl} with $\mu = d = 3$, $m=d^2+d = 12$ and $s=s(e)$.
\end{proof}
%%%%%%%%%%%%%%%%%%%%%%%%%%%%%%%%%%%%%%%%%%%%%%%%%%%%%%%%%%%%%%%%%%%%%%%%%%%%%%%%%%%%%%%%%
\subsection{CPwL elements $\So$}
\label{sec:cpwlnonconvex}
In this section, we provide a construction based on element-by-element
assembly of the shape functions, similar to that in the previous
sections, using both ReLU and BiSU activations.
A basis $ \{\theta_{p}^{\So}\}_{p\in \calV} $ of $\So(\calT, \domain)$
is uniquely defined by the relations
$ \theta_{p_i}^{\So}(p_j) = \delta_{ij} $, for $ p_i,p_j \in \calV $.
Define $ s(p) := |\{T \in \calT : p \in \overline{T}\}| \in \N $.
Note that
\begin{equation}
  \label{eq:mathfrakn}
  \mathfrak{s}(\calV) := \max_{p \in \calV} s(p)
\end{equation}
is bounded from above by a constant only dependent on $d$ and the
shape regularity constant $ C_{\rm sh} $ of $ \calT $.  See
\cite[Remark 11.5 and Proposition 11.6]{ErnGuermondBookI2021}, which
generalize to space dimension $d>3$.
The following proposition is analogous to Propositions \ref{prop:RT}
and \ref{prop:NE}.

\begin{proposition}
  \label{prop:S1}
  Given $ p \in\calV $, let $ T_1, \ldots T_{s(p)} \in \calT $,
  $s(p)\in \N $ denote the simplices adjacent to $ p $.  Let
  $ A_p^{(i)} \in \R^{1\times d}, b_p^{(i)} \in \R^1 $ for
  $ i = 1,\ldots,s(p) $ be such that
  \begin{align*}
    (A_p^{(i)}, b_p^{(i)})
    \begin{pmatrix}
      p_1 & (a_{i,1})_1  &  & (a_{i,d})_1 \\
      \vdots & & \ddots& \\
      p_d & (a_{i,1})_d  &  & (a_{i,d})_d \\
      1 & 1 & \cdots & 1
    \end{pmatrix} = (1,0,\ldots,0),
  \end{align*}
  where the points $ a_{i,j} \in \R^d $ are such that
  $T_i = \conv({p,a_{i,1},\ldots,a_{i,d}})$.  Then
  \begin{equation}
    \label{eq:S1def}
    \Phi_p^{\So}
    :=
    \sum_{i = 1}^{s(p)}\Phi_{1,1}^{\times} \sconc \Parallelc{\Phi_{1,2}^{\Id}
      \sconc \left (A_p^{(i)}, b_p^{(i)}, \Id_{\R}  \right ),  \Phi_{T_i}^{\mathbbm{1}}}
  \end{equation}
  satisfies
  $ \theta_p^{\So}(x) = \realiz{\Phi_p^{\So}}(x) \, \text{ for a.e.\ }
  x\in\domain $.  Furthermore, there exists $C>0$ independent of
  $\calT$ such that for all $p\in\calV$
  \begin{align*}
    \depth(\Phi_p^{\So}) = &\, 5 , \qquad
                             \size(\Phi_p^{\So}) \leq \, C s(p) d^2 \leq C \mathfrak{s}(\calV) d^2
                             .
  \end{align*}
\end{proposition}

\begin{proof}
  Observing that
  $ \theta_p^{\So}(x) = \realiz{\Phi_p^{\So}}(x) \, \text{ for all } x
  \in \domain \setminus \partial \calT$, we use Lemma \ref{lem:dpwl}
  with $\mu=1$, $m=d+1$ and $s=s(p)$ to estimate the NN size.
\end{proof}
%
%%%%%%%%%%%%%%%%%%%%%%%%%%%%%%%%%%%%%%%%%%%%%%%%%%%%%%%%%%%%%%%%%%
\section{ReLU NN emulation of CPwL shape functions}
\label{sec:cpwlrelu}
%%%%%%%%%%%%%%%%%%%%%%%%%%%%%%%%%%%%%%%%%%%%%%%%%%%%%%%%%%%%%%%%%%
%
For continuous shape functions which vanish on the boundary of their
support, one can construct NN emulations using the ReLU activation
function alone, as shown in \cite[Section 3]{HLXZ2020} for regular,
simplicial meshes \emph{with convex patches}.  
The purpose of this
section is to extend these results to arbitrary regular, simplicial
partitions $\calT$ of polytopal domains $\domain \subset \bbR^d$, in
any space dimension $d\geq 2$, using only ReLU activations,
significantly improving the network size bounds from \cite[Theorem
5.2]{HLXZ2020}.  In the sequel, for a vertex $ p \in \calV $ we write
\begin{equation}\label{def:patch}
  \omega(p):= \bigcup_{i=1\ldots,s(p)} \overline{T}_i,
\end{equation}
where $ T_1, \ldots T_{s(p)} \in \calT $ denote the simplices adjacent
to $p$.  We call $ \omega(p) $ a \emph{patch}.
One key assumption in \cite[Section 3]{HLXZ2020} was that $\omega(p)$
is convex for all vertices $ p \in \calV $.

Removing this assumption is the main topic of the Section
\ref{sec:cpwlnonconvexrelu}.  
We remark that the construction given in
Section \ref{sec:cpwlnonconvex} also does not require convexity of the
patches and, since no minimum is computed, the depth of the network is
independent of the input dimension $ d $ and the maximum number of
elements meeting in one point $\mathfrak{n}(\calV)$. 
In this section we avoid the use of BiSU activations, 
which could be considered not natural
for the emulation of continuous functions in $ \So(\calT,\domain) $.
%%%%%%%%%%%%%%%%%%%%%%%%%%%%%%%%%%%%%%%%%%%%%%%%%%%%%%%%%%%%%%%%%%
\subsection{Regular, simplicial partitions $\calT$ with convex patches}
\label{sec:cpwl}
%%%%%%%%%%%%%%%%%%%%%%%%%%%%%%%%%%%%%%%%%%%%%%%%%%%%%%%%%%%%%%%%%%
%
Under the assumption of convexity of patches, the hat basis functions
$ \{\theta_p^{\So}\}_{p \in \calV} \subset \So(\calT, \domain) $
satisfy \cite[Lemma 3.1]{HLXZ2020}
\begin{equation}
  \label{eq:thetap}
  \theta_p^{\So}(x) = \max\left \{0, \min_{i = 1,\ldots,s(p)} A_p^{(i)} x + b_p^{(i)}  \right \},
\end{equation}
with $ A_p^{(i)} \in \R^{1\times d}, b_p^{(i)} \in \R^1 $ such that
$ A_p^{(i)} x+b_p^{(i)} = \theta_p^{\So} |_{\mathstrut T_{i}}(x) $ for
all $ T_i \subset \omega(p) $, $ i =1, \ldots, s(p) $.

We now recall the emulation of the shape functions from
\cite{HLXZ2020}, and show the dependence of the constants on $d$.  We
remark that the explicit $d$-dependence was not studied in \cite{HLXZ2020}.
\begin{proposition}[{{\cite[Theorem 3.1]{HLXZ2020}}}]
  \label{prop:HLXZ2020basis}
  Consider $p\in\calV$ for which $\omega(p)$ is convex and let
  $ T_1, \ldots T_{s(p)} \in \calT $, $s(p)\in \N $ be the simplices
  adjacent to $ p $.  Let
  $(A_p^{(i)}, b_p^{(i)}) \in \R^{(d+1)\times 1}$, $i=1,\ldots,s(p)$
  be as in \eqref{eq:thetap}.

  Then
  \begin{align*}
    \Phi_p^{CPwL}
    := &\, \Big( \left(  1 ,   0 , \rho \right),
         \left( 1 ,   0 , \Id_{\R} \right) \Big)
                \sconc \Phi^{\min}_{s(p)} \sconc \left( \left(
                \begin{pmatrix} A_{p}^{(1)} \\ \vdots \\ A_{p}^{({s}(p))} \end{pmatrix},
    \begin{pmatrix} b_{p}^{(1)} \\ \vdots \\ b_{p}^{({s}(p))} \end{pmatrix},
    \Id_{\R^{s(p)}} \right) \right)
  \end{align*}
  satisfies $\realiz{\Phi_p^{CPwL} }(x) = \theta^{\So}_p(x)$ for all
  $x\in\domain$ and there exists $C>0$ independent of $d$ and $\calT$
  such that for all $p \in \calV$
  \begin{align*}
    \depth(\Phi_p^{CPwL} ) \leq \, 5 + \log_2({s}(p)),
    \qquad
    \size(\Phi_p^{CPwL} ) \leq \, C{s}(p) d
                                       .
  \end{align*}
  The depth only depends on $\calT$ through $s(p)$.
\end{proposition}
\begin{proof}
  The network depth and size can be bounded as
  \begin{align*}
    \depth(\Phi_p^{CPwL} ) = &\, \depth\left( \Big( \left(  1 ,   0 , \rho \right),
                               \left( 1 ,   0 , \Id_{\R} \right) \Big) \right)
                               + \depth( \Phi^{\min}_{{s}(p)} )
    \\
                             &\, + \depth\left( \left( \left(
                               \begin{pmatrix} A_{p}^{(1)} \\ \vdots \\ A_{p}^{({s}(p))} \end{pmatrix},
    \begin{pmatrix} b_{p}^{(1)} \\ \vdots \\ b_{p}^{({s}(p))} \end{pmatrix},
    \Id_{\R^{{s}(p)}} \right) \right) \right)
    \\
             = &\, 2 + \depth( \Phi^{\min}_{{s}(p)} ) + 1
                 \leq 3 + (2+\log_2({s}(p))  )
                 = 5 +\log_2({s}(p))
                 ,
    \\
    \size(\Phi_p^{CPwL} ) \leq &\, C \size\left( \Big( \left(  1 ,   0 , \rho \right),
                                 \left(  1,   0 , \Id_{\R} \right) \Big) \right)
                                 + C\size( \Phi^{\min}_{s(p)} )
    \\
                             &\, + C\size\left( \left( \left(
                               \begin{pmatrix} A_{p}^{(1)} \\ \vdots \\ A_{p}^{(s(p))} \end{pmatrix},
    \begin{pmatrix} b_{p}^{(1)} \\ \vdots \\ b_{p}^{(s(p))} \end{pmatrix},
    \Id_{\R^{s(p)}} \right) \right) \right)
    \\
    \leq &\, C ( 2 + s(p) + {s}(p) (d+1)) \leq C {s}(p) d
                                       .
  \end{align*}

\end{proof}
The preceding result can be used to construct emulations of shape
functions on non-convex patches which only use the ReLU activation.
%
%%%%%%%%%%%%%%%%%%%%%%%%%%%%%%%%%%%%%%%%%%%%%%%%%%%%%%%%%%%%%%%%%%
\subsection{Regular, simplicial partitions $\calT$ including non-convex patches}
\label{sec:cpwlnonconvexrelu}
%%%%%%%%%%%%%%%%%%%%%%%%%%%%%%%%%%%%%%%%%%%%%%%%%%%%%%%%%%%%%%%%%%
We now extend 
Section~\ref{sec:cpwl} to 
non-convex patches, i.e.\ we show that ReLU NNs can emulate CPwL
functions on arbitrary regular, simplicial meshes in $d\in\N$
dimensions.  To present this result in Theorem \ref{thm:relucpwl}
below, we introduce some notation.

Given $p \in\calV$ , let $ T_1, \ldots T_{s(p)} \in \calT$ denote the
simplices adjacent to $ p $.  For all $j=1,\ldots,s(p)$, let
$a_0 := p$ and $a_1,\ldots,a_d\in\R^d$ be such that
$T_j = \simp{a_0}{a_d}$ and let
$q_j:=p + \delta_j \sum_{i=1}^d (p-a_i)$ for some sufficiently small
$\delta_j>0$.  Then we define
\begin{equation}\label{def:ttilde}
  \tilde{T}_{ij} :=
  \begin{cases}
    \conv(\{q_j,a_0, \dots ,a_d\}\setminus\{a_i\})&\text{if }i\in\{1,\dots,d\},\\
    T_j&\text{if }i=0.
  \end{cases}
\end{equation}
Furthermore, set
\begin{equation}\label{def:omegatilde}
  \tilde{\omega}_j(p) :=  \bigcup_{i=0}^d \overline{\tilde{T}_{ij}}.
\end{equation}
We build basis functions for $ \So(\calT,\domain) $ starting from the
hat functions $\tilde{\theta}_{p,j}^{\So} \in C^0(\domain)$ for
$j=1,\ldots,s(p)$ defined by
\begin{align}
  \label{eq:tilde_basis_functions_def}
  \begin{aligned}
    &\tilde{\theta}_{p,j}^{\So}(p) \, = 1 \text{ and }
    \tilde{\theta}_{p,j}^{\So}(q) = 0 \text{ for all other vertices
      $q$ of }\tilde{\omega}_j(p),
    \\
    &\tilde{\theta}_{p,j}^{\So} |_{\tilde{T}_{ij}} \, \in \bbP_1
    \text{ for all } i=0,\ldots,d,
    \\
    &\tilde{\theta}_{p,j}^{\So}
    |_{\domain\backslash\tilde{\omega}_j(p)} \, = 0.
  \end{aligned}
\end{align}

\begin{figure}
  \centering
  \begin{subfigure}[b]{0.5\textwidth}
    \centering \includegraphics[width=0.8\textwidth]{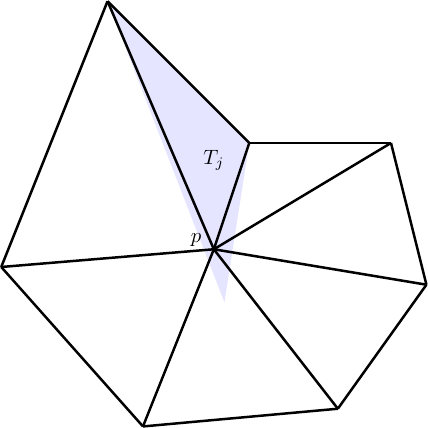}
    \caption{$\omega(p)$}
    \label{fig:omega}
  \end{subfigure}
  \hspace{0.5cm}
  \begin{subfigure}[b]{0.4\textwidth}
    \centering \includegraphics[width=0.4\textwidth]{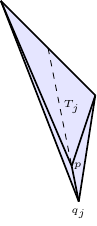}
    \caption{$\tilde\omega_j(p)$}
    \label{fig:omegatilde}
  \end{subfigure}
    \caption{The patches $\omega(p)$ and (shaded)
        $\tilde\omega_j(p)\subset\omega(p)$.}
    \end{figure}

    In Theorem \ref{thm:theorem_min_max_basis_functions} we show that
    CPwL basis functions with nonconvex support $\omega(p)$ can be
    computed as the maximum of $s(p)$ many CPwL basis functions with convex
    support, whose ReLU NN emulation was given in Section
    \ref{sec:cpwl}.  This maximum can be emulated exactly by a ReLU NN
    using the constructions in Section \ref{sec:nndef}, as shown in
    Theorem \ref{thm:relucpwl}.  We obtain the same bound on the ReLU
    NN size as the bound on the NN size in Proposition \ref{prop:S1}.
    The proofs of these results are postponed to Appendix \ref{sec:proofs}.

    \begin{theorem}
      \label{thm:theorem_min_max_basis_functions}
      For all $ p \in\calV $, let $ T_1, \ldots T_{s(p)} \in \calT $,
      $s(p)\in \N $ be the simplices adjacent to $ p $.  Then, for all
      $p\in\calV$ and all $x \in \omega (p)$
      \begin{equation}
        \label{eq:theorem_tilde_basis_fct}
        \theta_p^{\So}(x)
        =\max_{{j = 1,\ldots,s(p)}} \tilde{\theta}_{p,j}^{\So} (x)
        = \max_{j = 1,\ldots,s(p)}
        \max \Big \{ 0,  \min_{i \in  \{0,\dots,d\}} \tilde{A}_p^{(i,j)} x + \tilde{b}_p^{(i,j)}   \Big \} ,
      \end{equation}
      where each
      $x\mapsto \tilde{A}_p^{(i,j)} x + \tilde{b}_p^{(i,j)}$ is a
      globally linear function fulfilling
      $(\tilde{A}_p^{(i,j)} x + \tilde{b}_p^{(i,j)})
      |_{\tilde{T}_{ij}} = \tilde{\theta}_{p,j}^{\So}
      |_{\tilde{T}_{ij}}$.

    \end{theorem}
    \begin{theorem}
      \label{thm:relucpwl}
      For all $p\in\calV$ let $ T_1, \ldots T_{{s}(p)} \in \calT $,
      ${s}(p)\in \N $ be the simplices adjacent to $ p $.  For
      $\tilde{\theta}_{p,j}^{\So}$, $j=1,\ldots,s(p)$ defined in
      \eqref{eq:tilde_basis_functions_def},
      let $\tilde{\Phi}^{CPwL}_{p,j}$, $j=1,\ldots,s(p)$ be the NNs
      from Proposition \ref{prop:HLXZ2020basis} satisfying
      $\realiz{\tilde{\Phi}^{CPwL}_{p,j}} =
      \tilde{\theta}_{p,j}^{\So}$ on $\domain$.

      Then
      \begin{align}
        \Phi_p^{CPwL}
        := \Phi^{\max}_{s(p)} \sconc \Parallel{\tilde{\Phi}^{CPwL}_{p,1},\ldots,\tilde{\Phi}^{CPwL}_{p,s(p)}}
      \end{align}
      satisfies $\realiz{\Phi_p^{CPwL} }(x) = \theta^{\So}_p(x)$ for
      all $x\in\domain$ and
      \begin{align*}
        \depth(\Phi_p^{CPwL} ) \leq \, 7 + \log_2(s(p)) + \log_2(d+1) ,
        \quad
        \size(\Phi_p^{CPwL} ) \leq \, C d^2 s(p)
                                           .
      \end{align*}
    \end{theorem}

%%%%%%%%%%%%%%%%%%%%%%%%%%%%%%%%%%%%%%%%%%%%%%%%%%%%%%%%%%%%%%%%%%
    \section{NN emulation of lowest order conforming FE spaces. Approximation rates.}
    \label{sec:approximation}
%%%%%%%%%%%%%%%%%%%%%%%%%%%%%%%%%%%%%%%%%%%%%%%%%%%%%%%%%%%%%%%%%%

    Having defined explicit constructions of NN emulations of shape
    functions for all finite elements in the discrete de Rham complex
    of the lowest polynomial order \eqref{discderham}, we are now in
    position to formulate and prove our main results: 
    \emph{exact NN emulations}
    of each of the lowest order FE spaces in the de Rham complex,
    on regular, simplicial partitions $\calT$ of polytopal domains
    $\domain \subset \bbR^d$.  For
    $\blacklozenge \in \{ \So, \Ne, \RT, \Sz \}$, we obtain a vector
    space of NNs
    $\calNN(\blacklozenge;\calT,\domain) = \{ \Phi^{\blacklozenge,v} :
    v\in\blacklozenge(\calT,\domain) \}$ such that the realization of
    each NN $\Phi^{\blacklozenge,v}$ equals $v$ a.e. in $\domain$.

    With the networks $\calNN(\blacklozenge;\calT,\domain)$ at hand,
    we may lift known approximation results for finite elements to
    obtain constructive NN approximations of arbitrary functions in
    the Sobolev spaces belonging to the de Rham complex
    \eqref{derham}.

    Accordingly, we first construct NN emulations of the FE spaces in
    Proposition \ref{prop:basisnet}, from which the approximation
    results are derived in Theorem \ref{thm:derhamapx}.  To present
    the next statement, we define
    $ \mathfrak{s}(\calF) := \max_{f\in \calF} s(f) \leq 2 $,
    $ \mathfrak{s}(\calT) := 1$ and $ s(T) := 1 $ for all
    $ T\in \calT $.
    \begin{proposition}
      \label{prop:basisnet}
      Let $\domain\subset\R^d$, $d\geq2$, be a bounded, polytopal
      domain.  For every regular, simplicial triangulation $\calT$ of
      $\domain$ and every $\blacklozenge \in \{ \So, \Ne, \RT, \Sz \}$
      (with the N\'{e}d\'{e}lec space $\blacklozenge = \Ne$ excluded
      if $d > 3$),
      there exists a NN
      $\Phi^\blacklozenge :=\Phi^{\blacklozenge(\calT,\domain)}$ with
      ReLU and BiSU activations, which in parallel emulates the shape
      functions $\{ \theta^\blacklozenge_i \}_{i\in \calI}$ for
      $\calI \in \{ \calV, \calE, \calF, \calT \}$, respectively, that
      is
      $ \realiz{\Phi^\blacklozenge}\colon \domain \to \R^{|\calI|} $
      satisfies
      \begin{equation*}
        \realiz{\Phi^\blacklozenge}(x)_i
        = \, \theta^\blacklozenge_i(x)
        \quad\text{ for a.e. } x\in\domain
        \text{ and all } i\in\calI.
      \end{equation*}

      There exists $C>0$ independent of $d$ and $\calT$ such that
      \begin{align*}
        \depth(\Phi^\blacklozenge)
	= &\, \begin{cases} 5 & \text{ if } \blacklozenge \in \{ \So, \Ne, \RT \},
          \\ 3 &  \text{ if } \blacklozenge = \Sz, \end{cases}
	\\
        \size(\Phi^\blacklozenge) \le
          & C d^2 \sum_{i\in \calI} s(i) \le C d^2\mathfrak{s}(\calI) \dim( \blacklozenge(\calT,\domain)).
      \end{align*}

      For $\blacklozenge \in \{ \So, \Ne, \RT, \Sz \}$ and for every
      FE function
      $v = \sum_{i\in\calI} v_i \theta^\blacklozenge_i
      \in\blacklozenge(\calT,\domain)$, there exists a NN
      $\Phi^{\blacklozenge,v} :=
      \Phi^{\blacklozenge(\calT,\domain),v}$ with ReLU and BiSU
      activations, such that for a constant $C>0$ independent of $d$
      and $\calT$
      \begin{align*}
        \realiz{\Phi^{\blacklozenge,v}}(x)
	= &\, v(x)
            \quad\text{ for a.e. } x\in\domain
            ,
	\\
        \depth(\Phi^{\blacklozenge,v})
	= &\, \begin{cases} 5 & \text{ if } \blacklozenge \in \{ \So, \Ne, \RT \},
          \\ 3 &  \text{ if } \blacklozenge = \Sz, \end{cases}
	\\
        \size(\Phi^{\blacklozenge,v})
        \leq
          &\, C d^2 \sum_{i\in\calI} s(i)
            \leq C d^2 \mathfrak{s}(\calI) \dim( \blacklozenge(\calT,\domain)).
      \end{align*}

      The layer dimensions and the lists of activation functions of
      $\Phi^\blacklozenge$ and $\Phi^{\blacklozenge,v}$ are
      independent of $v$ and only depend on $\calT$ through
      $ \{s(i)\}_{i\in\calI} $ and
      $ \snorm{\calI} = \dim(\blacklozenge(\calT,\domain)) $.

      For each $\blacklozenge \in \{ \So, \Ne, \RT, \Sz \}$, the set
      \begin{equation}\label{eq:NNFEMDef}
        \calNN(\blacklozenge;\calT,\domain)
	:= \{ \Phi^{\blacklozenge,v} : v \in \blacklozenge(\calT,\domain) \}\;,
      \end{equation}
      together with the linear operation
      \begin{align}
        \label{eq:NNFEMlin}
        \Phi^{\blacklozenge,v} \widehat{+} \lambda\Phi^{\blacklozenge,w}
	:= &\, \Phi^{\blacklozenge,v+\lambda w},
             \qquad
             \text{ for all } v,w\in \blacklozenge(\calT,\domain) \text{ and }\lambda\in\R
      \end{align}
      is a vector space, and the map
      $\realiz{\cdot}: \calNN(\blacklozenge;\calT,\domain) \to
      \blacklozenge(\calT,\domain)$ 
      is a linear isomorphism.
    \end{proposition}
    \begin{remark}
      Note that $\sum_{i \in \calI} s(i) \leq c(\calI,d) \snorm\calT$,
      where $c(\calV,d) = d+1$ is the number of vertices of a
      $d$-simplex, $c(\calE,d)$ the number of edges of a $d$-simplex,
      $c(\calF,d)$ the number of faces of a $d$-simplex and
      $C(\calT,d) = 1$.  We obtain this inequality by observing that
      each element $T \in \calT$ contributes $+1$ to $c(\calI,d)$
      terms $s(i)$.  Therefore, we also have the bound
      $ \size(\Phi^{\blacklozenge}) \leq Cd^2 c(\calI,d)
      \snorm\calT$, independent of the shape regularity constant
      $ C_{\rm sh} $ of $ \calT $.
      The same bound holds for $\size(\Phi^{\blacklozenge,v})$.
    \end{remark}

    \begin{definition}
      \label{def:basisnet}
      For a given polytopal domain $ \domain \subset \R^d $, $ d\ge2 $ and a
      regular, simplicial triangulation $ \calT $ on $ \domain $, we
      call the network $ \Phi^{\blacklozenge} $ defined in Proposition
      \ref{prop:basisnet} a \emph{$\blacklozenge$-basis net}.
    \end{definition}

\begin{proofof}{Proposition \ref{prop:basisnet}}
  We define $\Phi^{\blacklozenge(\calT,\domain)}$ as the
  parallelization of networks from Propositions \ref{prop:S1},
  \ref{prop:NE}, \ref{prop:RT} or Equation \eqref{dnnS0}, namely
  $\Phi^{\blacklozenge(\calT,\domain)} := \Parallel{ \{
    \Phi^\blacklozenge_i \}_{i\in \calI} }$, from which the formula
  for the realization, the formula for the NN depth and the bound on
  the NN size of $\Phi^{\blacklozenge(\calT,\domain)}$ directly follow
  with Proposition \ref{prop:parallel}.

  The NN $\Phi^{\blacklozenge(\calT,\domain),v}$ is defined as the sum
  $\Phi^{\blacklozenge(\calT,\domain),v} := \sum_{i\in\calI}
  v_i\Phi^\blacklozenge_i$, where the sum of NNs is as defined in
  Proposition \ref{prop:sum}, and where the NNs
  $v_i\Phi^\blacklozenge_i$ are obtained from those in Propositions
  \ref{prop:S1}, \ref{prop:NE}, \ref{prop:RT} and Equation
  \eqref{dnnS0} by scaling all weights and biases in the last layer by
  $v_i$.  The formula for the realization, the formula for the depth
  and the bound on the NN size follow with Proposition \ref{prop:sum}.

  By comparing the definition of the parallelization in Proposition
  \ref{prop:parallel} and the sum in Proposition \ref{prop:sum}, we
  observe that their hidden layers are equal.  Therefore, the hidden
  layers of $\Phi^{\blacklozenge(\calT,\domain),v}$ and of
  $\Phi^{\blacklozenge(\calT,\domain)}$ coincide.

  By definition of $\Phi^{\blacklozenge(\calT,\domain),v}$ as linear
  combination of the basis NNs
  $\{ \Phi^\blacklozenge_i \}_{i\in \calI}$, which are the same for
  all $v$, the NN $\Phi^{\blacklozenge(\calT,\domain),v}$ is
  determined uniquely by the coefficients $\{ v_i \}_{i\in \calI}$.
  Therefore,
  $\realiz{\cdot}: \calNN(\blacklozenge;\calT,\domain) \to
  \blacklozenge(\calT,\domain)$ is a bijection.  With the linear
  operations defined in \eqref{eq:NNFEMlin}, this map is linear by
  definition, thus a linear isomorphism.
\end{proofof}

\begin{remark}
  \label{rem:basisnet}
  For all
  $v = \sum_{i\in\calI} v_i \theta^\blacklozenge_i
  \in\blacklozenge(\calT,\domain)$, for
  $\bsv = (v_i)_{i\in\calI} \in\R^{\snorm{\calI}}$, the network
  $\Phi^{\blacklozenge,v}$ can be obtained from $\Phi^\blacklozenge$
  as follows.
  Denoting the last layer weight matrix and bias vector of
  $\Phi^\blacklozenge_i$ by $A^{(i)}$ and $b^{(i)}$, those of
  $\Phi^\blacklozenge$ are given by
  $A = \operatorname{diag}( A^{(i_1)}, \ldots, A^{(i_{\snorm{\calI}})}
  )$ and
  $b = ( ( b^{(i_1)} )^\top, \ldots, ( b^{(i_{\snorm{\calI}})} )^\top
  )^\top$, and those of $\Phi^{\blacklozenge,v}$ are given by
  $ ( v_{i_1} A^{(i_1)}, \ldots, v_{i_{\snorm{\calI}}}
  A^{(i_{\snorm{\calI}})} )$ and $\sum_{i\in\calI} v_i b^{(i)}$ for an
  enumeration $i_1, \ldots i_{\snorm{\calI}}$ of $\calI$.

  Note that the sum defined in \eqref{eq:NNFEMlin} differs from the
  sum of neural networks from Proposition \ref{prop:sum}.  In
  \eqref{eq:NNFEMlin}, the hidden layers of
  $\Phi^{\blacklozenge,v+\lambda w}$ are independent of $v,w$ and
  $\lambda$ and depend only on $\blacklozenge(\calT,\domain)$.  These
  hidden layers coincide with those of $\Phi^{\blacklozenge}$, which
  emulates a basis of $\blacklozenge(\calT,\domain)$.

  For all $v\in\blacklozenge(\calT,\domain)$ there exists a unique NN
  $\Phi^{\blacklozenge,v} \in \calNN(\blacklozenge;\calT,\domain)$
  which realizes $v$.  However, there exist many other NNs, not in
  $\calNN(\blacklozenge;\calT,\domain)$, with the same realization.
\end{remark}

We apply the previous results to quasi-uniform, shape-regular families
of meshes $ \{\calT_h\}_{h > 0} $ in dimension $ d = 2,3$.  For
$ V = H^1(\domain)$, $H^0(\curl,\domain)$ for $d=3$,
$H^0(\divv,\domain)$ or $ L^2(\domain) $, define the template for the
respective smoothness space $ V^{\bullet} \subset V $ as follows
\begin{equation}
  \begin{split} \label{smoothsp}
    V = H^1(\domain) &\longleftrightarrow V^{\bullet} = H^2(\domain), \\
    \text{for } d=3: \qquad V = H^0(\curl,\domain)
    &\longleftrightarrow V^{\bullet} = H^1(\curl,\domain) := \set{v
      \in [H^{1}(\domain)]^d }{\curl v \in [H^{1}(\domain)]^d },
    \\ 
    V = H^0(\divv,\domain) &\longleftrightarrow V^{\bullet} = H^1(\divv,\domain) := \set{v \in [H^1(\domain)]^d}{\divv v \in H^1(\domain)}, \\
    V = L^2(\domain) &\longleftrightarrow V^{\bullet} = H^1(\domain).
  \end{split}
\end{equation}
We arrive at the following result.
\begin{theorem} \label{thm:derhamapx} 
  Given a bounded, contractible polytopal Lipschitz domain
  $ \domain \subset \R^{d}, d = 2,3 $, assume that
  $(V,\blacklozenge) \in \{ (H^1(\domain), \So)$,
  $ (H^0(\curl,\domain),\Ne)$, $( H^0(\divv,\domain),\RT ) $,
  $ (L^2(\domain), \Sz) \}$, that the regularity space
  $ V^{\bullet} \subset V $ is as in \eqref{smoothsp} and that $d=3$
  if $V = H^0(\curl,\domain)$.

  Assume given a family $\{ \calT_h \}_{h>0}$ of regular, simplicial
  partitions of the polytopal domain $ \domain $ which are uniformly
  shape-regular and quasi-uniform with respect to the mesh-size
  parameter $h$.

  Then there exists a constant $ C > 0$ (depending only on the shape
  regularity parameter $ C_{\rm sh} $ of the family
  $\{ \calT_h \}_{h}$ and on $ d $) 
  such that for all $ h > 0 $ and
  for every $v\in V^\bullet$, 
  there exists
  $\Phi_h \in \calNN(\blacklozenge;\calT_h,\domain)$ such that
  \begin{equation*}
    \norm[V]{v - \realiz{\Phi_h}}\le \, C h \norm[V^{\bullet}]{v}
  \end{equation*}
  and
  \begin{equation*}
    L(\Phi_h)
    = \,
    \begin{cases} 5 & \text{ if } \blacklozenge \in \{ \So, \Ne, \RT \},
      \\ 3 &  \text{ if } \blacklozenge = \Sz, \end{cases}
    \qquad
    M(\Phi_h) \le C h^{-d}.
  \end{equation*}
\end{theorem}
\begin{proof}
  Let $V_{h} = \blacklozenge(\calT_h,\domain)$ denote the lowest order
  FE space corresponding to $V$.  
  By Proposition \ref{prop:basisnet},
  for all $v_h\in V_h$ there exists a NN
  $\Phi_h := \Phi^{\blacklozenge,v_h} \in
  \calNN(\blacklozenge;\calT_h,\domain)$ such that
  $\realiz{\Phi_h}(x) = v_h(x)$ for a.e. $x\in\domain$.  In
  particular, for all $v_h\in V_h$ and all $v\in V^{\bullet}$
  \begin{align*}
    \norm[V]{v - \realiz{\Phi_h}} = \norm[V]{v - v_{h}}.
  \end{align*}
  We can then apply the approximation results e.g.  \cite[Theorem
  11.13]{ErnGuermondBookI2021} for $ V = H^1(\domain) $,
  \cite[Equations (5.7) and (5.8)]{AlonsoValli1999} for
  $ V = H^0(\curl,\domain) $ in case $d=3$, \cite[Theorem
  16.4]{ErnGuermondBookI2021} for $ V = H^0(\divv,\domain) $ and
  Poincar\'e's inequality for $ V = L^2(\domain) $.  More precisely,
  for a constant $ C $ only dependent on $ C_{\mathrm{sh}} $ and on
  $ d $,
  for all $v\in V^\bullet$ there exists $v_h\in V_h$ for which
  \begin{align}\label{eq:apxproperty}
    \norm[V]{v - v_{h}}
    \le Ch\norm[V^{\bullet}]{v}.
  \end{align}

  The formula for $ L(\Phi_h)$ follows from Proposition
  \ref{prop:basisnet}.  In addition, the bound on the NN size follows
  from Proposition \ref{prop:basisnet}, together with
  $ \dim(V_h)\sim h^{-d} $ as $h\downarrow 0$
  and the fact that $\mathfrak{s}(\calI)$,
  $ \calI\in \{\calV,\calE,\calF,\calT\} $ is bounded from above by a
  constant depending only on $ C_{\mathrm{sh}} $.
\end{proof}

\begin{remark}
  \label{remk:FEAprx}
  In \eqref{eq:apxproperty} in the proof of the theorem, the choice of
  $v_h$ (depending on $v$, given in the cited references) is made to
  have the approximation property \eqref{eq:apxproperty}.  However,
  other choices of $v_h \in V_h$ based on interpolation or
  quasi-interpolation can equally be emulated with NNs.  In
  \cite[Corollary 5.3]{ErnGuerFEQuasI}, the authors give a particular
  definition of quasi-interpolants in
  $V_{h} = \blacklozenge(\calT_h,\domain)$
  for $\blacklozenge \in \{ \So, \Ne, \RT \}$,
  requiring minimal regularity of the
  function $v$.  This gives existence of a constant $ C > 0 $ that is
  independent of $ v, h $ such that for all
  $ v \in [W^{r,p}(\domain)]^{d_L} $ there exists a
  $ \Phi_h \in \calNN(\blacklozenge;\calT_h,\domain) $ satisfying, for
  any $ p \in [1,\infty] $, $ r \in \{0,1\} $ or any
  $ p \in [1,\infty), r \in (0,1) $
  \begin{equation}
    \norm[{[L^p(\domain)]^{d_L}}]{v - \realiz{\Phi_h} }
    \le C h^r \snorm[{[W^{r,p}(\domain)]^{d_L}}]{v} \;.
  \end{equation}
  Here $ d_L = d $ if $ V_{h} = \RT(\calT_h, \domain) $ or
  $ V_{h} = \Ne(\calT_h, \domain) $ and $ d_L=1 $ otherwise.  See
  e.g. \cite[Section 2.2]{ErnGuermondBookI2021} for a definition of
  the Sobolev space $W^{r,p}(\domain)$ in which this result is stated.
\end{remark}
The following analogue of Proposition \ref{prop:basisnet} for ReLU
emulation of $\So$ also holds.

\begin{proposition}
  \label{prop:CPLbasisnet}
  Let $\domain\subset\R^d$, $d\geq2$, be a bounded, polytopal domain.
  For every regular, simplicial triangulation $\calT$ of $\domain$,
  there exists a NN $\Phi^{CPwL} :=\Phi^{CPwL(\calT,\domain)}$ with
  only ReLU activations, which in parallel emulates the shape
  functions $\{ \theta^{\So}_i \}_{i\in \calI}$ for $\calI = \calV$.
  That is, $ \realiz{\Phi^{CPwL}}\colon \domain \to \R^{|\calI|} $
  satisfies
  \begin{equation*}
    \realiz{\Phi^{CPwL}}(x)_i
    = \, \theta^{\So}_i(x)
    \quad\text{ for all } x\in\domain
    \text{ and all } i\in\calI.
  \end{equation*}
  There exists $C>0$ independent of $d$ and $\calT$ such that
  \begin{align*}
    \depth(\Phi^{CPwL})
    \leq &\, 8 + \log_2( \mathfrak{s}(\calI) ) + \log_2(d+1)
           ,
    \\
    \size(\Phi^{CPwL})
    \leq & C \snorm{\calI} \log_2( \mathfrak{s}(\calI) ) + C d^2 \sum_{i\in\calI} s(i)
           \leq C d^2 \mathfrak{s}(\calI) \dim( \So(\calT,\domain)).
  \end{align*}

  For all
  $v = \sum_{i\in\calI} v_i \theta^{\So}_i \in\So(\calT,\domain)$,
  there exists a NN $\Phi^{CPwL,v} := \Phi^{CPwL(\calT,\domain),v} $
  with only ReLU activations, such that for a constant $C>0$
  independent of $d$ and $\calT$
  \begin{align*}
    \realiz{\Phi^{CPwL,v}}(x)
    = &\, v(x)
	\quad\text{ for all } x\in\domain
	,
    \\
    \depth(\Phi^{CPwL,v})
    \leq &\, 8 + \log_2( \mathfrak{s}(\calI) ) + \log_2(d+1)
           ,
    \\
    \size(\Phi^{CPwL,v})
    \leq &\, C \snorm{\calI} \log_2( \mathfrak{s}(\calI) ) + C d^2 \sum_{i\in\calI} s(i)
           \leq C d^2 \mathfrak{s}(\calI) \dim( \So(\calT,\domain)).
  \end{align*}

  The layer dimensions and the lists of activation functions of
  $\Phi^{CPwL}$ and $\Phi^{CPwL,v}$ are independent of $v$ and only
  depend on $\calT$ through $ \{s(i)\}_{i\in\calI} $ and
  $ \snorm{\calI} = \dim(\So(\calT,\domain)) $.

  The set
  $\calNN(CPwL;\calT,\domain) := \{ \Phi^{CPwL,v} : v \in
  \So(\calT,\domain) \}$ together with the linear operation
  $\Phi^{CPwL,v} \widehat{+} \lambda\Phi^{CPwL,w} := \Phi^{CPwL,v+\lambda w}$
  for all $v,w\in \So(\calT,\domain)$ and all $\lambda\in\R$ is a
  vector space.

  The realization map
  $\realiz{\cdot}: \calNN(CPwL;\calT,\domain) \to \So(\calT,\domain)$
  is a linear isomorphism.
\end{proposition}

\begin{proof}
  We define $\Phi^{CPwL(\calT,\domain)}$ as the parallelization of
  networks from Theorem \ref{thm:relucpwl}, namely
  $\Phi^{CPwL(\calT,\domain)} := \Parallel{ \{ \Phi^{\Id}_{1,L_i}
    \sconc \Phi^{CPwL}_i \}_{i\in \calI} }$ for
  $L_i = 1 + \max_{j\in\calI} \depth( \Phi^{CPwL}_j ) - \depth(
  \Phi^{CPwL}_i )$, such that all components of the parallelization
  have equal depth.  For the depth and size of the components, we
  obtain with Theorem \ref{thm:relucpwl}
  \begin{align*}
    \depth( \Phi^{\Id}_{1,L_i} \sconc \Phi^{CPwL}_i )
    \leq &\, 1 + ( 7 + \log_2( \mathfrak{s}(\calI) ) + \log_2(d+1) )
           ,
    \\
    \size( \Phi^{\Id}_{1,L_i} \sconc \Phi^{CPwL}_i )
    \leq &\, C \size( \Phi^{\Id}_{1,L_i} ) + C \size( \Phi^{CPwL}_i )
           \leq C ( 8 + \log_2( \mathfrak{s}(\calI) ) + \log_2(d+1) )
           + C d^2 s(i)
           ,
  \end{align*}
  from which the stated results follow with Proposition
  \ref{prop:parallel} by the same arguments as in the proof of
  Proposition \ref{prop:basisnet}.

  The NN $\Phi^{CPwL(\calT,\domain),v}$ is defined as the sum
  $\Phi^{CPwL(\calT,\domain),v} := \sum_{i\in\calI}
  v_i\Phi^{\Id}_{1,L_i} \sconc \Phi^{CPwL}_i$.
  The results now follow from Proposition \ref{prop:sum} as in the
  proof of Proposition \ref{prop:basisnet}.
\end{proof}

Definition \ref{def:basisnet} and Remark \ref{rem:basisnet} apply,
with $CPwL$ instead of $\blacklozenge$, $\So(\calT,\domain)$ instead
of $\blacklozenge(\calT,\domain)$ and
$\Phi^{\Id}_{1,L_i} \sconc \Phi^{CPwL}_i$ instead of
$\Phi^\blacklozenge_i$.  In addition, a result analogous to Theorem
\ref{thm:derhamapx} follows from Proposition \ref{prop:CPLbasisnet},
with the formula for the depth replaced by $\depth(\Phi_h) \leq C$ for
a constant $ C > 0 $ only dependent on $ C_{\mathrm{sh}} $ and $ d $.
%%%%%%%%%%%%%%%%%%%%%%%%%%%%%%%%%%%%%%%%%%%%%%%%%%%%%%%%%%%%%%%%%%
\section{Neural emulation of trace spaces}
\label{sec:traces}
%%%%%%%%%%%%%%%%%%%%%%%%%%%%%%%%%%%%%%%%%%%%%%%%%%%%%%%%%%%%%%%%%%
%
In the previous sections, we have developed ReLU NN emulations of the
lowest order, de Rham compatible Finite Elements on cellular complexes
in the bounded Lipschitz polyhedral domains $\domain\subset\R^3$.  In
certain applications, however, corresponding boundary complexes are
required; we mention only variational boundary integral equations
which arise in computational electromagnetism (e.g.\
\cite{BHvS03_22,BDKSVW2020} and the references there).  We approximate
traces on the boundary $\Gamma = \partial\domain$, which is a finite
union of plane sides, with the network constructions developed
in Section \ref{sec:approximation} for $d=2$.  As has been emphasized
e.g.\ in \cite{BDKSVW2020}, trace spaces of the spaces occurring in
the de Rham complex satisfy exact sequence properties derived from the
compatibility of the corresponding sequences in $\domain$. We refer
to \cite{BHvS03_22,BDKSVW2020} and the references there for a
definition and basic properties of these spaces.
We recall the trace operators (e.g.\ from \cite[Definition 2.1]{BDKSVW2020}):
\begin{subequations}
  \label{eq:tracedef}
  \begin{align}
    \label{eq:tracedir}
    & \gamma_0: H^1(\domain) \to H^{1/2}(\Gamma) :
    & \gamma_0(u)(x_0) = & \lim_{x\to x_0} u(x),
    \\
    \label{eq:tracedirtan}
    & \breve\gamma_0: H^0(\curl,\domain) \to H^{-1/2}(\curl_\Gamma,\Gamma) :
    & \breve\gamma_0(u)(x_0) =  & \lim_{x\to x_0} u(x) - (u(x)\cdot n_{x_0}) n_{x_0},
    \\
    \label{eq:tracetan}
    & \gamma_t: H^0(\curl,\domain) \to H^{-1/2}_\times(\divv_\Gamma,\Gamma) :
    & \gamma_t(u)(x_0) = & \lim_{x\to x_0} u(x) \times n_{x_0},
    \\
    \label{eq:tracenor}
    & \gamma_n: H^0(\divv,\domain) \to H^{-1/2}(\Gamma) :
    & \gamma_n(u)(x_0) = & \lim_{x\to x_0} u(x) \cdot n_{x_0},
  \end{align}
\end{subequations}
for almost all $x_0\in\Gamma$, where we use $x$ to denote points in
$\domain$, and where $n_{x_0}$ denotes the outward unit normal to
$\Gamma$ in $x_0$.  These trace operators render the diagram in Figure
\ref{fig:boundaryderham} commutative (e.g.\ \cite[Figure
2]{BDKSVW2020}).  The trace operators in \eqref{eq:tracedef} are
surjective (e.g. \cite[Theorem 1]{BDKSVW2020}), thus the fractional
Sobolev spaces on $\Gamma$ in \eqref{eq:tracedef} comprise precisely
all traces of elements of the respective function spaces on $\domain$.
In addition, the trace operators in \eqref{eq:tracedef} are continuous
with respect to the norms defined in \cite[Section 2]{BDKSVW2020}, see
\cite[Theorem 1]{BDKSVW2020}.

\begin{figure} [H]
  \[ \xymatrix@C=2cm{
      H^1(\domain)\ar^{\gamma_0}[dd]\ar^-{\operatorname{grad}}[r] &
      H^0(\curl,\domain) \ar^{\breve\gamma_0}[d] \ar^-{\curl}[r]
      \ar_<<<<<<<<{\gamma_t}@/_1.5cm/[ddd]
      & H^0(\divv,\domain) \ar^{\gamma_n}[dd] \\
      & H^{-1/2}(\curl_{\Gamma},\Gamma) \ar^{\cdot\times n}[dd] \ar^-{\curl_{\Gamma}}[dr] & \\
      H^{1/2}(\Gamma) \ar^-{\operatorname{grad}_{\Gamma}}[ur]
      \ar^-{\curl_{\Gamma}}[dr]
      && H^{-1/2}(\Gamma) \\
      & H^{-1/2}_\times(\divv_\Gamma,\Gamma) \ar^-{\divv_\Gamma}[ur] &
    } \]
  \caption{Boundary complex \label{fig:boundaryderham} }
\end{figure}
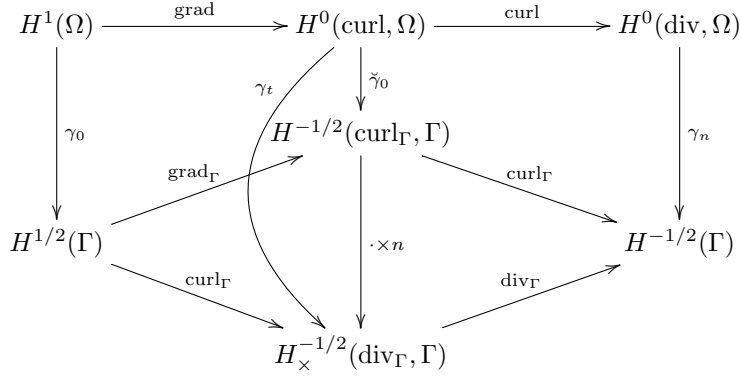

Given a regular simplicial partition $ \calT $ of $\domain$, for each
face $f$ of $\domain$, the set
$\calT_f = \set{ \interior( f \cap \overline{T} ) }{ T\in\calT }$ is a
regular, simplicial triangulation of $f$ (where the interior
$\interior(\ldots)$ is defined with respect to the subspace topology
on the face $f$).  Discretizations of the trace spaces can be defined
as the traces in the sense of \eqref{eq:tracedef} of the finite
element spaces on $\domain$ (see \cite[Section 1.6]{FKDN2015}).
The corresponding diagram for the lowest order conforming FEM spaces
also commutes (Figure \ref{fig:discboundaryderham}).
\begin{figure} [H]
  \[ \xymatrix@C=2cm{
      \So(\calT,\domain)\ar^{\gamma_0|_f}[dd]\ar^-{\operatorname{grad}}[r]
      & \Ne(\calT,\domain) \ar^{\breve\gamma_0|_f}[d] \ar^-{\curl}[r]
      \ar_<<<<<<<<{\gamma_t|_f}@/_1.5cm/[ddd]
      & \RT(\calT,\domain) \ar^{\gamma_n|_f}[dd] \\
      & \Ne(\calT_{f},f) \ar^{\cdot\times n}[dd] \ar^-{\curl_{\Gamma}}[dr] & \\
      \So(\calT_f,f) \ar^-{\operatorname{grad}_{\Gamma}}[ur]
      \ar^-{\curl_{\Gamma}}[dr]
      && \Sz(\calT_f,f) \\
      & \RT(\calT_f,f) \ar^-{\divv_\Gamma}[ur] & } \]
  \caption{Discrete boundary complex \label{fig:discboundaryderham} }
\end{figure}
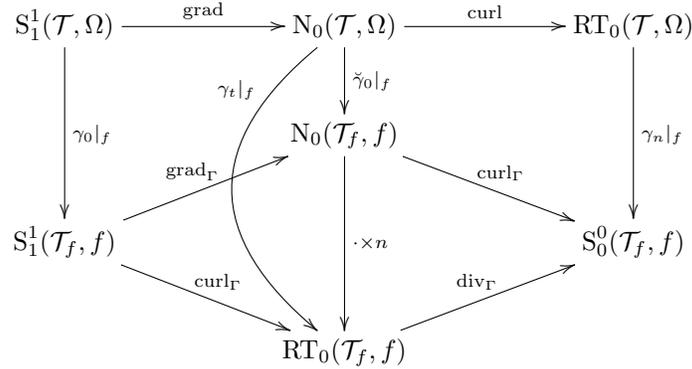

Upon parametrizing each face of $\domain$ by a polygon in $\R^2$, we
can construct NN approximations of the traces on $f$.  We parametrize
each face $f$ by an affine bijection $F_f: D_f \to f$ for some polygon
$D_f\subset\R^2$, which can be partitioned by
$\calT_{D_f} := \set{ F_f^{-1}(T) }{ T \in\calT_f }$.  Functions in
$\So(\calT_f, f)$, $\RT(\calT_f, f)$ and $\Sz(\calT_f, f)$ can be
pulled back to $D_f$.
In particular,
\begin{align*}
  \set{ u\circ F_f }{ u \in \So(\calT_f, f) } = \, \So(\calT_{D_f}, D_f),
  \quad
  \set{ u\circ F_f }{ u \in \Sz(\calT_f, f) } = \, \Sz(\calT_{D_f}, D_f).
\end{align*}
NN emulations of these spaces have already been provided in
Propositions \ref{prop:basisnet} and \ref{prop:CPLbasisnet}.
The spaces $\RT(\calT_f, f)$ and $\RT(\calT_{D_f}, D_f)$ are related
by the Piola transform.  For $J$ denoting the Jacobian of $F_f$,
\begin{align*}
  \set{ \det(J) J^{-1} (u\circ F_f) }{ u \in \RT(\calT_f, f) } = \RT(\calT_{D_f}, D_f).
\end{align*}
Thus, for $u\in\RT(\calT_f, f)$, a network that emulates
$u \circ F_f : D_f \to\R^3$ is given by $\det(J^{-1}) J \Phi $, for a
NN $\Phi \in \calNN(\RT;\calT_{D_f},D_f)$ from Proposition
\ref{prop:basisnet} emulating
$\det(J) J^{-1} (u\circ F_f) \in \RT(\calT_{D_f}, D_f)$.  Here, ReLU
activations imply that the affine transformation $\det(J^{-1}) J$ can
be emulated exactly either by applying this transformation to the
weights and biases of the output layer of $\Phi$, or by concatenating
$\Phi$ with a ReLU NN of depth one.  In both cases, the network size
is increased by at most $Cd^2$ (with $ C>0 $ independent of $ d $ and
$ \calT $),
and the network depth is increased by $0$ respectively $1$.

The shape functions of $\Ne(\calT_f, f)$ equal those of
$\RT(\calT_f, f)$ up to a rotation.  As explained in Section
\ref{sec:ne}, we can use results from Section \ref{sec:rt} for the NN
emulation of the $\Ne(\calT_{D_f}, D_f)$ shape functions.  Therefore,
for $u\in\Ne(\calT_f, f)$, a network that emulates
$u \circ F_f : D_f \to\R^3$ is given by
$\det(J^{-1}) J \Phi $, for a NN
$\Phi \in \calNN(\Ne;\calT_{D_f},D_f)$ from Proposition
\ref{prop:basisnet} emulating
$\det(J) J^{-1} (u\circ F_f) \in \Ne(\calT_{D_f}, D_f)$.

The preceding discussion in this section can be summarized as follows:
\begin{proposition} \label{prop:neuraltrace} Assume given a bounded
  polytopal domain $\domain \subset \R^3$ with boundary
  $\Gamma = \partial\domain$ consisting of a finite union of plane,
  polygonal faces $f$.  For a regular, simplicial partition $\calT$ of
  $\domain$, and for a face $f\subset\Gamma$ of $\domain$, consider
  the regular, simplicial partition
  $\calT_f = \set{ \interior( f \cap \overline{T} ) }{ T\in\calT }$ of
  $f$ with edges $\calE_f = \{ e : e \subset \overline{f} \}$ and
  vertices $\calV_f = \{ v : v \in \overline{f} \}$ (i.e., obtained as
  ``trace'' of $\calT$ on $f\subset \Gamma$).  Let $F_f: D_f \to f$ be
  a bijective affine parametrization of $f$ for some polygonal
  parameter domain $D_f\subset\R^2$ partitioned by
  $\calT_{D_f} := \set{ F_f^{-1}(T) }{ T \in\calT_f }$.  In the
  following, $C$ only depends on the shape regularity constant of the
  simplicial partition $\calT_f$.  Then we have the following.
  \begin{enumerate}
  \item[{(i)}] For all $\blacklozenge\in\{ \So, \Ne, \RT, \Sz \}$
    there exists a NN
    $\Phi^{\blacklozenge,F_f} := \Phi^{\blacklozenge(\calT_f,f),F_f}$
    with ReLU and BiSU activations, which in parallel emulates
    $\{ \theta^\blacklozenge_i \circ F_f \}_{i\in\calI}$ for
    $\calI \in \{ \calV_f, \calE_f, \calE_f, \calT_f \}$, i.e.
    $ \realiz{\Phi^{\blacklozenge,F_f}}\colon D_f \to \R^{|\calI|} $
    satisfies
    \begin{equation*}
      \realiz{\Phi^{\blacklozenge,F_f}}(x)_i
      = \, \theta^\blacklozenge_i \circ F_f(x)
      \quad\text{ for a.e. } x\in D_f
      \text{ and all } i\in\calI.
    \end{equation*}
  \item[{(ii)}] There exists $C>0$ independent of $\calT$ such that
    \begin{align*}
      \depth(\Phi^{\blacklozenge,F_f})
      = &\, \begin{cases} 5 & \text{ if } \blacklozenge \in \{ \So, \Ne, \RT \},
        \\ 3 &  \text{ if } \blacklozenge = \Sz, \end{cases}
               \qquad
               \size(\Phi^{\blacklozenge,F_f}) \le
               C \dim( \blacklozenge(\calT_f,f) ).
    \end{align*}
  \item[{(iii)}] For all $v\in \blacklozenge(\calT_f,f)$, there exists
    a DNN
    $\Phi^{\blacklozenge,v,F_f} :=
    \Phi^{\blacklozenge(\calT_f,f),v,F_f}$ with BiSU and ReLU
    activations, satisfying the same depth and size bounds as
    $\Phi^{\blacklozenge,F_f}$, such that
    $\realiz{ \Phi^{\blacklozenge,v,F_f} } = v \circ F_f$ a.e.\ in
    $D_f$.
    The set
    $\calNN(\blacklozenge;\calT_f,f;F_f) := \{
    \Phi^{\blacklozenge,v,F_f} : v\in \blacklozenge(\calT_f,f) \}$
    together with the linear operation
    $\Phi^{\blacklozenge,v,F_f} \widehat{+} \lambda\Phi^{\blacklozenge,w,F_f} :=
    \Phi^{\blacklozenge,v+\lambda w,F_f}$ for all $v,w\in \blacklozenge(\calT_f,f)$
    and all $\lambda\in\R$ is a vector space.

  \item[{(iv)}] There also exists a DNN
    $\Phi^{CPwL,F_f} := \Phi^{CPwL(\calT_f,f),F_f}$ of depth $C$ and
    size at most $C\dim(\So(\calT_f,f))$, with only ReLU activations,
    such that
    \begin{equation*}
      \realiz{\Phi^{CPwL,F_f}}(x)_i
      = \, \theta^{\So}_i \circ F_f(x)
      \quad\text{ for all } x\in \overline{D_f}
      \text{ and all } i\in\calI.
    \end{equation*}
  \item[(v)] For every $v\in \So(\calT_f,f)$ there exists a DNN
    $\Phi^{CPwL,v,F_f} := \Phi^{CPwL(\calT_f,f),v,F_f}$ with only ReLU
    activations, which satisfies the same depth and size bounds as
    $\Phi^{CPwL,F_f}$, and $\realiz{\Phi^{CPwL,v,F_f}} = v \circ F_f$
    everywhere in $\overline{D_f}$.  The set
    $\calNN(CPwL;\calT_f,f;F_f) := \{ \Phi^{CPwL,v,F_f} : v \in
    \So(\calT_f,f) \}$ together with the linear operation
$$ 
\Phi^{CPwL,v,F_f} \widehat{+} \lambda\Phi^{CPwL,w,F_f} := \Phi^{CPwL,v+\lambda w,F_f} 
\quad \mbox{for all}\;\; v,w\in \So(\calT_f,f) \;\;\mbox{and all}\;\; \lambda\in\R 
$$
is a vector space.
  \end{enumerate}

\end{proposition}

%%%%%%%%%%%%%%%%%%%%%%%%%%%%%%%%%%%%%%%%%%%%%%%%%%%%%%%%%%%%%%%%%%
\section{Extensions and conclusions}
\label{sec:ExtConcl}
%%%%%%%%%%%%%%%%%%%%%%%%%%%%%%%%%%%%%%%%%%%%%%%%%%%%%%%%%%%%%%%%%%
We conclude this paper by indicating some extensions of the main
results, as well as further possible directions of research.

%%%%%%%%%%%%%%%%%%%%%%%%%%%%%%%%%%%%%%%%%%%%%%%%%%%%%%%%%%%%%%%%%%
\subsection{Higher order polynomial spaces}
\label{sec:hofem}
%%%%%%%%%%%%%%%%%%%%%%%%%%%%%%%%%%%%%%%%%%%%%%%%%%%%%%%%%%%%%%%%%%
For polynomial degree $k\in\N$ and space dimension $d\ge 2$ denote in
the following by
$\bbP_{k} :={\rm
  span}\set{\prod_{j=1}^dx_j^{\nu_j}}{\sum_{j=1}^d\nu_j\le {k}}$ the
space of $d$ - variate polynomials of total degree at most $k$.
As observed in \cite{JMS-53-159}, networks employing the
``ReLU$^r$''\footnote{Also referred to as ``rectified power unit''
  (RePU).}  activation
\begin{equation*}
  \relu_r(x):=\relu(x)^r=\max\{0,x\}^r
\end{equation*}
for some fixed integer $r\ge 2$, can be used to express multivariate
polynomials in ${\bbP_{k}}$ exactly.  We use here a formulation of
this result from \cite{OSZ19}\footnote{We apply this result here with
  the multiindex set
  $\Lambda:=\set{(\nu_1,\dots,\nu_d)\in\N_0^d}{\sum_{j=1}^d\nu_j\le
    {k}}$, which has cardinality bounded by $(k+1)^d$.  Here, we
  denoted $\N_0 = \{0,1,\ldots\}$.}, extended to vector-valued
polynomials by parallelization:
\begin{proposition}[{{\cite[Proposition
      2.14]{OSZ19}}}] \label{prop:repu} Fix $d$, $\mu\in\N$, $r\in\N$,
  $r\ge 2$ and a polynomial degree $k\in\N$.

  Then there exists a constant $C>0$ independent of $d$, $\mu$ and $k$
  but depending on $r$ such that for any multivariate polynomial
  $w\in [{\bbP_{k}}]^{\mu}$ there is a NN $\Phi_{w}$, employing
  ReLU$^r$ activation,
  such that $ \realiz{\Phi_{w}}(x) = w(x) $, for all $ x\in \R^d
  $ and such that $\size(\Phi_{w})\le C
  {\mu}(k+1)^d$ and $\depth(\Phi_{w})\le Cd\log_2(k+1)$.
\end{proposition}

Combining Proposition \ref{prop:repu} with Proposition
\ref{prop:multbystep-alt} and Lemma
\ref{lem:indicemulation-lowdimalt}, 
by a similar argument as in Lemma \ref{lem:dpwl} we obtain a
generalization of this result to piecewise polynomial functions on
regular, simplicial partitions for all interelement-conformities which
arise from compatibility with the complex \eqref{derham}.
\begin{lemma}[Emulation of piecewise higher order polynomial elements]
  \label{lem:dpwp}
  Let $r\in\N$, $r\ge 2$.  
  For $d,s,\mu,k\in\N$ let
  $\domain\subset\R^d$ be a bounded polytope with boundary
  $\partial\domain$ being a finite union of plane, polytopal faces
  and let $\calT$ be a regular, simplicial partition of $\domain$ with
  $s = \snorm \calT$ elements, $\calT = \{ T_i \}_{i=1,\ldots,s}$.
  Let $u:\domain\to\R^{\mu}$ be a function that for all $i=1,\ldots,s$
  satisfies $u|_{T_i} \in[{\bbP_{k}}]^{\mu}$.

  Then there exists
  a NN $\Phi_u^{PwP}$ employing ReLU, ReLU$^r$ and BiSU activations
  and satisfies $u(x) = \realizc{\Phi_u^{PwP}}(x)$ for all
  $x \in \cup_{i=1}^s T_i$ and $\realizc{\Phi_u^{PwP}}(x) = 0$ for all
  $x\in\R^d\setminus\cup_{i=1}^s T_i$.  Furthermore,
  \begin{align*}
    \depth(\Phi_u^{PwP}) \leq C d\log_2(k+1)\;,\qquad
    \size(\Phi_u^{PwP}) \leq &\, C s \mu (k+1)^d.
  \end{align*}
  Here the constant $C$ is independent of $\calT$, $d$, $s$, $\mu$ 
  and of $k$ but depends on $r$.
\end{lemma}
Our results thus straightforwardly extend to piecewise polynomial
spaces of arbitrarily high order, covering all de Rham compatible
element families on simplicial partitions on polytopes as described in
\cite{FKDN2015}.  Importantly, as in the case of low-order finite
elements, the network size only scales linearly in the number
$s=\snorm\calT$ of simplices of the triangulation $\calT$.  Similarly,
also the results of Section \ref{sec:traces} extend to higher order
polynomials.

We now state the corresponding generalization of Proposition
\ref{prop:basisnet} for three types of higher order finite elements.
For arbitrary polynomial degree $k\in\N$ we recall the Lagrange FE
space
\begin{equation}\label{def:Sok}
  \Sop{k}(\calT, \domain) := \set{v \in H^1(\domain)}{v|_T \in \bbP_{k}, \ \forall T\in \calT } \subset  H^1(\domain)
\end{equation}
from \cite[Section 7.4]{ErnGuermondBookI2021}.  For all $k\in\N_0$ we
recall $ \mathbb{RT}_{k} = (\mathbb{P}_{k})^{d} \oplus x\Span\{ x^\alpha :
\alpha \in \N_0^d, \snorm{\alpha} = k \}$ and the Raviart-Thomas FE
space
\begin{equation} \label{def:RTk} \RTp{k}(\calT,\domain) := \set{v\in
    (L^1(\domain))^d}{v|_T \in \mathbb{RT}_{k} \ \forall T \in \calT
    \text{ and } [v\cdot n_f]_f = 0 \ \forall f \subset \domain}
  \subset H^0(\divv,\domain)
\end{equation}
from \cite[Sections 14.2 and 14.3]{ErnGuermondBookI2021}, and let for
all $k\in\N_0$
\begin{equation} \label{def:Szk} \Szp{k}(\calT,\domain) := \set{v\in
    L^1(\domain)}{v|_T \in \bbP_{k} \ \forall T \in \calT }\subset
  L^2(\domain) .
\end{equation}
For all $k\in\N$ and
$\blacklozenge\in\{ \Sop{k}, \RTp{k-1}, \Szp{k} \}$ and for a suitable
index set $\calI$ we will denote by
$\{ \theta^\blacklozenge_i \}_{i\in\calI}$ any collection of shape
functions of $\blacklozenge(\calT,\domain)$, each of which is
supported on $s(i) \leq \mathfrak{s}({\calI})$ elements of $\calT$.

\begin{proposition}
  \label{prop:hobasisnet}
  Let $\domain\subset\R^d$, $d\geq2$, be a bounded, polytopal domain
  and let $r\in\N$, $r\geq2$ be the power in the ReLU$^r$ activation,
  and let $k\geq 1$ denote the element degree.

  Then we have the following.
  \begin{enumerate}
  \item[{(i)}] For every regular, simplicial triangulation $\calT$ of
    $\domain$, every $k\in\N$ and every
    $\blacklozenge \in \{ \Sop{k}, \RTp{k-1}, \Szp{k} \}$ there exists
    a NN $\Phi^\blacklozenge :=\Phi^{\blacklozenge(\calT,\domain)}$
    with ReLU, ReLU$^r$ and BiSU activations, which in parallel
    emulates the basis functions
    $\{ \theta^\blacklozenge_i \}_{i\in \calI}$, that is
    $ \realiz{\Phi^\blacklozenge}\colon \domain \to \R^{|\calI|} $
    satisfies
    \begin{equation*}
      \realiz{\Phi^\blacklozenge}(x)_i
      = \, \theta^\blacklozenge_i(x)
      \quad\text{ for a.e. } x\in\domain
      \text{ and all } i\in\calI.
    \end{equation*}
  \item[{(ii)}] There exists $C>0$ independent of $d$, $k$ and $\calT$, but
    depending on $r$, such that with $\mu=1$ if
    $\blacklozenge\in\{ \Sop{k}, \Szp{k} \}$ and $\mu=d$ if
    $\blacklozenge = \RTp{k-1}$,
    \begin{align*}
      \depth(\Phi^\blacklozenge)
      \leq &\, C d\log(k+1)
             ,
             \qquad
             \size(\Phi^\blacklozenge) \leq
             C \mu(k+1)^d \sum_{i\in\calI} s(i).
    \end{align*}
  \item[{(iii)}] For
    every FE function
    $v = \sum_{i\in\calI} v_i \theta^\blacklozenge_i
    \in\blacklozenge(\calT,\domain)$, exists a NN
    $\Phi^{\blacklozenge,v} := \Phi^{\blacklozenge(\calT,\domain),v}$
    with ReLU, ReLU$^r$ and BiSU activations, such that for a constant
    $C>0$ independent of $d$, $k$ and $\calT$, but depending on $r$,
    \begin{align*}
      \realiz{\Phi^{\blacklozenge,v}}(x)
      = &\, v(x)
          \quad\text{ for a.e. } x\in\domain
          ,
      \\
      \depth(\Phi^{\blacklozenge,v})
      \leq &\, C d\log(k+1)
             ,
             \qquad
             \size(\Phi^{\blacklozenge,v}) \leq
             C {\mu}(k+1)^d \sum_{i\in\calI} s(i).
    \end{align*}

    The layer dimensions and the lists of activation functions of
    $\Phi^\blacklozenge$ and $\Phi^{\blacklozenge,v}$ are independent
    of $v$ and only depend on $\calT$ through $ \{s(i)\}_{i\in\calI} $
    and $ \snorm{\calI} = \dim(\blacklozenge(\calT,\domain)) $.
  \item[{(iv)}] For each
    $\blacklozenge \in \{ \Sop{k}, \RTp{k-1}, \Szp{k} \}$,
    \begin{equation}\label{eq:hoNNFEMDef}
      \calNN(\blacklozenge;\calT,\domain)
      := \{ \Phi^{\blacklozenge,v} : v \in \blacklozenge(\calT,\domain) \}\;,
    \end{equation}
    together with the linear operation
    $\Phi^{\blacklozenge,v} {\widehat{+}} \lambda\Phi^{\blacklozenge,w} 
    :=
    \Phi^{\blacklozenge,v+\lambda w}$ for all
    $v,w\in \blacklozenge(\calT,\domain)$ and $\lambda\in\R$ is a
    vector space, and the map
    $\realiz{\cdot}: \calNN(\blacklozenge;\calT,\domain) \to \blacklozenge(\calT,\domain)$ 
    is a linear isomorphism.
  \end{enumerate}

\end{proposition}

\begin{proof}
  For all $i\in\calI$ let $\Phi^\blacklozenge_i$ be the NN
  approximation of $\theta^\blacklozenge_i$ from Lemma
  \ref{lem:dpwp}.  Possibly after concatenating each
  $\Phi^\blacklozenge_i$ with $\Phi^{\Id}_{\mu,2}$, which only affects
  the constant $C$ in the bounds on depth and size from Lemma
  \ref{lem:dpwp}, we may assume that
  $\depth( \Phi^\blacklozenge_i ) \geq 3 = \depth(
  \Phi_{T_i}^{\mathbbm{1}} )$.  From the proof of Proposition
  \ref{prop:repu}, where $\Phi_{w}$ is the $\mu$-fold parallelization
  of a ReLU$^r$ network from \cite[Proposition 2.14]{OSZ19} applied
  with $\Lambda = \{ \bsnu\in\N_0^d : \sum_{j=1}^d\nu_j \leq k \}$, we
  see that the depth and the layer dimensions depend on 
  $d$, $\mu$ and $r$, but not on $w$.  The same holds for the network in
  Lemma \ref{lem:dpwp}, by an argument similar to that in the
  proof of Lemma \ref{lem:dpwl}.
  We can therefore define
  $\Phi^{\blacklozenge(\calT,\domain)} := \Parallelc{ \{
    \Phi^\blacklozenge_i \}_{i\in\calI} }$ and the sum
  $\Phi^{\blacklozenge(\calT,\domain),v} := \sum_{i\in\calI}
  v_i\Phi^\blacklozenge_i$ and obtain the linear structure of
  $\calNN(\blacklozenge;\calT,\domain)$ by the same arguments as in
  the proof of Proposition \ref{prop:basisnet}.

  It remains to prove the formula for the realization and to estimate
  the NN depth and size.  Firstly, we observe that indeed
  $ \theta_i^{\blacklozenge}(x) = \realizc{\Phi^{\blacklozenge}}(x)_i
  $
  for all $ x \in \domain \setminus \partial \calT $, where
  $ \partial \calT := \bigcup_{T\in \calT} \partial T$, and
  $ \realizc{\Phi^{\blacklozenge}}(x)_i = 0 $ else.

  Secondly, we apply Lemma \ref{lem:dpwp} with $s = s(i)$ and
  $\mu=1$ if $\blacklozenge\in\{ \Sop{k}, \Szp{k} \}$ and $\mu=d$ if
  $\blacklozenge = \RTp{k-1}$.  We obtain that
  \begin{align*}
    \depth( \Phi^{\blacklozenge} )
    \leq &\, C \mu \log_2(k+1)
           ,
           \qquad
           \size( \Phi^{\blacklozenge} )
           \leq \sum_{i=1}^{s(i)}  \size( \Phi^\blacklozenge_i )
           \leq \sum_{i\in\calI} C s(i) \mu (k+1)^d
           ,
  \end{align*}
  and the same bounds hold for the depth and size of
  $\Phi^{\blacklozenge,v}$.
\end{proof}

Definition \ref{def:basisnet} applies, and also Remark
\ref{rem:basisnet}.

\begin{remark}
\label{rem:hobasisnetrelu}
  ReLU NNs (and thus also ReLU+BiSU NNs) are known to be efficient at
  \emph{approximating} multivariate polynomials, see
  e.g. \cite{LS2017,yarotsky,OSZ19}.
  Thus, also ReLU+BiSU (rather than ReLU+ReLU$^r$+BiSU) networks could
  be employed to extend our results to higher order polynomial spaces,
  however only in an approximate sense.

  The resulting PwL NN realizations may violate the discrete exact
  sequence property however.
\end{remark}

%%%%%%%%%%%%%%%%%%%%%%%%%%%%%%%%%%%%%%%%%%%%%%%%%%%%%%%%%%%%%%%%%%%
\subsection{Crouzeix-Raviart elements $\CR$}
\label{sec:CRElement}
%%%%%%%%%%%%%%%%%%%%%%%%%%%%%%%%%%%%%%%%%%%%%%%%%%%%%%%%%%%%%%%%%%%
While this work focused on conformal discretization of functions in
the compatible spaces in \eqref{derham}, the result of Lemma
\ref{lem:dpwl} is more general and includes the non-conformal
Crouzeix-Raviart elements (e.g. \cite[Section
7.5]{ErnGuermondBookI2021}) of lowest order for $ d \ge 2 $.  Due to
the importance and widespread use of the Crouzeix-Raviart elements
(e.g. \cite{CrouzFalk,ChambPock,balci2021crouzeixraviart} and the
references there), we state a NN emulation result of these elements.
For $d \geq 2$ and a polytopal domain $\domain\subset\R^d$, let
$\calT$ be a regular, simplicial triangulation of $\domain$ as in
Section \ref{sec:nnfem}.  The lowest order Crouzeix-Raviart FE space
is defined as
\begin{equation} \label{def:CR0} \CR(\calT,\domain) := \set{v\in
    L^1(\domain)}{v|_T \in \mathbb{P}_1 \ \forall T \in \calT \text{
      and } \int_f[v]_f = 0 \ \forall f\in \calF, f \subset \Omega },
\end{equation}
where $ [v]_f $ denotes the jump of a function across $ f $, that is,
given a unit normal vector $n_f$ to $f$,
$ [v]_f(x_0) = \lim_{\epsilon\searrow 0} (v(x_0+\epsilon n_f)-
v(x_0-\epsilon n_f)) $ for all $ x_0 \in f $.
Analogously to the case of Raviart-Thomas FE, the space
$\CR(\calT,\domain)$ has one degree of freedom per face
$ f \in \calF $.  The corresponding shape functions are, for
$ f \subset \partial \domain $ and thus $s(f) = 1$,
$ \theta^{\CR}_f(x) := d(\frac1d - (1-\frac{|f|(x-a_1)\cdot
  n_f}{d|T_1|})) \mathbbm{1}_{T_1} $, where
$ f\subset \overline{T_1} $, $ T_1 \in \calT $ and $ a_1 $ is the only
vertex of $ T_1 $ that does not belong to $ \overline{f} $.  For
interior faces $ f \subset \domain $ and thus $s(f) = 2$, we construct
$ \theta^{\CR}_f $ by assembling local shape functions of the
neighboring simplices $ T_1, T_2 $ with
$ \overline{f} = \overline{T_1}\cap \overline{T_2}$,
\begin{equation} \label{eq:CRshape} \theta_f^{\CR}(x):=
  \begin{cases}
    d(\frac1d - (1-\frac{|f|(x-a_1)\cdot n_f}{d|T_1|}))  & \text{ if } x \in T_1, \\
    d(\frac1d - (1+\frac{|f|(x-a_2)\cdot n_f}{d|T_2|}))  & \text{ if } x \in T_2, \\
    0 & \text{ if } x \notin \overline{T_1\cup T_2},
  \end{cases}
\end{equation}
where $ a_1,a_2 $ are the the only vertices of $ T_1,T_2 $,
respectively, not belonging to $ \overline{f} $.  The following
proposition allows us to apply Proposition \ref{prop:basisnet},
Definition \ref{def:basisnet} and Remark \ref{rem:basisnet} with
$\blacklozenge = \CR$.
\begin{proposition} \label{prop:CR}
  Given $ f \in \calF $,
  let $ \{ T_i \}_{i=1}^{s(f)} $ be the simplices adjacent to $ f $
  and let
  $ a_i := (\calV \cap \overline{T_i}) \setminus \overline{f} \in
  \R^{d}, i=1,\ldots,s(f) $.  Then, there exist
  $ A_{f,T_i}\in \R^{1\times d},b_{f,T_i} \in \R $,
  $ i=1,\ldots,s(f) $ such that
  \begin{equation}
    \label{eq:CRdef2}
    \Phi_f^{\CR} := \sum_{i = 1}^{s(f)} \Phi^{\times}_{1,\kappa} \sconc \Parallelc{ \Phi_{1,2}^{\Id} \sconc
      \left (\left ( A_{f,T_i},b_{f,T_i},\Id_{\R}\right )\right ) , \Phi_{T_i}^{\mathbbm{1}}}
  \end{equation}
  satisfies $\theta_f^{\CR}(x) = \realiz{\Phi^{\CR}_f}(x)$ for a.e.\
  $ x \in \domain$, for any
  \begin{equation}
    \label{eq:CRkappa2}
    \kappa \geq d-1.
  \end{equation}
  In addition, there exists a constant $C>0$ that is independent of
  $d$ and $\calT$ such that for all $f\in\calF$
  \begin{align*}
    \depth(\Phi_f^{\CR}) = 5
    , \qquad
    \size(\Phi_f^{\CR}) \leq C d^2 s(f) \leq 2 C d^2
    .
  \end{align*}
\end{proposition}
\begin{proof}
  The values of $ A_{f,T_i},b_{f,T_i} $ can be read from
  \eqref{eq:CRshape}.
  Similar to Proposition \ref{prop:RT},
  $ \theta_f^{\CR}(x) = \realiz{\Phi_f^{\CR}}(x) $ for all
  $ x \in \domain \setminus \partial \calT $, where
  $ \partial \calT := \bigcup_{T\in \calT} \partial T$.

  We conclude applying Lemma \ref{lem:dpwl} with $\mu = 1$, $m=d+1$
  and $s = s(f)$.
\end{proof}
The same idea carries over to higher order Crouzeix-Raviart elements
and canonical hybrid elements \cite[Section
7.6]{ErnGuermondBookI2021}, along the lines of Section
\ref{sec:hofem}.
%
%%%%%%%%%%%%%%%%%%%%%%%%%%%%%%%%%%%%%%%%%%%%%%%%%%%%%%%%%%%%%%%%%%
\subsection{Domains of general topology}
\label{sec:Topol}
%%%%%%%%%%%%%%%%%%%%%%%%%%%%%%%%%%%%%%%%%%%%%%%%%%%%%%%%%%%%%%%%%%
%
In our discussion of the de Rham
complex (see Section \ref{sec:derham})
we assumed throughout that the physical domain $\domain$
is contractible.
This renders the topology of $\domain$ trivial: its
Betti-numbers are $b_0=1$, $b_1=b_2=b_3=0$.  As is well-known, for
polytopal domains $\domain$ with a non-trivial topology (e.g. domains
$\domain\subset \mathbb{R}^3$ with voids) in \eqref{derham} the
cohomology spaces
$$
\begin{array}{rcl rcl}
  \calH_0
  &:=&
       {\rm Ker}\;{\operatorname{grad}} / {\rm Im}\;  i
  &
    \calH_1
  &:=&
       {\rm Ker}\;{\curl}/ {\rm Im} \; \operatorname{grad}
  \\
  \calH_2
  &:=&
       {\rm Ker}\;\divv / {\rm Im}\; \curl
  &
    \calH_3
  &:=&
       L^2(\domain)/{\rm Im}\; \divv
\end{array}
$$
are nontrivial.
Our neural network emulation results are given without topological restrictions
on the bounded polytopal domain $\Omega$. 
Therefore,
the presently proposed DNN emulations of
de Rham compatible FE spaces on simplicial partitions preserve these
cohomology spaces provided that the corresponding discrete homology spaces
$\calH_i(\calT)$ for the FE spaces in $\domain$ are
isomorphic to $\calH_i$. This property has been verified for 
several large classes of FE spaces (see, e.g., \cite{CohomDDR22} and
the references there). 

%%%%%%%%%%%%%%%%%%%%%%%%%%%%%%%%%%%%%%%%%%%%%%%%%%%%%%%%%%%%%%%%%%
\subsection{Conclusions}
\label{sec:Concl}
%%%%%%%%%%%%%%%%%%%%%%%%%%%%%%%%%%%%%%%%%%%%%%%%%%%%%%%%%%%%%%%%%%
%
The present construction of deep NN emulations of de Rham compatible
Finite Element spaces was given for the lowest order 
Finite Element families on regular, simplicial partitions $\calT$ of
$\domain$.  Generalizing recent work \cite{HLXZ2020}, we provided
exact emulation of continuous piecewise linear functions (``Courant''
Finite Elements) on arbitrary, regular simplicial partitions in any
space dimension by ReLU networks. As shown, for uniformly shape
regular partitions the network size in this construction merely scales
linearly with the number of elements.

As is well known (e.g.\ \cite{FKDN2015} and the reference there) the
presently emulated, lowest order element families are embedded in
hierarchies of higher-order Finite Element families for arbitrary
polynomial order.
We argued that admitting higher order, so-called ReLU$^r$ activations
with $r\in\N$, $r\geq 2$ allows to exactly emulate the higher order
element families from \cite{FKDN2015} along the lines of the present
constructions.

Compatible constructions similar to the ones developed here are also
possible on \emph{affine partitions} $\calT$ (comprising elements that
are affine images of reference elements) which contain other element
shapes, in particular quadrilaterals ($d=2$) and hexahedral elements
($d=3$).  We refer to \cite[Sec. 4 and 6]{FKDN2015} for details on the
shape functions.

The present results, in particular Proposition \ref{prop:neuraltrace},
can be the basis to extend the recently proposed frameworks of
``PiNN'' \cite{RPK_2019} and ``deep Ritz'' \cite{EYuDeepRitz} for DNN
discretization of PDEs to larger classes of PDEs, and to corresponding
boundary integral formulations (see, e.g., \cite{SS11} for such
methods, and \cite{AHS22_989} for a realization of this approach for a
model problem).
While in this paper we mainly concentrated on the de Rham formalism,
our ideas and proofs naturally extend also to compatible
discretizations of more general structures, as occur in the so-called
Finite Element Exterior Calculus (FEEC) (e.g.\ \cite{AFWII} and the
references there).

Similarly, with Lemma \ref{lem:dpwp} other nonconforming FEM
such as Hybridized, High Order (``HHO'') FEM can be emulated with
appropriate functionals which account for element interface unknowns
and reduced interelement conformity, see,
e.g. \cite[Prop. 1.8]{cicuttin2021hybrid}.
%%%%%%%%%%%%%%%%%%%%%%%%%%%%%%%%%%%%%%%%%%%%%%%%%%%%%%%%%%%%%%%%%%
\appendix
\section{Proofs}
\label{sec:proofs}
%%%%%%%%%%%%%%%%%%%%%%%%%%%%%%%%%%%%%%%%%%%%%%%%%%%%%%%%%%%%%%%%%%

%%%%%%%%%%%%%%%%%%%%%%%%%%%%%%%%%%%%%%%%%%%%%%%%%%%%%%%%%%%%%%%%%%
\subsection{Proofs from Section \ref{sec:nndef}}
\label{sec:proofsnndef}
%%%%%%%%%%%%%%%%%%%%%%%%%%%%%%%%%%%%%%%%%%%%%%%%%%%%%%%%%%%%%%%%%%

\begin{proofof}{Proposition \ref{prop:multbystep-alt}}
  This proof is in two steps.
  In Step 1, we define a function of $x,y$ that computes the desired output for $d=1$.
  In Step 2, we construct a NN which exactly emulates that function $d$ times and
  estimate its depth and size.

  \textbf{Step 1.}  For $x,y\in\R$ let
  \begin{align*}
    f(x,y) &:= \tfrac12 \left( \relu(x+y) + \relu(-x-y) - \relu(x-y) - \relu(-x+y) \right).
  \end{align*}

  Note that for all $x\in[-1,1]$ and $y\in[0,1]$ such that
  $\snorm{x}\leq y \leq 1$ it holds that
  $f(x,y) = \tfrac12 \relu(x+y) + 0 - 0 - \tfrac12 \relu(-x+y) = x $
  and that for all $x\in[-1,1]$ it holds that
  $f(x,0) = \tfrac12 \relu(x) + \tfrac12 \relu(-x) - \tfrac12
  \relu(-x) - \tfrac12 \relu(x) = 0$.  Hence, $f$ satisfies
  \begin{align*}
    f(x,y) = &\, xy, \text{ for all } x\in [-1,1] \text{ and } y\in\{0,1\},
  \end{align*}
  and thus for all $x\in[-\kappa,\kappa]$ and $y\in\{0,1\}$ it follows
  that
  \begin{align*}
    \kappa f(\tfrac{x}{\kappa},y) = xy.
  \end{align*}

  \textbf{Step 2.}  For $d=1$, let

\begin{align*}
  \Phi^{\times}_{1,\kappa} := \left(
  \left( \begin{pmatrix} \tfrac{1}{\kappa} & 1  \\ -\tfrac{1}{\kappa} & -1 \\ \tfrac{1}{\kappa} & -1 \\ -\tfrac{1}{\kappa} & 1 \end{pmatrix},
                                                                                                                             \begin{pmatrix} 0 \\ 0 \\ 0 \\ 0 \end{pmatrix}, \begin{pmatrix} \relu \\ \relu \\ \relu \\ \relu \end{pmatrix} \right) ,
  \left( \begin{pmatrix} \tfrac \kappa2 & \tfrac \kappa2 & -\tfrac \kappa2 & -\tfrac \kappa2 \end{pmatrix},  0 ,  \Id_\R  \right)
                                                                             \right),
\end{align*}
which satisfies
\begin{align*}
  \realiz{\Phi^{\times}_{1,\kappa}} (x,y) = &\, \kappa f(\tfrac{x}{\kappa},y) \;
                                              \text{ for all } (x,y)\in\R^2,
  \\
  \depth(\Phi^{\times}_{1,\kappa}) = &\, 2 , \qquad
                                       \size(\Phi^{\times}_{1,\kappa}) = \, 8 + 4 = 12
                                             .
\end{align*}

Similarly, with
$u_1 := ( \tfrac1\kappa, -\tfrac1\kappa, \tfrac1\kappa, -\tfrac1\kappa
)^\top$, $u_2 := ( 1, -1, -1, 1)^\top$ and
$u_3 := ( \tfrac\kappa2, \tfrac\kappa2, -\tfrac\kappa2, -\tfrac\kappa2
)$, we define for $d>1$ the following network with layer sizes
$N_0 = d+1$, $N_1 = 4d$ and $N_2 = d$:
\begin{align*}
  \Phi^{\times}_{d,\kappa} := \left(
  \left( \begin{pmatrix} u_1 & & & u_2 \\ & \ddots & & \vdots \\ && u_1 & u_2 \end{pmatrix},
                                                                          \begin{pmatrix} 0 \\ \vdots \\ 0 \end{pmatrix}, \begin{pmatrix} \relu \\ \vdots \\ \relu \end{pmatrix} \right) ,
  \left( \begin{pmatrix} u_3 & & \\ & \ddots & \\ &  & u_3 \end{pmatrix}, \begin{pmatrix} 0 \\ \vdots \\ 0 \end{pmatrix}, \begin{pmatrix} \Id_\R \\ \vdots \\ \Id_\R \end{pmatrix} \right)
  \right),
\end{align*}
which satisfies
\begin{align*}
  \realiz{\Phi^{\times}_{d,\kappa}} (x_1,\ldots,x_d,y) = &\,
                                                           ( \kappa f(\tfrac{x_1}{\kappa},y), \ldots, \kappa f(\tfrac{x_d}{\kappa},y) )^\top \in \R^d,
                                                           \text{ for all } (x,y)\in\R^d\times\R,
  \\
  \depth(\Phi^{\times}_{d,\kappa}) = &\, 2 , \qquad
                                       \size(\Phi^{\times}_{d,\kappa}) = \, 12 d
                                             .
\end{align*}
\end{proofof}

\begin{proofof}{Lemma \ref{lem:indicemulation-lowdimalt}}
  From
$$
1 - \heavi(y) - \heavi(-y) = \begin{cases}
  1 & \text{ if } y=0, \\
  0 & \text{ otherwise},
\end{cases}
\qquad \text{ for all } y\in\R,
$$
it follows that for all $x\in\R^d$
\begin{align*}
  \realiz{\Phi^{\mathbbm{1}}_{\domain}}(x) = &\, \heavi
            \left( \sum_{i=1}^{n} ( 1 - \heavi(A_i x + b_i) - \heavi( - A_i x - b_i) )
                      + \sum_{i=n+1}^N \heavi(A_i x + b_i) - (N-\tfrac14) \right)
  \\
  = &\,
      \begin{cases} 1 & \text{ if } x\in \domain,
        \\
        0 & \text{ otherwise},
      \end{cases}
  \\
  \depth(\Phi^{\mathbbm{1}}_\domain)
  = &\, 3
      ,
      \qquad
      \size(\Phi^{\mathbbm{1}}_\domain)
      \leq \, ((N+n)d+(N+n)) + ((N+n)+1) + 1 = (d+2)(N+n) + 2
.
\end{align*}
\end{proofof}
%%%%%%%%%%%%%%%%%%%%%%%%%%%%%%%%%%%%%%%%%%%%%%%%%%%%%%%%%%%%%%%%%%
\subsection{Proofs from Section \ref{sec:nnfem}}
\label{sec:proofsnnfem}
%%%%%%%%%%%%%%%%%%%%%%%%%%%%%%%%%%%%%%%%%%%%%%%%%%%%%%%%%%%%%%%%%%
\begin{proofof}{Proposition \ref{prop:RTv2}}
    Below, we prove the result for $f\subset\domain$, i.e. ${s}(f) = 2$.
    The case $f\subset\partial\domain$, i.e. ${s}(f) = 1$, follows analogously.

    Observe that we can write
    $ T_i = \conv(\{a_i,p_1,\ldots,p_{d}\}) $
    where
    $   \{p_1,\ldots,p_{d}\} := \calV\cap\overline{f} $.
    The point values $ \theta_f^{\RT}(p_j)\cdot n_f = 1 $, $ \forall j = 1,\ldots,d $
    are well-defined by continuity of $ \theta_f^{\RT} \cdot n_f$ across $ f $.
    Therefore, we can take $ A_{n_f}^{(i)} \in \R^{1\times d}, b_{n_f}^{(i)} \in \R $, $ i = 1,2 $
    to be the matrices and vectors solving
    \begin{align}\label{RTaffine}
        (A_{n_f}^{(i)}, b_{n_f}^{(i)})
        \begin{pmatrix}
            (a_{i})_1 & (p_1)_1  &  & (p_{d})_1 \\
            \vdots & & \ddots& \\
            (a_{i})_d & (p_1)_d  &  & (p_{d})_d \\
            1 & 1 & \cdots & 1
        \end{pmatrix} = (0,1,\ldots,1)
        ,
    \end{align}
    where  $(0,1,\ldots,1) = (-1)^{i-1} \tfrac{\snorm{f}}{d\snorm{T_i}} (0, (p_1-a_i)\cdot n_f, \ldots, (p_d-a_i)\cdot n_f )$.
    With this choice, since $ (\theta_f^{\RT}(x) \cdot n_f ) \in [0,1]  $ for $ x \in T_1 \cup T_2 \cup f $,
    it holds that $ \theta_f^{\RT}(x)\cdot n_f = \realizc{\Phi^{\RT,\perp}_f}(x)$
    for a.e.\ $ x \in \domain $ and every $ x\in f $.
    On the other hand, the discontinuous tangential component can be assembled element by element,
    as in Proposition \ref{prop:RT}:
    matrices and vectors
    $ (A_{t_j}^{(i)},b_{t_j}^{(i)}) \in \R^{1\times(d+1)} $,
    $ i=1,2 $ and $ j=1,\ldots,d-1 $ which only depend on $ T_i $ and $ t_j $
    can be computed as in \eqref{RTaffine}, but with different right-hand sides,
    namely $(-1)^{i-1} \tfrac{\snorm{f}}{d\snorm{T_i}} (0, (p_1-a_i)\cdot t_j, \ldots, (p_d-a_i)\cdot t_j ) \in \R^{1\times (d+1)}$.

    Finally, we estimate the network depth and size.
    For $d=2$, as in Remark \ref{rmk:EstMinMax} 
    let
    \begin{align*}
        \Phi^{\min}_2 :=&\, \left(
        \left( \begin{pmatrix} -1 & 1 \\ 0 & 1 \\ 0 & -1 \end{pmatrix},
        \begin{pmatrix} 0 \\ 0 \\ 0\end{pmatrix} ,
        \begin{pmatrix} \relu \\ \relu \\ \relu \end{pmatrix} \right) ,
        \left( \begin{pmatrix} -1 & 1 & -1 \end{pmatrix},  0 ,
        \Id_\R \right)
        \right)
        ,
        \\
        \depth(\Phi^{\min}_2) = &\, 2,\qquad
        \size(\Phi^{\min}_2) = \, 7.
    \end{align*}
    In particular, we use that 
    $\depth( \Phi^{\min}_2 ) + 1
    = \depth( \Phi_{T_1}^{\mathbbm{1}} + \Phi_{f}^{\mathbbm{1}} + \Phi_{T_2}^{\mathbbm{1}} ) = 3$,
    i.e. in \eqref{eq:RTdefn} both components in the parallelization have equal depth.
    Also, because the networks for $s(f) = 1$
    have smaller sizes than those for $s(f) = 2$,
    we will only estimate the sizes of the latter.
    In the bound on the size of
    $\Phi_{T_1}^{\mathbbm{1}} + \Phi_{f}^{\mathbbm{1}} + \Phi_{T_2}^{\mathbbm{1}}$
    in \eqref{eq:RTdefn},
    we use for $T_1,T_2$ Lemma \ref{lem:indicemulation-lowdimalt} with $N = d+1$ and $n=0$,
    whereas for $f$ we use Lemma \ref{lem:indicemulation-lowdimalt} with $N = d+1$ and $n=1$.
    The size of the network in \eqref{eq:RTdeft} is estimated using Lemma \ref{lem:dpwl}
    with ${\mu}=1$, $m=d+1$ and $s=2$.
    \begin{align*}
        \depth(\Phi^{\RT,\perp}_f)
        = &\, \depth( \Phi^\times_{1,1} )
        + \depth( \Phi_{T_1}^{\mathbbm{1}} + \Phi_{f}^{\mathbbm{1}} + \Phi_{T_2}^{\mathbbm{1}} )
        = 5
        ,
        \\
        \size(\Phi^{\RT,{\perp}}_f) \leq &\, 2 \size( \Phi^\times_{1,1} ) + 2 \size\left( \Parallelc{\Phi_2^{\min}\sconc \left ( \left (
            \begin{pmatrix} A_{n_f}^{(1)} \\ A_{n_f}^{(2)} \end{pmatrix},
            \begin{pmatrix} b_{n_f}^{(1)} \\ b_{n_f}^{(2)} \end{pmatrix}, \Id_{\R^2} \right ) \right ),
            \Phi_{T_1}^{\mathbbm{1}} + \Phi_{f}^{\mathbbm{1}} + \Phi_{T_2}^{\mathbbm{1}} } \right)
        \\
        \leq &\, 2 \size( \Phi^\times_{1,1} ) + 4 \size( \Phi^{\min}_2 ) + 4 \size\left( \left ( \left (
        \begin{pmatrix} A_{n_f}^{(1)} \\ A_{n_f}^{(2)} \end{pmatrix},
        \begin{pmatrix} b_{n_f}^{(1)} \\ b_{n_f}^{(2)} \end{pmatrix}, \Id_{\R^2} \right ) \right ) \right)
        \\
        &\, + 2 \size( \Phi_{T_1}^{\mathbbm{1}} ) + 2 \size( \Phi_{f}^{\mathbbm{1}} )
        + 2 \size( \Phi_{T_2}^{\mathbbm{1}} )
        \\
        \leq &\, C ( C + C + C d + C d^2 + C d^2 + C d^2 )
        \leq C d^2
        .
    \end{align*}
\end{proofof}

%%%%%%%%%%%%%%%%%%%%%%%%%%%%%%%%%%%%%%%%%%%%%%%%%%%%%%%%%%%%%%%%%%
\subsection{Proofs from Section \ref{sec:cpwlrelu}}
\label{sec:proofscpwlrelu}
%%%%%%%%%%%%%%%%%%%%%%%%%%%%%%%%%%%%%%%%%%%%%%%%%%%%%%%%%%%%%%%%%%

In this section we give proofs of Theorems  \ref{thm:theorem_min_max_basis_functions} and \ref{thm:relucpwl}.
Our proof strategy is to write a non-convex patch as a suitable union of
convex ones, and thereby reduce the problem to the convex case.

\begin{lemma} \label{lemma:move_point_outwards}
    For $d\in\N$, let $T=\simp{a_{0}}{a_d}$ be a simplex and
    $\delta>0$.
    Define $q:=a_{0} + \delta \sum_{i=1}^d (a_{0}-a_i)$.
    Then $T_\delta:=\simp{q,a_1}{a_d}$ is a simplex and $a_{0}\in T_\delta$.
\end{lemma}
\begin{proof}
    Without loss of generality, $a_{0} = 0$.
    To show that $T_\delta$ is a simplex,
    it suffices to verify $a_{0} = 0 \in T_\delta$,
    as it then follows that $T\subset T_\delta$, i.e.\
    $T_\delta$ has nonempty interior and is thus a simplex.
    By definition,
    \begin{equation*}
        T_\delta = \left\{ \alpha_{0} \left( \delta \sum_{i=1}^d -a_i \right) +  \sum_{i=1}^d \alpha_i a_i :
        \sum_{i=0}^d \alpha_i = 1 \, \text{ and } \alpha_i > 0 \right\}.
    \end{equation*}
    Therefore, $a_{0} \in T_\delta$ is equivalent to
    \begin{equation*}
        \alpha_{0} \delta \sum_{i=1}^d a_i =  \sum_{i=1}^d \alpha_i a_i,
    \end{equation*}
    which holds if and only if $\alpha_{0} \delta = \alpha_i
    $ for all
    $i=1,\ldots,d$. A viable choice satisfying
    $\sum_{i=0}^d \alpha_i=1$ and $\alpha_i>0$ is $\alpha_0 =
    (1+d\delta)^{-1}$ and $\alpha_i = \delta
    (1+d\delta)^{-1}$ for all $i=1,\ldots,d$, and thus $a_0\in
    T_\delta$.
\end{proof}
\begin{proposition}
    \label{prop:new_patch_from_interior_point}
    Given a simplex $T=\simp{a_{0}}{a_d}$ and a point $p\in T$, let
    \begin{equation*}
        T_i := \conv(\{p,a_{0}, \dots ,a_d\}\setminus \{a_i\} )  \quad \text{ for all } i \in \{0,\dots,d\}.
    \end{equation*}
    Then
    \begin{equation}\label{eq:Ticlos}
        \overline{\underset{i \in \{0,\dots,d\} }{\bigcup} T_i} = \overline{T}
    \end{equation}
    and this is a patch,
    and $\set{T_i}{i\in\{0,\ldots,d\}}$ is a regular simplicial partition of $T$.
\end{proposition}
\begin{proof}
    Let $p\in T$, i.e.
    \begin{equation*}
        p = \sum_{i=0}^d \alpha_i a_i, \quad 
        \alpha_i > 0 \text{ and } \sum_{i=0}^d \alpha_i = 1.
    \end{equation*}
    First we show that
    $T_0$ is a simplex, which by symmetry implies that
    $T_i$ is a simplex for all
    $i\in\{0,\dots,d\}$.  It suffices to check that $p-a_1\notin
    \operatorname{span}\{a_2-a_1,\dots,a_d-a_1\}$, since then
    $\{p-a_1,a_2-a_1,\dots,a_d-a_1\}$ is a set of linearly
    independent vectors. This is true since
    $\{a_0-a_1,a_2-a_1,\dots,a_d - a_1\}$ 
    are linearly independent vectors,
    $\alpha_1=1-\sum_{i\neq 1}\alpha_i$ and thus
    \begin{equation*}
        p - a_1 = \sum_{i\neq 1} \alpha_i (a_i-a_1),
    \end{equation*}
    with
    $\alpha_0>0$ does not belong to
    $\operatorname{span}\{a_2-a_1,\dots,a_d-a_1\}$.
    Furthermore,
    $\bigcup_{i \in \{0,\dots,d\}} \overline{T_i} \subset\overline{T}$
    follows by the fact that $p\in
    T$ implies $T_i\subset T$ for all $i\in\{0,\dots,d\}$.

    Next we show
    $\bigcup_{i \in \{0,\dots,d\}} \overline{T_i} \supset
    \overline{T}$. Fix $\bar{p} := \sum_{j=0}^d \gamma_j a_j$
    with $\gamma_j\geq 0$ satisfying $\sum_{j=0}^d \gamma_j = 1$,
    i.e.\ $\bar p$ is an arbitrary point in $\overline{T}$. We wish
    to show that $\bar p \in \overline{T}_i$ for some
    $i\in\{0,\dots,d\}$, i.e.
    \begin{equation*}
        \sum_{j=0}^d \gamma_j a_j=
        \sum_{j\neq i}^d \beta_j a_j + \beta_i \sum_{j=0}^d \alpha_j a_j
        \quad
        \text{for some}\quad \beta_j \geq 0,~\sum_{j=0}^d \beta_j =1.
    \end{equation*}
    This is equivalent to
    \begin{equation}
        \label{eq:barp}
        \sum_{j=0}^d (\gamma_j - \beta_i \alpha_j) a_j = \sum_{j\neq i}^d \beta_j a_j.
    \end{equation}
    We now show that there exist $(\beta_j)_{j=0}^d$ for which this holds.  Let
    \begin{equation}\label{eq:iargmin}
        i \in \underset{j \in \{ 0, \dots, d  \}}{\argmin} \frac{\gamma_j}{\alpha_j},
    \end{equation}
    which is well-defined 
    because $\gamma_j\geq0$ and $\alpha_j>0$ for all
    $j\in\{0,\ldots,d\}$.  Since
    $\sum_{j=0}^d \alpha_{j} = 1 = \sum_{j=0}^d \gamma_{j}$, for $i$
    in \eqref{eq:iargmin} it must hold
    $\frac{\gamma_i}{\alpha_i}\leq 1$.  Equation \eqref{eq:barp} holds
    if $\gamma_i - \beta_i \alpha_i = 0$ and
    $\gamma_j - \beta_i \alpha_j = \beta_j$ for all $j\neq i$.  The
    former is satisfied for
    $\beta_i = \frac{\gamma_i}{\alpha_i} \in [0,1]$ and the latter is
    satisfied if $\beta_j = \gamma_j - \beta_i \alpha_j$, which
    implies
    $\beta_j = \alpha_j ( \tfrac{\gamma_j}{\alpha_j} -
    \tfrac{\gamma_i}{\alpha_i} ) \geq 0$ and
    $\beta_j = \alpha_j ( \tfrac{\gamma_j}{\alpha_j} -
    \tfrac{\gamma_i}{\alpha_i} ) \leq \gamma_j \leq 1$.
    It is left to show that $\sum_{j=0}^d\beta_j=1$.
    We have
    \begin{equation*}
        \begin{aligned}
            \sum_{j\neq i}^d \beta_j + \beta_i & = \sum_{j\neq i}^d
            \alpha_j ( \tfrac{\gamma_j}{\alpha_j} -
            \tfrac{\gamma_i}{\alpha_i} ) + \tfrac{\gamma_i}{\alpha_i} =
            \sum_{j\neq i}^d \gamma_j + \tfrac{\gamma_i}{\alpha_i}
            \left( 1 - \sum_{j\neq i}^d \alpha_j \right) = \sum_{j\neq
                i}^d \gamma_j + \tfrac{\gamma_i}{\alpha_i} \alpha_i = 1.
        \end{aligned}
    \end{equation*}
    We found $i\in\{0,\dots,d\}$ and $(\beta_j)_{j=0}^d$ for
    which \eqref{eq:barp} holds. Thus $\bar{p} \in \overline{T_i}$,
    and
    $\bigcup_{i \in \{0,\dots,d\}} \overline{T_i} \supset
    \overline{T}$.

    It is left to show that for all $m\neq n\in\{0,\ldots,d\}$
    the intersection of $\overline{T_m}$ and $\overline{T_n}$ is the
    closure of a sub-simplex of both.  Consider
    \begin{equation*}
        \begin{aligned}
            \overline{T_m} &= \setc{ \sum_{j\neq m}^d \beta_j a_j + \beta_m p }{\sum_{j = 0}^d \beta_j = 1 \text{ and } \beta_{j} \geq 0 }, \\
            \overline{T_n} &= \setc{ \sum_{j\neq n}^d \beta_j a_j +
                \beta_n p }{\sum_{j = 0}^d \beta_j = 1 \text{ and }
                \beta_{j} \geq 0 } .
        \end{aligned}
    \end{equation*}
    Then
    \begin{equation*}
        \overline{T_m} \cap \overline{T_n}
        =
        \setc{ \sum_{j\neq m,n}^d \beta_j a_j + \beta p }{\beta+\sum_{j\neq m,n}^d \beta_j = 1 \text{ and } \beta, \beta_{j}\geq 0},
    \end{equation*}
    which, by definition, is the closure of a sub-simplex of both $\overline{T_m}$ and $\overline{T_n}$.
\end{proof}

For a set $S\subset\R^d$, in the following we call
$x\in S$ a \emph{star point} of $S$ iff for all $y\in S\setminus\{x\}$
holds $\conv(\{x,y\})\subset S$.
\begin{lemma}
    \label{lemma:same_side_hyper_plane_star}
    Let $ p \in\calV\cap\interior\domain $ and let $ T_1, \ldots T_{s(p)} \in \calT $
    be the simplices adjacent to $ p $. For each $j=1,\ldots,s(p)$, we
    denote by $P_{T_j}$ the hyperplane passing through all vertices of
    $T_j$ except $p$. Then, any $x \in \interior \omega(p)$ that is on
    the same side of the hyperplane $P_{T_j}$ as $p$ for all
    $j=1,\ldots,s(p)$ is a star point for the patch $\omega(p)$.
\end{lemma}
\begin{proof}
    We divide the proof of the claim in three steps:

    \textbf{Step 1.}
    We claim that $\partial \omega(p) \subset \bigcup_{j=1}^{s(p)} P_{T_j}$ for all $ p \in\calV\cap\interior\domain $.
    To prove this, we define a regular partition $ \calT' $ of $ \R^d $ that extends $ \calT $,
    i.e.\ such that $ \calT \subset \calT' $.
    For every point $z$ on $\partial\omega(p)$,
    $z$ is on the boundary of an element $T\subset\omega(p)$
    and of an element $T'\subset\R^d\setminus\omega(p)$, $T,T'\in\calT'$.
    By regularity of $\calT'$,
    $\overline{T}\cap\overline{T'}$ is the closure of a subsimplex $f$
    of both $\overline{T}$ and $\overline{T'}$.
    Because $T$ is a simplex with $p$ as one of its vertices,
    if $z$ is not in $P_T$, then $f$ touches $p$,
    which implies that $T'\subset\omega(p)$ and gives a contradiction.

    \textbf{Step 2.}
        We show that any star point $ x $ of $ \interior \omega(p) $ is a star point of $ \omega(p) $. Given $ q\in \omega(p) $, define a sequence $ \{q_n\}\subseteq \interior\omega(p) $ such that $ q_n \to q $, then for all $ t\in [0,1] $ we obtain
        $ \interior\omega(p) \ni xt + q_n(1-t) \to xt + q(1-t) \in \omega(p) $, as $ \omega(p) $ is closed.

    \textbf{Step 3.}
    Assume 
    that $x\in\interior\omega(p)$ is not a star point of $\omega(p)$.
    By Step 2, there exists $q\in{\interior} \omega(p)$ and $t\in [0,1]$
    such that $tx + (1-t)q \notin {\interior} \omega(p)$.
    Therefore, there exist
    $\underline{t},\bar{t} \in ( 0,1 )$ satisfying
    $\underline{t} < \bar{t}$, and $T\in\calT$, $T\subset\omega(p)$,
    such that, using Step 1,
    $\bar{t} x + (1-\bar{t})q \in P_{T}$
    and $sx + (1-s)q \in T$,
    $\forall s\in (\underline{t} , \bar{t})$.  Since $p$ and $T$ lie on
    the same side of $P_T$,
    $sx + (1-s)q$ lies on the other side of $P_{T}$ than $p$
    for all $s\in (\bar{t}, 1]$.
    For $s=1$,  this implies that $x$ is on the other side of the
    hyperplane
    $P_{T}$ than $p$. Thus if $x\in\interior\omega(p)$ is
    not a star point, it lies on the other side of at least one
    hyperplane
    $P_{T_j}$ than $p$.  Therefore, if a point
    $x\in\interior\omega(p)$ is on the same side of $P_{T_j}$
    as $p$ for all $j=1,\ldots,s(p)$, then $x$ is a star point for
    $\omega (p)$, by contradiction.
\end{proof}

\begin{remark}
    The converse implication of Lemma
    \ref{lemma:same_side_hyper_plane_star} holds as well: the set of all
    star points of the patch $\omega(p)$ coincides with the intersection
    $\bigcap_{j=1}^{s(p)}\overline{H_j}$ of all closed half-spaces
    $\overline{H_j}$, where $H_j\subset\R^d$ is defined as the set of
    all points that lie on the same side of $P_{T_j}$ as $p$.
\end{remark}

In the following,
denote by $ B_\epsilon(p) \subset \R^d $ the ball of radius $ \epsilon > 0 $ centered at $ p \in \R^d$,
with respect to the Euclidean norm.
\begin{lemma}
    \label{lemma:ball_epsilon}
    For all $ p \in\calV\cap\interior\domain$, there exists $\epsilon>0$ such that
    $B_\epsilon(p) \subset \omega(p)$ and such that every
    $x \in B_\epsilon(p)$ is a star point of $\omega(p)$.
\end{lemma}
\begin{proof}
    For $j=1,\dots,s(p)$ denote by $H_j\subset\R^d$ the open
    half-space containing all points that lie on the same side of
    $P_{T_j}$ as $p$.  Then $\bigcap_{j=1}^{s(p)}H_j$ is open,
    contains $p$ and is a subset of the set of all star-points of
    $\omega(p)$ by Lemma \ref{lemma:same_side_hyper_plane_star}.
\end{proof}

For interior vertices we obtain the following result.
Recall the definitions of $ \tilde{T}_{ij} $ and $ \tilde{\omega}_j(p) $
given in \eqref{def:ttilde} and \eqref{def:omegatilde}, respectively.
\begin{lemma}
    \label{lemma:tilde_g_for_theorem}
    Given $p \in\calV\cap\interior\domain$, let $ T_1, \ldots T_{s(p)} \in \calT$ be the
    simplices adjacent to $ p $.  For all $j=1,\ldots,s(p)$, let
    $a_0 := p$ and $a_1,\ldots,a_d\in\R^d$ be such that
    $T_j = \simp{a_0}{a_d}$ and let
    $q_j:=p + \delta_j \sum_{i=1}^d (p-a_i)$ for some sufficiently small
    $\delta_j>0$.
    Then $\tilde{\omega}_j(p)$ is convex and
    $\overline{T_j} \subset \tilde{\omega}_j(p) \subset \omega (p)$.
    The sets $\{\tilde{T}_{ij}\}_{i=0,\ldots,d}$ form a regular
    partition of $\tilde{\omega}_j(p)$.
\end{lemma}

\begin{proof}
    For $\epsilon>0$ as in Lemma \ref{lemma:ball_epsilon}, let
    $\delta_j>0$ in the definition of $q_j$ be such that
    $\|q_j-p\|_2 = \frac{\epsilon}{2}$. Then $q_j \in B_\epsilon(p)\subset
    \omega(p)$ is a star point of $\omega(p)$ by Lemma
    \ref{lemma:ball_epsilon}.
    Therefore, $\tilde{T}_{ij}\subset\omega(p)$ for all $i=1,\ldots,d$,
    and thus $\tilde{\omega}_j(p)\subset\omega(p)$.  It also holds that
    $\overline{T_j} = \overline{\tilde{T}}_{0j}
    \subset\tilde{\omega}_j(p)$.  After observing that
    $\interior\tilde{\omega}_j(p) = (T_j)_{\delta_j}$ in the notation of
    Lemma \ref{lemma:move_point_outwards}, Lemma
    \ref{lemma:move_point_outwards} shows that $\tilde{\omega}_j(p)$ is
    a simplex and thus convex. The sets
    $\{\tilde{T}_{ij}\}_{i=0,\ldots,d}$ form a regular partition of
    $\tilde{\omega}_j(p)$ by Proposition
    \ref{prop:new_patch_from_interior_point} ($q_j,a_1,\ldots,a_d,p$ in
    the notation of this proof correspond to $a_0,\ldots,a_d,p$ in the
    notation of the lemma).
\end{proof}

\begin{corollary}
    Let $\tilde{\omega}_j(p)$ be as in
    Lemma \ref{lemma:tilde_g_for_theorem},
    then
    \[ \omega(p) = \bigcup_{j=1}^{s(p)}
    \tilde{\omega}_j(p).  \]
\end{corollary}
\begin{proof}
    Using $\overline{T_j} \subset \tilde{\omega}_j(p) \subset \omega(p)$
    for all $j = 1,\ldots,s(p)$ we have
    \begin{equation*}
        \omega(p) = \bigcup_{j=1}^{s(p)}
        \overline{T_j} \subset \bigcup_{j=1}^{s(p)}
        \tilde{\omega}_j(p) \subset \omega(p).
    \end{equation*}
\end{proof}
\begin{corollary}
    \label{coroll:subset_conex}
    In the notation of Lemma \ref{lemma:tilde_g_for_theorem},
    for $ p \in\calV \cap\interior\domain $ and $j=1,\ldots,s(p)$, let
    $\tilde{\theta}_{p,j}^{\So} \in C^0(\domain)$ be the hat function on $ \tilde{\omega}_j(p) $
    defined in \eqref{eq:tilde_basis_functions_def}.
    For all $ p \in\calV\cap\interior\domain $ and $j=1,\ldots,s(p)$,
    these functions can be written as
    \begin{equation*}
        \tilde{\theta}_{p,j}^{\So}(x)
        = \max\left \{0, \min_{i = 0,\ldots,d} \tilde{A}_p^{(i,j)} x + \tilde{b}_p^{(i,j)}  \right \},
        \quad x\in\domain,
    \end{equation*}
    where each $x\mapsto\tilde{A}_p^{(i,j)} x + \tilde{b}_p^{(i,j)}$ is
    a globally linear function fulfilling
    $(\tilde{A}_p^{(i,j)} x + \tilde{b}_p^{(i,j)})
    |_{\tilde{T}_{ij}} =
    \tilde{\theta}_{p,j}^{\So}|_{\tilde{T}_{ij}}$.

\end{corollary}
\begin{proof}
    Since $\tilde{\omega}_j(p)$ is a convex patch, the statement follows
    by Prop.~\ref{prop:HLXZ2020basis}
    (which corresponds to \cite[Theorem 3.1]{HLXZ2020}).
    The function $x\mapsto\tilde{A}_p^{(i,j)} x + \tilde{b}_p^{(i,j)}$ in
    our notation corresponds to $g_k$ 
    in the notation of \cite{HLXZ2020}, and
    $\tilde{\theta}_{p,j}^{\So}$ corresponds to $\phi_i$.
\end{proof}
\begin{lemma}
    \label{lemma:lin_function_sm_th_or}
    For all $p\in\calV$, $j,k=1,\ldots,s(p)$ and
    $i=0,\ldots,d$, let $x\mapsto\tilde{A}_p^{(i,j)} x +
    \tilde{b}_p^{(i,j)}$ be as defined in Corollary
    \ref{coroll:subset_conex} and let $x\mapsto A_p^{(k)} x +
    b_p^{(k)}$ be the function defined by $(A_p^{(k)} x + b_p^{(k)})
    |_{T_k} = \theta_p^{\So} |_{T_k}$.
    Then,
    \begin{equation*}
        0 \leq
        \tilde{A}_p^{(i,j)} x + \tilde{b}_p^{(i,j)}  \leq A_p^{(k)} x + b_p^{(k)},
        \quad \forall x \in T_{k} \cap \tilde{T}_{ij},
    \end{equation*}
    for all $j,k = 1,\ldots,s(p)$ and $i=0,\ldots,d$.
\end{lemma}
\begin{proof}
    First consider
    $p\in\calV\cap\interior\domain$, such that
    $p\in\interior\omega(p)$.  For all
    $j,k\in\{1,\ldots,s(p)\}$ and
    $i\in\{0,\ldots,d\}$, we first note that $\tilde{A}_p^{(i,j)} p +
    \tilde{b}_p^{(i,j)} = A_p^{(k)} p + b_p^{(k)} =
    1$ as well as $\tilde{A}_p^{(i,j)} y + \tilde{b}_p^{(i,j)} =
    0$ for all $y$ on the face 
    of $\tilde{T}_{ij}$ opposite to $p$.
    Similarly, $A_p^{(k)} y + b_p^{(k)} = 0$ for all
    $y$ on the face of $T_k$ opposite to $p$.
    Next, let $x\in T_{k} \cap
    \tilde{T}_{ij}$ in case this set is nonempty.  Let
    $L$ be the halfline starting in $p$ through
    $x$.  Note that $p\in\interior\tilde{\omega}_j(p) \subset
    \interior\omega(p)$ is a star point of both
    $\tilde{\omega}_j(p)$ and
    $\omega(p)$.  It follows from $\tilde{\omega}_j(p) \subset \omega
    (p)$ that the intersection point $y_1 \in
    L\cap\partial\tilde{\omega}_j(p)$ is closer to
    $p$ than (or equal to) the intersection point $y_2 \in
    L\cap\partial\omega(p)$.  Because $y\mapsto\tilde{A}_p^{(i,j)} y +
    \tilde{b}_p^{(i,j)}$ linearly interpolates between the value
    $1$ in $p$ and $0$ in $y_1$, and $y\mapsto A_p^{(k)} y +
    b_p^{(k)}$ linearly interpolates between $1$ in $p$ and
    $0$ in $y_2$, it follows that $\tilde{A}_p^{(i,j)} y +
    \tilde{b}_p^{(i,j)} \leq A_p^{(k)} y +
    b_p^{(k)}$ for all points $y$ between $p$ and
    $y_1$, which includes
    $x$.  Finally, the first inequality in the lemma follows from
    Corollary \ref{coroll:subset_conex}.

    For
    $p\in\calV\cap\partial\domain$, we can apply the argument above
    after extending
    $\calT$ to a regular, simplicial partition of all of
    $\R^d$, of which only the elements touching $p$ are relevant.
\end{proof}

\begin{proofof}{Theorem \ref{thm:theorem_min_max_basis_functions}}
    For all $j=1,\ldots,s(p)$, applying Lemma
    \ref{lemma:lin_function_sm_th_or} for all $i=0,\ldots,d$ and all
    $k=1,\ldots,s(p)$ shows that
    $0 \leq \tilde{\theta}_{p,j}^{\So}(x) \leq \theta_{p}^{\So}(x)$ for
    all $x\in\tilde{\omega}_j(p)$.
    Together with
    $\tilde{\theta}_{p,j}^{\So}(x) = 0 $ for all
    $x\in\domain\backslash\tilde{\omega}_j(p)$, this shows that
    $0 \leq \tilde{\theta}_{p,j}^{\So}(x) \leq \theta_{p}^{\So}(x)$ for
    all $x\in\domain$.  To finish the proof, recall that for all
    $j=1,\ldots,s(p)$ and $x\in T_j$
    $$
    \theta_p^{\So}(x) = A_p^{(j)} x + b_p^{(j)} = \tilde{A}_p^{(0,j)} x +
    \tilde{b}_p^{(0,j)} = \tilde{\theta}_{p,j}^{\So}(x) .
    $$
    The first and the last equality hold by definition, and the second
    holds because both functions are linear and equal the value $1$ in $p$
    and $0$ in the other vertices of $T_j$.
\end{proofof}

  \begin{proofof}{Theorem \ref{thm:relucpwl}}
    Because
    $\tilde{\omega}_j(p) = \cup_{i=0}^d \overline{\tilde{T}}_{ij}$ is
    a regular partition of the convex set $\tilde{\omega}_j(p)$ by
    Lemma \ref{lemma:tilde_g_for_theorem}, we can apply Proposition
    \ref{prop:HLXZ2020basis} and it follows that for all
    $j=1,\ldots,s(p)$
    \begin{align*}
        \depth(\tilde{\Phi}^{CPwL}_{p,j} ) \leq \, 5 + \log_2(d+1)
        , \quad
        \size(\tilde{\Phi}^{CPwL}_{p,j} ) \leq \, C d(d+1) \leq C d^2
    \end{align*}
    and that all NNs $\{ \tilde{\Phi}^{CPwL}_{p,j} \}_{j=1}^{s(p)}$ have equal depth,
    see Proposition \ref{prop:HLXZ2020basis}.
    The fact that $\realiz{\Phi_p^{CPwL} }(x) = \theta^{\So}_p(x)$ for
    all $x\in\domain$ follows from Theorem
    \ref{thm:theorem_min_max_basis_functions}, and the network depth
    and size are bounded as follows:
    \begin{align*}
        \depth(\Phi_p^{CPwL} ) = &\, \depth( \Phi^{\max}_{s(p)} ) + \depth( \tilde{\Phi}^{CPwL}_{p,1} )
        \leq 2 + \log_2(s(p)) + 5 + \log_2(d+1),
        \\
        \size(\Phi_p^{CPwL} ) \leq &\, 2 \size( \Phi^{\max}_{s(p)} )
        + 2 \sum_{j=1}^{s(p)} \size( \tilde{\Phi}^{CPwL}_{p,j} )
        \leq C s(p) + s(p) C d^2
        \leq C d^2 s(p)
        .
    \end{align*}

\end{proofof}

%%%%%%%%%%%%%%%%%%%%%%%%%%%%%%%%%%%%%%%%%%%%%%%%%%%%%%%%%%%%%%%%%%
\subsection*{Acknowledgement} \label{sec:Ackn}
%%%%%%%%%%%%%%%%%%%%%%%%%%%%%%%%%%%%%%%%%%%%%%%%%%%%%%%%%%%%%%%%%%
ChS acknowledges stimulating discussions at the workshop ``Deep
learning and partial differential equations MDLW03 15 November 2021 to
19 November 2021'' at the Isaac Newton Institute, Cambridge, UK,
during the program ``Mathematics of deep learning MDL 1 July 2021 to
17 December 2021''. Excellent online conferencing and discussion
facilitation by the INI and sabbatical leave from ETH Z\"urich during
the autumn term 2021 are warmly acknowledged.

{\small
\bibliography{bibliography}
}
\end{document}